\documentclass[10pt,a4paper]{article}
\usepackage[margin=3cm]{geometry}

\usepackage{amsmath,amssymb}
\usepackage[normalem]{ulem}
\usepackage{cancel}
\usepackage{enumerate}
\usepackage[font=footnotesize]{caption}

\usepackage{graphicx}
\usepackage{subcaption}
\usepackage{xcolor}
\usepackage{pgfplots}
    \pgfplotsset{compat=1.11}
    \usetikzlibrary{patterns,patterns.meta,arrows.meta}

\usepackage[colorlinks=true,hyperindex=true]{hyperref}
    \definecolor{MainAccent}{HTML}{57068c}
    \definecolor{SecAccent}{HTML}{ab82c5}
    \hypersetup{colorlinks,citecolor=MainAccent,filecolor=MainAccent,linkcolor=MainAccent,urlcolor=MainAccent}

\makeatletter
\newcommand{\labitem}[2]{%
\def\@itemlabel{\textbf{#1}}
\item
\def\@currentlabel{{\upshape #1}}\label{#2}}
\makeatother

\usepackage{amsthm}
\newtheoremstyle{boldnote}
  {\topsep} 
  {\topsep} 
  {\rm}
  {} 
  {\bfseries} 
  {.} 
  {.5em} 
  {\thmname{#1}\thmnumber{ #2}: \thmnote{\bfseries#3}}

\newtheorem{theorem}{Theorem}[section]
    \newtheorem{proposition}[theorem]{Proposition}
    \newtheorem{conjecture}[theorem]{Conjecture}
    \newtheorem{lemma}[theorem]{Lemma}
    \newtheorem{corollary}[theorem]{Corollary}
    
    \newtheorem*{claim*}{Claim}
    
\theoremstyle{definition}
    \newtheorem{remark}[theorem]{Remark}

    \newtheorem{assumption}{Assumption}[subsection]
    
    \newtheorem{assumptionD}{Assumption}
    
    \newtheorem{assumptionC}{Assumption}
    
\theoremstyle{boldnote}
    \newtheorem{scenario}{Scenario}

\numberwithin{equation}{section}

\newcommand{\nn}{\mathbb{N}}
\newcommand{\zz}{\mathbb{Z}}
\newcommand{\rr}{\mathbb{R}}
\newcommand{\one}{\mathbf{1}}
    \newcommand{\N}{\nn}
    
    \newcommand{\R}{\rr}
\newcommand{\eps}{\varepsilon}
\newcommand{\Exp}[1]{\mathrm{e}^{#1}}
\renewcommand{\complement}{\mathsf{c}}

	\makeatletter
	\providecommand*{\diff}%
	{\@ifnextchar^{\DIfF}{\DIfF^{}}}
	\def\DIfF^#1{%
	\mathop{\mathrm{\mathstrut d}}%
	\nolimits^{#1}\gobblespace}
	\def\gobblespace{%
	\futurelet\diffarg\opspace}
	\def\opspace{%
	\let\DiffSpace\!%
	\ifx\diffarg(%
	\let\DiffSpace\relax
	\else
	\ifx\diffarg%
	\let\DiffSpace\relax
	\else
	\ifx\diffarg\{%
	\let\DiffSpace\relax
	\fi\fi\fi\DiffSpace}
    \newcommand{\dd}{\diff}

\newcommand{\PP}{\mathbb{P}}
\newcommand{\pp}{\mathbf{P}}
\newcommand{\ee}{\mathbf{E}}
\newcommand{\cF}{\mathcal{F}}
\newcommand{\empi}[1]{\mu^{#1}}
\newcommand{\invm}{\pi}
\newcommand{\cL}{\mathcal{L}}
\newcommand{\BB}{\mathbb{B}}
\newcommand{\cM}{\mathcal{M}}
\newcommand{\convEvent}{B}

\newcommand{\bX}{\mathbb{X}}
\newcommand{\bXo}{\bX^{\circ}}
    \newcommand{\X}{\bX}

\newcommand{\bU}{\mathbb{U}}

\newcommand{\Ed}{\mathsf{E}}
\newcommand{\Or}{\mathsf{O}}
\newcommand{\Supt}{\mathcal{S}}
\newcommand{\Attr}{\mathcal{A}}
\newcommand{\Faces}{\mathbb{F}}

\newcommand{\oldlambda}{\Lambda_2}
\newcommand{\oldbarlambda}{\bar{\Lambda}_2}

\newcommand{\oldbarlambdastar}{\bar{\Lambda}_{2}^{+}}
\newcommand{\eivalline}{\lambda}
\newcommand{\eivalh}{\lambda_1}
\newcommand{\eivalv}{\lambda_2}
\newcommand{\eivalmoins}{\lambda_-}
\newcommand{\eivalplus}{\lambda_+}

\newcommand{\eivalplustilde}{\tilde{\lambda}_+}
\newcommand{\factorLyap}{\theta}
\newcommand{\rootEd}{\mathcal{R}}
\newcommand{\Index}{\Pi}

\newcommand{\rotat}{\Theta}

\newcommand{\Qset}{Q}
\newcommand{\Sset}{S}
    \newcommand{\Ssetin}{\Sset_{\mathrm{i}}}
    \newcommand{\Ssetout}{\Sset_{\mathrm{o}}}
\newcommand{\Rset}{R}

\newcommand{\rplus}{\delta}
\newcommand{\rmoins}{\epsilon}

\newcommand{\cT}{\tau}
\newcommand{\hit}{\zeta}
\newcommand{\exit}{\xi}
    \newcommand{\exituv}{\exit_{U_{r,2r}}}
    \newcommand{\exituh}{\exit_{U_{2r,r}}}
    \newcommand{\exitg}{\exit_{G_{r,r}}}

\newcommand{\zplus}{z^*}
\newcommand{\zmoins}{{z_*}}
\newcommand{\tauplus}{{\tau^*}}
\newcommand{\taumoins}{{\tau_*}}
\newcommand{\tauwtv}{\tau}
\newcommand{\yplus}{{y^*}}
\newcommand{\ymoins}{{y_*}}

\definecolor{green}{rgb}{0.0, 0.6, 0.0}
\definecolor{red}{rgb}{0.8, 0.1, 0.1}
\tikzset{cpatternone/.style={fill=black!10}}
\tikzset{cpatterntwo/.style={pattern={Lines[angle=45,distance={5pt}]}}}
\tikzset{cpatternthree/.style={pattern={Lines[angle=-45,distance={3pt}]}}}
\pgfmathsetmacro{\err}{0.25}
\pgfmathsetmacro{\errbutsmaller}{0.15}
\pgfmathsetmacro{\errplus}{0.25}
\pgfmathsetmacro{\errmoins}{0.825}
\pgfmathsetmacro{\near}{0.1}
\pgfmathsetmacro{\xmaximum}{1.01}
\pgfmathsetmacro{\ymaximum}{.17}
\tikzset{every picture/.style={baseline={(current bounding box.north)}}}
\newcommand{\swsaddlesmall}{\draw[->] (0,0) -- (0.5*\err,0); \draw[-<] (0,0) -- (0,0.5*\err);}
\newcommand{\swsaddletildesmall}{\draw[->] (1,0) -- (1,0.5*\err); \draw[-<] (1,0) -- (1-0.5*\err,0);}

\begin{document} 

\title{ Random attractors and nonergodic attractors \\ for diffusions with degeneracies}
\author{Yuri Bakhtin\textsuperscript{a}, Renaud Raquépas\textsuperscript{b} and Lai-Sang Young\textsuperscript{a}}
\date{}

\maketitle

\begin{center}
\small
\begin{tabular}{c c c}
   a. New York University & & b. Duke University \\
   Courant Institute of Mathematical Sciences & & Department of Mathematics\\
  New York, NY, United States & & Durham, NC, United States\\
\end{tabular}
\end{center}

\medskip

\begin{abstract}
    We consider a diffusion on a bounded domain, assuming that the system is irreducible
    inside the domain and that the diffusion has varying degree of degeneracy on the domain's boundary. The long-term statistical properties of typical trajectories started inside the domain may be governed by one invariant measure or more than one invariant measure. We describe various possible scenarios. In dimensions~$1$ and~$2$ under boundary hyperbolicity assumptions, we give a complete classification of the limiting behavior and answer the question whether sequential averaging involving more than one invariant distribution occurs. In all cases, we compute the set of weak limit points of empirical measures. Our hitting-time estimates used to prove transience or recurrence are based on a new version of the Foster--Lyapunov technique. Extensions to nonhyperbolic boundaries and higher dimensions are discussed and an application to growth rates in scalable networks is given.

    \medskip

    \noindent \textit{MSC 2020} \hspace{1.4em} 37A25, 37H30, 60J60

    \noindent \textit{Keywords} \hspace{1.75em} diffusion, noise degeneracy, boundary, nonergodic behavior, attractor, heteroclinic cycling, Lyapunov function 
\end{abstract}

\tableofcontents

\section{Introduction}
\label{sec:intro}

In this paper, we study stochastic dynamical systems defined by SDEs with a noise term that is degenerate on parts of the phase space. 
In the basic setup that our results apply to, we may assume that a bounded 
finite-dimensional domain $\bX$ with a piecewise smooth boundary is invariant under the dynamics, and that the diffusion  is nondegenerate (elliptic) in the interior of $\bX$  and degenerate on its boundary, necessarily so as trajectories cannot exit the domain. 
Although we prove our results under a weaker irreducibility assumption in place of ellipticity, in this introduction we restrict ourselves to the basic setup described above. 

Systems of this type arise naturally, e.g., in scalable reaction networks, such as metabolic networks involving the stochastic interconversion of a finite number of substances, or ecological networks describing the interaction of a finite number of species. The interaction is called scalable if changes in the relative composition of the constituent substances are independent of the total size of the system. As the quantities of substances or numbers of animals cannot be negative, such networks are a priori defined on the positive orthant $\{(x_1, \dotsc, x_N) \in \rr^N: x_i \ge 0\}$; scalability implies that the dynamics factor onto a system on the $(N-1)$-dimensional simplex $\Delta_{N-1} =\{(x_1, \dotsc, x_N): \sum_{i=1}^N x_i =1\}$; see \cite{LKYJW20}.

We are interested in the large-time dynamics of trajectories starting from initial conditions in $\bXo$, the interior of $\bX$.  For compact manifolds without boundary,
it is a standard fact that a stochastic system defined 
by a nondegenerate diffusion gives rise to a unique, hence ergodic, invariant probability measure having a Lebesgue density, and the large-time statistical properties of solutions are described by this invariant measure and do not depend on the starting point. The presence of a boundary on which the diffusion degenerates leads to a richer and more complex picture. Near the boundary, the trajectories may feel repulsion or attraction. Transient behavior with no invariant measure on $\bXo$ may emerge if trajectories are asymptotically attracted to the boundary.   Note, however, that whether or not the system carries an invariant measure on $\bXo$, the long-term scenarios starting from all initial conditions in $\bXo$ are qualitatively similar, as all starting points are 
accessible from each other.
In this sense, the system has only ``one future”, and it is this future that we seek to describe.

In the reaction network example above, attraction to the boundary translates into a tendency for one or more of the substances to deplete, or one or more of the species to go extinct. As one of the components of the system nears depletion, the dynamics are altered; see, e.g., \cite{LKYJW20,Nandori-LSY:MR4544421}. 

Leaning on the example of scalable reaction networks for motivation, we will consider a setting where the invariant domain $\bX$ is a polytope, such as a simplex or a cube. In addition to our assumption on the invariance of $\bXo$ and ellipticity of the diffusion on $\bXo$,  we impose invariance and relative ellipticity assumptions on the restriction of the system onto  each lower-dimensional polytope of the boundary of $\bX$.
Thus the diffusion has different degrees of degeneracy on boundary polytopes of different dimensions. In this setting, each ergodic probability measure is concentrated either on $\bXo$ or on one of its lower-dimensional faces, and there are finitely many such measures.  The situation where $\bXo$ is recurrent and carries an invariant measure despite the presence of the boundary is called stochastic persistence (see~\cite{Bepreprint} and references therein for results). We are also interested in the role played by ergodic invariant measures on the boundary of $\bXo$. 

Not all of these measures are equally interesting to us, however.
For almost all diffusion paths, every weak limit point of the empirical distribution is an invariant measure (see \cite{Bepreprint}) and thus a mixture of ergodic ones. To contribute to the long-term statistical properties of the trajectory, an ergodic measure must contribute to this mixture with a positive weight. We will view only those ergodic measures as “visible”, or “observable”.  

The relatively simple model setting above allows one to explore what might constitute an observable invariant measure, to characterize them and to study their accompanying dynamics. In this paper, under fairly general assumptions, we give a complete description of the dynamical picture in dimensions 1 and 2. Our results also provide a partial picture in higher dimensions, but a complete classification will require more work.

The main challenge is that we do not have the tools of ergodic theory at our disposal for dealing with dynamics of trajectories starting from $\bXo$ if it does not carry an invariant measure. Even though all initial conditions in $\bXo$ have essentially the same ``future'', that future is not necessarily governed by the usual laws of stationarity. 

For example, when there are multiple attractors on the boundary, every initial condition in $\bXo$ is attracted to each attractor with positive probability, contrary to the ergodic decomposition picture for Markov chains (but there is no contradiction because points in $\bXo$ are not typical with respect to any invariant measure). In this situation, where the trajectory has to choose one of the attractors randomly, there still is an element of ergodic theory that survives. 
Namely, statistical properties of a trajectory are governed by the ergodic probability measure concentrated on the attractor it is attracted~to.  

In general, convergence of time averages cannot be taken for granted. We will describe conditions giving rise to scenarios where none of the ergodic measures governs the averaging in the long run. Instead, in these situations, there are several relevant ergodic measures concentrated on their respective faces that contribute to the averaging. 
The system spends increasingly long epochs near these faces and makes relatively fast transitions between them. During each of these epochs the statistics of the trajectory are governed by the relevant ergodic measure on the face, so that measure is ``observable'' for this realization, but on longer time intervals, the system switches between these long epochs and none of them dominates, so the classical convergence of averages does not hold, and we can speak of a nonergodic attractor.     

Our results suggest that this scenario is quite typical of stochastic dynamics on domains with boundaries. One natural feature of the system leading to this phenomenon  is an attracting heteroclinic cycle on the boundary. However, in the presence of noise, this phenomenon is much broader than pure cycling along a chain of heteroclinic orbits connecting several saddle points.

In dimensions 1 and 2, under fairly general regularity assumptions on the diffusion (weaker than the usual ellipticity assumption) and generic hyperbolicity assumptions near the boundary, we provide a complete classification of possible behaviors. They fall into one of the following categories:  (A) scenarios with a unique ergodic attractor, (B)  scenarios where there are several ergodic attractors
and the diffusion randomly chooses one of them, (C)  scenarios with a nonergodic attractor supporting a family of nonergodic measures sampled continuously by the diffusion.

\medskip

The picture we are describing is relevant not just for dynamics in a polytope with boundaries. It is applicable to more general diffusions allowing for invariant manifolds (possibly, with corners) of varying dimensions organized into cell complexes. 
In Section~\ref{sec:discussion}, we discuss the broader context of our results and connections to the existing literature. 

Some of our results on recurrence and transience rely on an extension of the Foster--Lyapunov technique. Namely, developing an idea from~\cite{YM12}, we prove a new estimate on the average hitting time for a set. Compared to the usual setup where the Lyapunov function applied to the process is required to have a negative drift on the complement of the set, we consider 
a less restrictive one allowing this decay assumption to fail in some regions where the process spends little time; see Section~\ref{sec:YM}.

\paragraph*{Acknowledgments} The research of RR was partially funded by the Natural Sciences and Engineering Research Council of Canada and the Fonds de Recherche du Qu\'ebec -- Natures et technologies while at the Courant Institute.  YB is grateful to NSF for partial support through award DMS-2243505.
LSY is grateful to NSF for partial support through awards DMS-1901009 and DMS-2350184

\section{Main results}
\label{sec:main-results}
\subsection{Dimension 1}
\label{sec:1d}

As the boundary of a polytope is the union of lower-dimensional polytopes, we seek to build a theory of diffusions on polytopes starting with the lowest dimension. The relative simplicity of dimension~1 allows us to state succinctly our hypotheses and results all of which have higher dimensional analogs.

Consider a process~$X = (X(t))_{t\geq 0}$ driven by  a $1$-dimensional Stratonovich SDE
\begin{equation}
  \label{eq:1d-SDE} \dd X(t) = b (X(t)) \dd t +  \sigma (X(t)) \circ \dd W(t)
\end{equation}
on $\bX=[0,1]\subset\rr^1$, 
where $W = (W(t))_{t\geq 0}$ is  standard Brownian motion defined on a complete probability space $(\Omega,\cF,\pp)$. 
Throughout the paper, we use ``$\circ\dd W$'' for Stratonovich integrals and simply ``$\dd W$'' for It\^o integrals. 
We usually denote elements of $\Omega$ by $\omega$. We denote by $(\cF_t)_{t\ge 0}$ the augmented natural filtration of $W$ (for each $t\ge 0$, $\cF_t$ is generated by $(W(s))_{0\leq s\le t}$, and $\pp$-null sets).
We will work under the following assumptions:

\begin{assumption}[{\bf Smoothness}]
\label{req:smoothness-1d}
    The drift $b$ and diffusion $\sigma$ belong to $C^2(U)$, where $U\subset \rr$ is an open set containing $\bX$. 
\end{assumption}

\begin{assumption}[{\bf Invariance of boundary}]
\label{req:invars-1d}
    Both drift and diffusion leave invariant $\partial  \bX$, or equivalently $b(0)=\sigma(0)=b(1)=\sigma(1)=0$. 
\end{assumption}

\begin{assumption}[{\bf Irreducibility}]
\label{req:ellip-1d}
    For every $x\in \bXo$ and every nonempty open interval $(a,b) \subset \bXo$, there is $t>0$ such that $\pp\{X^x(t) \in (a,b)\} > 0$.
\end{assumption}

\begin{remark}
    Together, the above assumptions imply that, for every $x\in \bX^\circ = (0,1)$, we have $\sigma(x)\ne 0$, i.e., that the diffusion is elliptic. To see this, suppose for the sake of contradiction that there is $x_*\in(0,1)$ such that $\sigma(x_*) = 0$, and then without loss of generality that $b(x_*) \ge 0$. 
    By Assumption~\ref{req:smoothness-1d}, for all $t>0$ and all continuous controls $u:[0,t]\to\rr$, the (unique, continuously differentiable) solution $x(\,\cdot\,)$ of the controlled ODE
    $\dot{x}(s) = b(x(s)) + u(s) \sigma(x(s))$ with initial condition $x(0)=x_*$
    satisfies $x(t)\ge x_*$. Then, the Stroock--Varadhan support theorem~\cite[Section~{8.3}]{Stroock:MR1980149} implies that $\pp\{X^{x_*}(t) \in (0,x_*)\} = 0$ for all $t \geq 0$, contradicting Assumption~\ref{req:ellip-1d}.
\end{remark}

Assumptions~\ref{req:smoothness-1d} and~\ref{req:invars-1d} ensure that, given any initial $x\in\bX$, there is a (unique strong) solution 
$(X^x(t))_{t \geq 0}$ to the SDE~\eqref{eq:1d-SDE} in $\bX$. If $x\in\bXo$, then $X^x(t) \in \bXo$ for all $t \geq 0$, with probability~1. 
Adding Assumption~\ref{req:ellip-1d} to Assumptions~\ref{req:smoothness-1d} and~\ref{req:invars-1d}
guarantees, in particular, that there is at most one invariant probability measure concentrated on~$\bX^\circ$ (we say that a probability measure~$\mu$ is \emph{concentrated} on a set~$C$ if $\mu(C)=1$).

\begin{assumption}[{\bf Boundary hyperbolicity}]
\label{req:hyperbolicity-1d} 
    The coefficient $\eivalline^0$ in the linearization $b(x) = \eivalline^0x + o(x)$ about $x = 0$ is nonzero, and so is its counterpart~$\eivalline^1$ about $x = 1$.
\end{assumption} 

Assumption~\ref{req:hyperbolicity-1d} ensures that $x=0$ is either a {\it source} (if $\eivalline^0>0$) or
a {\it sink} (if $\eivalline^0<0$), and the same is true for $x=1$. 
This assumption (and its counterparts in higher dimensions) is easier to state and interpret for Stratonovich SDEs, which is the main reason why we define the diffusion via~\eqref{eq:1d-SDE}.

Our main result describes the asymptotic behavior of the process $X^x(t)$.
Denoting the (possibly empty) set of sinks by~$\Attr$, for $x\in\bXo$ and $k\in\Attr$,  
we define the convergence event $\convEvent^x_k=\{\lim_{t\to \infty} X^x(t)= k\}$ and its probability $p^x_k=\pp(\convEvent^x_k)$. 
The  empirical distribution ({\it aka} occupation measure) is a Borel probability measure on $\X$  defined by
\begin{equation}
  \label{eq:empirical}
 \empi{x}_t =\empi{x}_{t,\omega} = \frac{1}{t} \int_0^t \delta_{X^x_\omega(s)}
    \dd s,\quad x\in\X,\ t > 0,\ \omega\in\Omega.
\end{equation}
Convergence of measures (and in particular of empirical distributions) will always be considered with respect to the topology of weak convergence, i.e., tested against continuous functions on~$\bX$.

\begin{theorem}    
\label{thm:1d}
\label{prop:1d-sink}
    Suppose that Assumptions \ref{req:smoothness-1d}--\ref{req:hyperbolicity-1d}  hold. 
    \begin{enumerate}[I.]
    \item\label{it:there-are-sinks}
    If~$\Attr \neq \emptyset$, then we have $p^x_k>0$ for each $k \in \Attr$ and every $x \in \bXo$,  and 
    \begin{align}
        \label{eq:attr-1d-prob-1}
        \sum_{k \in \Attr} p_k^x = 1, \quad x \in \bXo.
    \end{align}
    \item\label{it:1d-averaging} 
    If $\Attr =\emptyset$, then there is a unique ergodic invariant probability measure~$\invm$ concentrated on $\bXo$, and this measure has a positive Lebesgue density on~$\bXo$. For every $x \in \bXo$, 
    {almost surely, we have $\empi{x}_t \to \invm$.}
\end{enumerate}
\end{theorem}

A proof of this simple result is given in Section~\ref{sec:proofs-1d}. Its elements will be used in broader settings. 
We will also discuss this result in broader context in Section~\ref{sec:summary-result}.

\subsection{Dimension 2}
\label{sec:2d-setup-result}

For definiteness, we work with a diffusion process
on the square $\bX=[0,1]^2$. Our results and proofs extend {\it mutatis mutandis}
to arbitrary convex polygons and their diffeomorphic sets.  
We assume that $X=(X_1,X_2) \in \bX$ is driven by a  Stratonovich SDE:
\begin{equation}
\label{eq:main-eq-vec}
        \dd X(t)=b(X(t))\dd t+\sigma_1(X(t))\circ \dd W_1(t) + \sigma_2(X(t))\circ \dd W_2(t),
\end{equation}    
where $W_1 = (W_1(t))_{t \geq 0}$ and $W_2 = (W_2(t))_{t \geq 0}$ are independent standard Brownian motions defined on a complete probability space~$(\Omega,\cF,\pp)$, and where $\sigma_1$, $\sigma_2$, and $b$ are vector fields.
We denote  the augmented natural filtration of~$W = (W_1, W_2)$ by~$(\cF_t)_{t\ge 0}$. 
Equivalently, this SDE can be written as a system
    \begin{subequations}
    \label{eq:maineq}
      \begin{align}
      \dd X_1(t) & =  b_1 (X(t)) \dd t + \sigma_{11} (X) \circ \dd W_1 (t) + 
      \sigma_{21} (X) \circ \dd W_2 (t),  \label{eq:main-eq-h} \\
      \dd X_2(t) & =  b_2 (X(t)) \dd t + \sigma_{12} (X) \circ \dd W_1 (t) + 
      \sigma_{22} (X) \circ \dd W_2 (t), \label{eq:main-eq-v}
      \end{align}
    \end{subequations}   
where 
$b(x)=(b_1(x),b_2(x))$, 
$\sigma_1(x)=(\sigma_{11}(x),\sigma_{12}(x))$, and  $\sigma_2=(\sigma_{21}(x),\sigma_{22}(x))$. Throughout the paper, for any $n\in\N$, 
any points $x,x'\in\R^n$, we denote $d(x,x')=\sum_{i=1}^n|x_i-x'_i|$, and for 
$x\in\R^n$ and  $S\subset \R^n$,
$d(x,S)=\inf\{d(x,x'):\ x'\in S\}$.

The assumptions below are parallel to those in dimension~$1$. 

\begin{assumption}[{\bf Smoothness}]
\label{req:smoothness-2d}
\label{it:S0}
    We assume that the vector fields $b,\sigma_1,\sigma_2:U\to\rr^2$ are smooth (i.e., they
     belong to~$C^\infty$) on some open set $U$ satisfying $\bX \subset U \subset \rr^2$.
\end{assumption}

For most of our results, we will, in fact, need these vector fields to belong only to $C^2$,
similarly to the $1$-dimensional case; see Remark~\ref{rem:hypoelliptic-discussion}.  

\begin{assumption}[{\bf Invariance of boundary components}]
\label{req:invar-2d}
\label{it:I0} We assume that on each edge, the vector fields $b$, $\sigma_1$, $\sigma_2$ are tangent to that edge. As a consequence, these vector fields vanish at all vertices of the square.
\end{assumption}

Assumptions~\ref{req:smoothness-2d} and~\ref{req:invar-2d}  ensure that for every initial condition $x\in\X$, there is a unique $\X$-valued strong solution $X^x(t)=X^x_\omega(t)$ of 
system~\eqref{eq:maineq} and $\bXo$ is invariant: if $x\in\bXo$, then, with probability~$1$, $X^x(t)\in\bXo$ for all $t\ge0$.  
All vertices
$\Or^0 = \Or = (0, 0)$, $\Or^1 = (1, 0)$, $\Or^2  = (1, 1)$, $\Or^3 = (0, 1)$; and (open) edges $\Ed^0 = \Ed = (0,1)\times\{0\}$, $\Ed^1 = \{1\} \times (0,1)$, $\Ed^2 = (0,1)\times\{1\}$, $\Ed^3 = \{0\} \times (0,1)$ of the square 
are also invariant under the dynamics; see Figure~\ref{fig:square}.

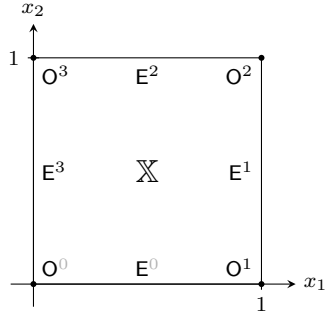
\begin{figure}
    \centering
    \begin{tikzpicture}[scale=3]
        \footnotesize
        \node at (.5,0.5) {\large $\bX$}; 

        \fill (0,0) circle (.375pt);
        \node[anchor=south west] at (0,0) {$\Or^{\color{black!30}0}$};
        \draw (0,0) -- (1,0);
        \node[anchor=south] at (.5,0) {$\Ed^{\color{black!30}0}$};
    
        \fill (1,0) circle (.375pt);
        \node[anchor=south east] at (1,0) {$\Or^1$};
        \draw (1,0) -- (1,1);
        \node[anchor=east] at (1,.5) {$\Ed^1$};
    
        \fill (1,1) circle (.375pt);
        \node[anchor=north east] at (1,1) {$\Or^2$};
        \draw (1,1) -- (0,1);
        \node[anchor=north] at (.5,1) {$\Ed^2$};
    
        \fill (0,1) circle (.375pt);
        \node[anchor=north west] at (0,1) {$\Or^3$};
        \draw (0,1) -- (0,0);
        \node[anchor=west] at (0,0.5) {$\Ed^3$};

        \draw[-stealth] (-.1,0) -- (1.15,0);
        \node[anchor=west] at (1.15,0) {$x_1$};
        \draw (1,-.025) -- (1,.025);
        \node[anchor=north] at (1,-.025) {$1$};
    
        \draw[-stealth] (0,-.1) -- (0,1.15);
        \node[anchor=south] at (0,1.15) {$x_2$};
        \draw (-.025,1) -- (.025,1);
        \node[anchor=east] at (-.025,1) {$1$};
    \end{tikzpicture}
    \caption{Sketch of the 2-dimensional setup in the coordinates $(x_1,x_2)$.
    } 
    \label{fig:square}
\end{figure}

The semigroup $(P^t)_{t\ge 0}$ associated with the SDE~\eqref{eq:main-eq-vec} is defined for bounded measurable functions~$f: \bX \to \rr$ according to the formula
\[
    (P^tf)(x)= \ee[f(X^x(t))].
\]
Under Assumptions~\ref{req:smoothness-2d} and~\ref{req:invars-1d}, this semigroup has the Feller property, i.e., it maps continuous functions to continuous functions for all $t\ge 0$. The restrictions of the process to the interior of the square, edges, and vertices are also Feller.  In particular, for every $t > 0$ and every open set $V$, the transition probabilities~$\pp\{X^x(t) \in V\}$ are jointly lower semicontinuous in~$x \in \bX$ and $t \in [0,\infty)$.

We now turn to a $2$-dimensional counterpart of the irreducibility Assumption~\ref{req:ellip-1d} that was used in dimension~1.

\begin{assumptionC}[{\bf Irreducibility}]
    \label{req:irreducibility}
    In addition to assuming that the 1-dimensional irreducibility assumption~\ref{req:ellip-1d} holds along each edge, we assume that, for every $x\in \bXo$ and every nonempty open subset $V$ of~$\bXo$, there exists $t>0$ such that $\pp\{X^x(t) \in V\} > 0$.
\end{assumptionC}  

Unlike in dimension~1, this assumption does not imply any form of ellipticity. In order to retain some important properties that were afforded to us by ellipticity in dimension 1, we introduce a form of H\"ormander's condition at one point. 

There are several useful versions of H\"ormander's condition of various generalities; see, e.g.,~\cite[Section~{6.5.2}]{Benaim-Hurth:MR4559704}. We will work with one that is sometimes called simply the {\it H\"ormander condition}, but also sometimes the {\it parabolic H\"ormander condition}. 
To introduce it, we first recall that the Lie bracket of two smooth vector fields $u,v$ at a point $x$
is a vector defined by 
\[
[u,v](x)=u (x) D v (x)- v(x) D u (x),
\] 
or, in coordinates, 
\[
[u,v]_i(x)=\sum_{j=1}^2 \big(u_j(x)\partial_j v_i(x)-v_j(x)\partial_j u_i(x)\big),\quad i=1,2.  
\]
Let $\mathfrak{L}$ be the vector space spanned by 
\[
    \sigma_1,\sigma_2,\ [\sigma_i,\sigma_j],\, 0\le i,j \le 2,\ [[\sigma_i,\sigma_j],\sigma_k],\, 0\le i,j,k \le 2,\ \ldots,
\] 
where we temporarily set $\sigma_0=b$ for notational convenience.

\begin{assumptionC}[{\bf H\"ormander condition}] 
    \label{req:Hor} We assume that 
    there is $x_* \in \bXo$  such that $\mathfrak{L}(x_*) = \rr^2$.
\end{assumptionC}

\begin{remark} 
\label{rem:hypoelliptic-discussion}    
    An important consequence of 
    Assumptions~\ref{req:smoothness-2d},~\ref{req:invars-1d},~\ref{req:irreducibility}, and~\ref{req:Hor} is that the interior of the square and all edges carry at most one invariant probability measure for the diffusion each; see~\cite[Theorem~6.34]{Benaim-Hurth:MR4559704}. 
    In fact, Assumption~\ref{req:Hor} will be needed only for this uniqueness statement in case~\ref{it:recurrent-case} of Theorem~\ref{thm:2d-trich} but we state it here to stress the connection with the ellipticity requirement in dimension~$1$. 
    The uniqueness of invariant measure is also the only part where we need $C^\infty$-smoothness of the vector fields to invoke the existing results based on hypoellipticity and Malliavin calculus.
    Had we required the stronger ellipticity property instead of the H\"ormander's hypoellipticity condition, we would have needed only a $C^2$ requirement in Assumption~\ref{req:smoothness-2d}.
    Also, for this uniqueness property in the interior,
    a weaker requirement than~\ref{req:irreducibility}  (namely, accessibility of just one point $x_*$ as in Assumption~\ref{req:Hor}) would suffice, but we also rely on~\ref{req:irreducibility} to control diffusion exit times for large compact sets in the proofs of two of our main results, Theorems~\ref{thm:2d-attr} and~\ref{thm:2d-cycle}; see Lemma~\ref{lem:exit-compact}.  
\end{remark}

\begin{remark} 
    Assumptions~\ref{req:irreducibility} and~\ref{req:Hor} cover a broad class of processes such as elliptic or hypoelliptic diffusions on $\bXo$ but they can be broadened further with minimal modifications of the proofs. Namely, one can take any compact set $K \subset \bXo$ and require Assumption~\ref{req:irreducibility} to hold only for open sets $V$ satisfying $V\subset \bXo\setminus K$ provided that Assumption~\ref{req:Hor} holds for a point $x_*\in\bXo\setminus K$. In their present form, Assumptions~\ref{req:irreducibility} and~\ref{req:Hor} are effectively stated for $K=\emptyset$.
\end{remark}  

As in dimension 1, we impose hyperbolicity conditions on the behavior of the drift $b$ near the boundary. Three such notions of hyperbolicity are presented below as Assumptions~\ref{req:hyperbolicity-2d}--\ref{req:index-stab-ne-1}. The Stratonovich form of the system helps to simplify these assumptions.

\begin{assumptionD}[{\bf Hyperbolicity of vertices}]
\label{req:hyperbolicity-2d}
    At all vertices, all the eigenvalues of linearizations of $b$ are nonzero.
\end{assumptionD}

In our setting, eigenvalues at all vertices must be real, thus,
Assumption~\ref{req:hyperbolicity-2d} means that each vertex is a {\it source} (when both eigenvalues are positive), a {\it sink} (when both eigenvalues are negative), or a {\it saddle} (when the eigenvalues have opposite signs).  We will also sometimes refer to such sinks as \emph{attracting vertices}.

The next hyperbolicity assumption concerns edges. If an invariant probability measure $\invm_\Ed$ is concentrated on the edge $\Ed=(0,1)\times\{0\}$ as in part~\ref{it:1d-averaging} of Theorem~\ref{thm:1d} applied to the restriction of~\eqref{eq:main-eq-vec} to~$\overline{\Ed}$, then we define its {\it transversal Lyapunov exponent} by 
\begin{equation}
\label{eq:average-repulsion-over-edge}
    \oldbarlambda = \int_{\Ed}\Lambda_2(x)\,\invm_{\Ed}(\dd x),
\end{equation}  
where, denoting $\partial_i=\frac{\partial}{\partial x_i}$ and $\partial_{ij}=\frac{\partial^2}{\partial x_i \partial x_j}$, 
\begin{equation}
\label{eq:new-def-oldlambda}
    \Lambda_2(x)=\partial_2 b_2(x)+\frac{1}{2}\big(\sigma_{11}(x)\partial_{12}\sigma_{12}(x)+\sigma_{21}(x)\partial_{12}\sigma_{22}(x)\big). 
\end{equation}  
It can be shown that $\oldbarlambda$ is a genuine Lyapunov exponent, i.e., the exponential rate of attraction or repulsion in the $x_2$-direction near~$\Ed$; see Remark~\ref{rem:Lambda_as_attr_rate}.  This definition is similar to the one given in 
~\cite{FK22,FK23} but it explicitly separates the contribution from the drift and the noise. 
In contrast with the one-dimensional situation, the quadratic covariation terms contributed by the noise are required in this definition.  
In the case of diagonal noise where $\sigma_{12}(x)=\sigma_{21}(x)=0$, the definition of $\oldbarlambda$ writes, more intuitively, as 
\begin{equation}
\label{eq:average-repulsion-over-edge-diag}
    \oldbarlambda = \int_{\Ed}\partial_2 b_2(x)\,\invm_{\Ed}(\dd x).
\end{equation} 

Denoting by~$\Supt$ the set of edges on which an invariant probability measure is concentrated, we can define 
transversal Lyapunov exponents analogously for all edges in $\Supt$. They are not defined for edges on which no invariant probability measure is concentrated. 
\begin{assumptionD}[{\bf Hyperbolicity of edges}]
\label{req:Lyap-of-edge}
  \rm For every edge in $\Supt$, its transversal Lyapunov exponent is nonzero.
\end{assumptionD}

Assumption~\ref{req:Lyap-of-edge} means that each edge on which an invariant probability measure is concentrated can be
either {\it attracting} (when $\oldbarlambda<0$) or {\it repelling} (when $\oldbarlambda>0$). 
Let $\Attr$ be the set of attracting vertices and edges. 
Each~$A\in\Attr$ comes with a unique ergodic probability measure~$\invm_A$. It is either a Dirac mass or a measure with the properties discussed in part~\ref{it:1d-averaging} of Theorem~\ref{thm:1d}. 

Let us now introduce some notation in order to state a 2-dimensional analog of Proposition~\ref{prop:1d-sink}. 
We will write $X^x(t)\to S$ 
as $t\to\infty$, or simply 
$X^x(t)\to S$,
as a shorthand for $\bigcap_{t > 0} \overline{\{X^x(s) : s\geq t\}} = S$, i.e, to state that
$S$ is the set of limit points of $X^x(t)$ as $t\to\infty$.
Note that $X^x(t) \to S$ implies $d(X^x(t),S) \to 0$, but not the other way around.
For  $x\in\bXo$ and $A\in\Attr$, we define
\begin{equation}
\label{eq:def-conv-prob}
    \convEvent^x_A
    =
    \left\{ 
        X^x(t)\to \overline{A}
    \right\}
    \qquad\text{and}\qquad   
    p^x_A=\pp(\convEvent^x_A).
\end{equation}
Similarly to the $1$-dimensional case, the empirical distribution $ \empi{x}_t$ is defined by~\eqref{eq:empirical}. 

\begin{theorem}[Attracting vertices and edges]
\label{th:2d_a} 
\label{thm:2d-attr}
    Suppose that Assumptions \ref{req:smoothness-2d}--\ref{req:irreducibility}
    and \ref{req:hyperbolicity-2d}--\ref{req:Lyap-of-edge} hold. 
    If there is at least one attracting vertex or edge, i.e., $\Attr \neq \emptyset$, then
    $p^x_A>0$ for each $A \in \Attr$ and every $x\in\bXo$, and        
    \begin{align}
        \label{eq:attr-with-prob-1}
        \sum_{A \in \Attr} p^x_A = 1,\quad x\in\bXo.
    \end{align}
    Moreover, for each $A \in \Attr$ and every $x\in\bXo$, on $\convEvent^x_A$, we almost surely have $\empi{x}_t \to \invm_A$ as $t\to \infty$. 
\end{theorem}  

We prove this theorem in Sections~\ref{sec:local-attr} and~\ref{sec:proofs-Lyap}. 

In dimension~$2$, it is possible for the boundary as a set to be attracting or repelling without
any individual vertex or edge being attracting or repelling. We begin with some definitions.

For a saddle at $\Or^k$, there are two possible \emph{orientations}. We define the orientation
of the saddle to be the sign of $\det(v_+^k,v_-^k)$, where eigenvectors~$v_+^k$ and $v_-^k$ for the eigenvalues $\eivalplus^k>0$ and $\eivalmoins^k<0$ of the linearization of $b$ at a vertex~$\Or^k$ are chosen to point along edges towards adjacent vertices\,---\,this is possible due to Assumptions~\ref{req:smoothness-2d} and~\ref{req:invar-2d}.  

We say that the vertices form a 
{\it stochastic cycle} if they are all saddles and 
are consistently oriented, i.e., they have the same orientation.
Stochastic cycles are reminiscent of heteroclinic cycles in deterministic dynamics, except
that  no assumptions are made about the dynamics on the edges connecting the saddles.
The \emph{stability index} $\Index$ of a stochastic cycle is defined to be
\begin{equation}
    \label{eq:product-of-rhos} 
\Index =  \rho^0 \rho^1 \rho^2 \rho^3 
\end{equation}
where, for each $k=0,1,2,3$,
\begin{equation}
    \label{eq:rhos}
\rho^k=\frac{|\eivalmoins^k|}{\eivalplus^k}.
\end{equation}
We say the cycle is {\it stable} if $\Index > 1$, and {\it unstable} if $\Index < 1$.

\begin{assumptionD}[{\bf Hyperbolicity of stochastic cycles}]
\label{req:index-stab-ne-1} If the vertices form a stochastic cycle, we 
 assume that $\Index\ne 1$.
\end{assumptionD}

The next theorem summarizes the picture when there is a  stable stochastic 
cycle. Roughly, almost all sample paths
spend asymptotically $100\%$ of their time near the 4 vertices, with
the duration of visit to each corner increasing exponentially.

\begin{theorem}[Stable stochastic cycle]
\label{th:2d_b} 
\label{thm:2d-cycle}
    In addition to Assumptions~\ref{req:smoothness-2d}--\ref{req:irreducibility} and
    \ref{req:hyperbolicity-2d}, we assume the vertices of $\bX$
    form a stable stochastic cycle, which without loss of generality we take to be
    positively oriented. 
    Then, given any $x \in \bX^\circ$, the following holds with probability~1: 
    \begin{enumerate}
        \item\label{it:conv-to-bdry} 
        We have $X^x(t) \to \partial\bX$ as $t\to\infty$.

        \item\label{it:mostly-corners} 
        For any choice of neighborhoods $U^k$ of $\Or^k$, $k \in \{0,1,2,3\}$, we have 
        $
            \empi{x}_t(\bigcup_{k=0}^3 U^k) \to 1
        $ as $t\to\infty$.
        \item\label{it:recent-corner-is-essential}
        For any choice of a neighborhood $U^k$ of $\Or^k$, $k \in \{0,1,2,3\}$, there exists a constant $c \in (0,1)$ and a sequence $(T^k_n)_{n\in\nn}$ of times converging to $+\infty$ with the property that 
        $
        X^x(t) \in U^k
        $
        for all $t \in [cT^k_n, T^k_n]$ and $n\in\nn$.
    \end{enumerate}
\end{theorem}  

Parts~\ref{it:conv-to-bdry} and~\ref{it:recent-corner-is-essential} of Theorem~\ref{thm:2d-cycle}, imply that, almost surely, visits to the different neighborhoods~$U^k$ occur cyclically. In fact, this cycling is key to our analysis in Section~\ref{sec:main-heteroc-proof}.

While Theorem~\ref{thm:2d-cycle} describes the set of limit points of the process $(X_t)_{t\geq 0}$ in the space $\bX$ when the vertices form a stable stochastic cycle, we are also interested in the limit points of the process~$(\empi{x}_t)_{t > 0}$ in the space $\PP(\X)$ of Borel probability measures on~$\bX$. The following general result of Bena\"im~\cite[Section~{2.1}]{Bepreprint} already narrows down the search to convex combinations of only a few ergodic measures.

\begin{theorem}
\label{thm:benaim}
    Suppose that $P=(P^t)_{t\ge 0}$ is the semigroup associated with a homogeneous Feller Markov process $(X_t)_{t\geq 0}$ with continuous paths on a compact metric space $\X$. 
    Then, for each $x\in\X$, with probability 1, all weak limit points of empirical distributions~$(\empi{x}_t)_{t > 0}$ as $t\to\infty$ are invariant under~$P$.  
\end{theorem}   

In the case of a stable stochastic cycle, Theorem~\ref{thm:1d} precludes the existence of invariant measures giving positive measure to any edge, and Theorem~\ref{thm:2d-cycle} precludes the existence of invariant measures giving positive measure to interior of the square, leaving us with the 3-dimensional simplex $\Delta$ of mixtures of $\delta_{\Or^0},\delta_{\Or^1},\delta_{\Or^2},\delta_{\Or^3}$. The next proposition (and the comments that follow) completes this picture.

\begin{proposition}
\label{prop:quadri}
    {Under the conditions of Theorem~\ref{thm:2d-cycle},} given any $x \in \bXo$, with probability~1, the set of weak limit points of~$(\empi{x}_t)_{t > 0}$ is a closed curve~$\Gamma$ that consists of 4 distinct, nondegenerate line segments in the relative interior of the simplex~$\Delta$, and is independent of $x$. In particular, given any $x \in \bXo$, with probability~1, the measures~$(\empi{x}_t)_{t > 0}$ fail to converge weakly.
\end{proposition}

We identify the precise line segments forming the closed curve~$\Gamma$ in Proposition~\ref{prop:quadri} as part of its proof. The proof also shows that, given any~$x$, almost surely, up to any given tolerance, the path~$(\empi{x}_t)_{t > 0}$ eventually cycles along this closed curve, at a speed that converges to 0. 

\medskip

The picture described in Theorem~\ref{thm:2d-cycle} and 
Proposition~\ref{prop:quadri}, both proved in Section~\ref{sec:main-heteroc-proof},
is quite different from the ergodic behavior of nondegenerate diffusions 
on compact manifolds without boundary, where all empirical measures $\empi{x}_{t,\omega}$ 
(for all $x$ and almost every~$\omega$) converge to the unique ergodic invariant probability measure of the process. Instead, for all initial conditions inside the square, the process is almost surely attracted to the boundary, visiting the four corners of the square cyclically. While transitions from corner to corner along edges are relatively fast, the times the process
spends in the corners grow fast enough to ensure that the single most recent visit to a corner dominates the entire empirical measure, so the empirical measures do not converge. Thus, despite the compactness of the phase space $\X$ and the high regularity of the evolution (the solutions are given by a flow of diffeomorphisms defined on the entire~$\X$ including the boundary), the long-term behavior of the process is not described in terms of the traditional ergodic theory. It is one of the main messages of this paper that degeneracies of the diffusion\,---\,a property forced by the presence of a boundary in our setting\,---\,can substantially impact
the laws governing the large-time behavior of the process.

Our final classification result in dimension~2 depends on all the assumptions introduced above. This is the only result where condition~\ref{req:Hor} is used; see Remark~\ref{rem:hypoelliptic-discussion}.

\begin{theorem}[Complete classification]
\label{thm:main2d} 
\label{thm:2d-trich}
    Under Assumptions~\ref{req:smoothness-2d}--\ref{req:index-stab-ne-1}, one and only one of the following holds:
    \begin{enumerate}[I.]
        \item \label{it:attractor-case} There is an attracting vertex or edge, i.e., $\Attr\neq\emptyset$, and Theorem~\ref{thm:2d-attr} applies.
        \item \label{it:cycle-case} There is a stochastic cycle with  stability index $\Index > 1$, and Theorem~\ref{thm:2d-cycle} 
        and Proposition~\ref{prop:quadri} apply.
        \item \label{it:recurrent-case} There is an ergodic invariant probability measure~$\invm_\circ$ concentrated on $\bXo$, absolutely continuous with respect to the Lebesgue measure on~$\bXo$ and, for every $x \in \bXo$, 
        almost surely, we have $X^x(t) \to \bX$ and $\empi{x}_t \to \invm_\circ$ as $t\to\infty$.
    \end{enumerate}
\end{theorem}

\subsection{Towards a more general dynamical picture}\label{sec:summary-result}

Our results describe large-time behaviors of typical sample paths starting
from $\bXo$. For a fixed starting point $x \in \bXo$, we say that a dynamical
behavior is typical if it
occurs with positive probability. The irreducibility assumption on the diffusion on $\bXo$ implies that this notion does not depend on the initial condition $x \in \bXo$, and we can simply talk about typical large-time behavior.

In dimensions~$1$ and~$2$, under the assumptions given in  Sections~\ref{sec:1d} and~\ref{sec:2d-setup-result}, we
have a complete classification of all typical large-time behaviors. 
Namely, each diffusion
system satisfying these conditions fits one of the following three 
mutually exclusive
scenarios (organized slightly differently than
it is done in Theorem~\ref{thm:main2d}): 

\begin{scenario}[Unique ergodic attractor]
    \label{scen:uniqe-erg-attr}
    There is a closed set $S$ and an ergodic probability measure $\invm$ with support $S$ such that, for every $x \in \bXo$, almost surely, $X^x(t) \to S$ and $\empi{x}_t=\frac{1}{t} \int_0^t \delta_{X^x(s)}\dd s\to \invm$ as $t\to\infty$.
\end{scenario}    
This scenario may happen with either $S={\X}$ and $\invm = \invm_\circ$, or with $S$ being the closure of the unique element~$A$ of~$\Attr$ and $\invm$ being the unique ergodic measure~$\invm_A$ concentrated on~$A$.
\begin{scenario}[Random ergodic attractors]
    \label{scen:random-erg-attr}
    There are $k \geq 2$ distinct closed sets $S_1, S_2, \dotsc, S_k$ and ergodic measures $\invm_1, \invm_2, \dotsc, \invm_k$ with respective supports $S_1, S_2, \dotsc, S_k$ such that, for every~$x \in \bXo$, there exist positive numbers $p_1^x, p_2^x, \dotsc, p_k^x$ summing to 1~giving the probabilities that $X^x(t) \to {S_i}$ and $\empi{x}_t \to \invm_i$ as $t\to\infty$. In addition, ${S_j} \not\subset {S_i}$ whenever $j\neq i$.
\end{scenario}
Note that
although the supports of $\invm_i$ can be viewed as random ergodic attractors, one cannot classify the points of $\bXo$ into $k$ ergodic components, one associated with each $\invm_i$, as the numbers $p_1^x, p_2^x, \dotsc, p_k^x$ are all positive for every~$x$.

In both dimensions~1 and~2, this may happen with $S_1, S_2, \dotsc, S_k$ the closures of the distinct elements of~$\Attr$.
Note that the noninclusion property constraints the 
possible configurations: none of the $S_i$ can be all of~$\bX$, and if one of the $S_i$ is the closure of an edge, then none of the vertices it connects can be an~$S_j$.
\begin{scenario}[Unique nonergodic attractor]
    \label{scen:nonerg-attr}
    There exists a closed set~$S$ such that, for every $x \in \bXo$, almost surely, we have $X^x(t) \to S$ as $t \to \infty$, but the empirical distributions $\empi{x}_t$ do not converge as $t\to\infty$. The set~$S$ supports a continuous family~$\Gamma$ of nonergodic limit points of $\empi{x}_t$ as $t \to \infty$.
\end{scenario}
This scenario is ruled out in dimension~1, but may happen in dimension 2, with $S = \partial \bXo$ and $\Gamma$ the curve described in Theorem~\ref{prop:quadri}.

\medskip

In this paper, we have taken the view that not all ergodic invariant measures are equally important, namely, some are {\it observable}  and others are not. In Scenarios~\ref{scen:uniqe-erg-attr} and~\ref{scen:random-erg-attr}, the observable measures are the ergodic measures supported by attractors. For all large $t$, one observes only the statistics described by these ergodic measures.
In Scenario~\ref{scen:nonerg-attr}, 
the limit points of the empirical measures are mixtures of 
Dirac measures at vertices, and contributions of these Dirac measures  to the mixtures are positive.
This means that all these ergodic measures
are observable: a typical trajectory is sequentially exposed to averaging with respect to each of them. They all contribute to the long-term statistics of the system, and these statistics are time dependent. However, the length of epochs during which the statistics of each measure is sampled grows exponentially.

Our broader aim in future work is to identify all observable invariant measures of diffusions satisfying some genericity assumptions. In addition to dimensions~$1$ and $2$, our results in this
paper give a partial description of observable invariant measures in arbitrary dimension, though
in high dimensions the scenarios are not limited to those described above, and 
more work is needed for a complete classification; see Section 8.

Our results can be interpreted in terms of recurrence and transience (persistence and extinction in the terminology of~\cite{Bepreprint,FSpreprint}). In fact, we provide a more detailed classification of dynamical and statistical behavior, especially for the transient case, where we describe how the diffusion gets attracted to the boundary.

\section{Proofs in dimension 1}
\label{sec:proofs-1d}
In this section, we give a proof of Theorem~\ref{prop:1d-sink}. Although we do not claim any novelty here, this study prepares us for the analysis in dimension~2.

We begin with a lemma collecting some ingredients we need. For~$x,x'\in\bXo$, we denote 
$\tau_{x,x'}=\inf\{t\ge 0:\ X^x(t)=x' \}$, where we set $\inf \emptyset=+\infty$.
For~$x,x'\in\bXo$ and $k\in\{0,1\}$, we denote
\begin{align*}
    p_k^x(x')&=\pp (B_k^x \cap  \{\tau_{x,x'}=\infty\}),\\
    q^{x}(x')&= \pp \{\tau_{x,x'}<\infty\}.
\end{align*}     
    
\begin{lemma} \label{lem:1d-aux} 
Suppose Assumptions~\ref{req:smoothness-1d}--\ref{req:ellip-1d} hold.
\begin{enumerate}
    \item\label{it:repelling-endpoint0}
    If $\lambda_0>0$, then there is $x'\in(0,1)$ such that $p_0^x(x')=0$ and $q^x(x')=1$ for all $x\in(0,x')$; moreover,
    for any such choice of $x'$ and for all $x\in(0,x')$, the tail $\pp\{\tau_{x,x'}>t\}$ decays exponentially in~$t$.
    \item\label{it:attracting-endpoint0}
    If $\lambda_0<0$, then there is $x'\in(0,1)$ such that  $p_0^x(x')>0$ and $q^x(x')<1$ for all $x\in(0,x')$. 
    \item\label{it:exit-from-middle} 
    If $0<x'<x<x''<1$, then $\pp\{\tau_{x,x'}\wedge \tau_{x,x''}<\infty \}=1$.
\end{enumerate}
\end{lemma}        
\begin{proof} 
    Part~\ref{it:exit-from-middle} follows from nondegeneracy of $\sigma$ on $(x',x'')$; see, e.g., Lemma~7.4. in~\cite[Section~{5.7.A}]{KaSh}. For parts~\ref{it:repelling-endpoint0} and \ref{it:attracting-endpoint0}, we
    change variables $Y(t)=\ln X(t)$ and rewrite the resulting SDE in the It\^o form
    \begin{equation}
    \label{eq:Ito-SDE-for_Y}
        \dd Y(t) = \ell(Y(t))\dd t + s (Y(t))\dd W(t),
    \end{equation}
    where $s$ is a positive continuous  function and $\ell$ is a continuous function satisfying 
    \begin{equation}
    \label{eq:Lyap-exp-as-limit-at-infty}
        \lim_{y\to-\infty} \ell(y)=\eivalline^0.
    \end{equation}
    Since $\lambda_0\ne0$, there is $y'\in(0,1)$ and $c>0$ such that for all $y<y'$, we have $\ell(y)>c$  in the case of $\lambda_0>0$ and $\ell(y)<-c$ in the case of~$\lambda_0<0$. This allows to compare the increments of $Y$ to 
    the increments of one of the  
    processes $Z(t)=\pm c t+M(t)$ for some martingale $M\in\cM$. The definition of the class $\cM$ and various results on such processes $Z$ are given in Appendix~\ref{app:martingale}. In particular, part~\ref{it:repelling-endpoint0} follows from Corollary~\ref{cor:new-exp-moment},
    and part~\ref{it:attracting-endpoint0} follows from Lemma~\ref{lem:drift-wins-over-mart-bound} and Corollary~\ref{cor:drift-wins-over-mart-div}.
\end{proof}

\begin{proof}[Proof of Theorem~\ref{prop:1d-sink}]    
    Suppose both endpoints are sinks. Using part~\ref{it:attracting-endpoint0} of Lemma~\ref{lem:1d-aux} near both endpoints, we find $x_0, x',x'',x_1$  satisfying $0<x_0<x'<x''<x_1<1$ such that $p_0^{x_0}(x')>0$ and $p_1^{x_1}(x'')>0$. Then, we use part~\ref{it:exit-from-middle} of Lemma~\ref{lem:1d-aux} and the strong Markov property to see that the following happens with probability $1$: the process makes at most finitely many transitions between points $x_0$ and $x_1$, after which it converges to one of the endpoints. In fact, the probability that the number of such returns to $x_0$ exceeds $n$ is bounded by $((1-p_0^{x_0}(x'))(1-p_1^{x_1}(x'')))^n$.

    If only one endpoint (say, $1$) is a sink, and the other one is not, then we can use parts~\ref{it:repelling-endpoint0} 
    and~\ref{it:attracting-endpoint0} of Lemma~\ref{lem:1d-aux} near, respectively, $0$ and $1$, and find points $x_0, x',x'',x_1$  satisfying $0<x_0<x'<x''<x_1<1$ such that $q^{x_0}(x')=1$ and $p_1^{x_1}(x'')>0$. Combining this with part~\ref{it:exit-from-middle} of Lemma~\ref{lem:1d-aux} and the strong Markov property, we obtain that the following happens with probability $1$: the process makes finally many transitions between $x_0$ and $x_1$, after which it converges to $1$.
    This completes the proof of part~\ref{it:there-are-sinks} of Theorem~\ref{prop:1d-sink}.

    To prove part~\ref{it:1d-averaging}, we assume that no endpoint is a sink. By Assumption~\ref{req:hyperbolicity-1d}, this implies both endpoints are sources, and, due to part~\ref{it:repelling-endpoint0} of Lemma~\ref{lem:1d-aux} and its natural counterpart at the other endpoint, there is a compact set $R\subset \bXo$ with the following property: for every starting point in $
    \bXo\setminus R$ the hitting time for $R$ has exponential tails. Moreover, for every compact set $K\subset \bXo$, the exponential bounds on tails of hitting times for $R$ are uniform over starting points in $K$. 
    This implies existence of an invariant probability measure concentrated on~$\bXo$; see, e.g.,~\cite[Section~{4.4}]{Kha}.
    Uniqueness, convergence and absolute continuity now follow from Assumptions~\ref{req:smoothness-1d} and~\ref{req:ellip-1d}; see, e.g.,~\cite[Sections~{4.4--4.9}]{Kha} or \cite[Section~{3.0}]{KuSh}.
\end{proof}

\begin{remark} 
    In this proof, it is crucial that functions like $\Phi(x)=\ln x$ near~$0$  play the role of a 
    Foster--Lyapunov function for the process $X$ making the bounded variation part of the process $Y(t)=\Phi(X(t))$ monotone, increasing or decreasing. As a result, the long-term evolution of~$Y$ (and $X$) is dominated by these unidirectional changes.  We will use similar Lyapunov functions in our study of $2$-dimensional systems.  
    The relation~\eqref{eq:Lyap-exp-as-limit-at-infty} also allows to interpret~$\eivalline^0$ as the Lyapunov exponent, i.e., the exponential growth/decay rate of $X$ near $0$. Its 2-dimensional counterpart will also play an important role.
\end{remark}  

\begin{remark}
\label{rem:mixing}
    The aforementioned bounds on the hitting time of the compact set~$R$, together with classical lower bounds on transition probabilities from~$R$ implied by Assumptions~\ref{req:smoothness-1d} and~\ref{req:ellip-1d}, suffice to use a classical coupling argument that gives exponential mixing in the total-variation distance; see e.g.~\cite[Section~{3.0}]{KuSh}. As a particular consequence, we have that, for every $x$ in $\bXo$ and every  bounded measurable $f : \bX \to \rr$, 
    \[ 
        \int_0^\infty \Big|\ee[f(X^x(t))] - \int_{\bX} f \dd\invm \Big| \dd t < \infty.
    \]
\end{remark}

\section{Local behavior near attracting vertices and edges in dimension 2}
\label{sec:local-attr}

\subsection{Preliminaries}
\label{sec:coord-prelim}

For $f \in C^2(\bXo)$, the (local) generator $\cL f$ for the diffusion~\eqref{eq:main-eq-vec} is  defined by
\begin{align}
\label{eq:def-cL-as-diff}
    \cL f
        &= \sum_{i} \Bigl(b_{i} + \frac 12 \sum_{j} \sum_{m} \partial_{j}\sigma_{mi} \sigma_{mj} \Bigr)\partial_{i} f 
        + \frac 12 \sum_{i,j} \sum_{m} \sigma_{mi} \sigma_{mj} \partial_{ij}f.
\end{align}
The It\^o formula and the SDE~\eqref{eq:main-eq-vec} imply that if $f \in C^2(\bXo)$, then, for all $x$, with probability~$1$,
\begin{equation}
\label{eq:first-explicit-Ito}
    f(X^x(t)) - f(x) = \int_{0}^t \cL f(X^x(s)) \dd s + \sum_{i}\sum_{m} \int_0^t \partial_{i} f(X^x(s))\sigma_{mi}(X^x(s))\dd W_{m}(s),\quad t\ge 0. 
\end{equation}
We refer to \cite[Sections 4.2, 7.3, 6.1]{Oksendal:MR2001996} for the basic reminders about the It\^o formula, generators, and the transition between It\^o and Stratonovich diffusions resulting in a correction to the drift.  
The function $\cL f$ can be interpreted as the average local rate of change of the process $(f(X(t)))_{t\geq0}$. We are particularly interested in functions~$f$
admitting bounds on $\cL f$ in key regions. Relevant results are given in this section and Sections~\ref{sec:YM} and~\ref{sec:proofs-Lyap}.

In view of Assumptions~\ref{req:smoothness-2d} and~\ref{req:invar-2d},  the coefficients in~$\cL$ have linear expansions near~$\Or$ that allow us to approximate the rate of expected change of the logarithmic distance from~$\Or$.
\begin{lemma}
\label{lem:Lyap-corner-pre}
    Under Assumptions~\ref{req:smoothness-2d} and~\ref{req:invar-2d}, 
    as $d(x,\Or)\to 0$, we have
    \begin{align}
        \label{eq:logL-near-O}
        (\cL(\ln{}\circ x_1))(x) = \eivalh + o(1)
        \quad\text{ and }\quad
        (\cL(\ln{}\circ x_2))(x) = \eivalv + o(1),
    \end{align}
    where $\eivalh = \partial_1 b_1(0,0)$ and $\eivalv = \partial_2 b_2(0,0)$.
\end{lemma}

\begin{proof}
    The following summation conventions are used here and in other generator computations below: sums over~$i$ and $j$ range over~$\{1,2\}$ (two coordinates); sums over~$m$ range over~$\{1,2\}$ (two noises).
    It is convenient to separate terms with different behaviors as $x_i \to 0$ for a given~$i$:
    \begin{align}
    \label{eq:def-cL-as-diff-order-split}
        \frac 12 \sum_{j} \sum_{m} \partial_{j}\sigma_{mi} \sigma_{mj} \partial_{i} f 
        &=
        \frac 12 \sum_{m} \partial_{i}\sigma_{mi} \sigma_{mi} \partial_{i} f + \frac 12 \sum_{m} \partial_{i'}\sigma_{mi} \sigma_{mi'} \partial_{i} f,
    \end{align}
    where $i'$ is used for the element of $\{1,2\}$ that is not~$i$.
    We
    compute using~\eqref{eq:def-cL-as-diff}, \eqref{eq:first-explicit-Ito}, and~\eqref{eq:def-cL-as-diff-order-split}:
    \begin{align*}
        \cL(\ln{}\circ x_{i})
        &= \frac{b_{i}}{x_{i}}  + \sum_{m}\left[\frac{ \partial_{i'} \sigma_{mi} \sigma_{mi'}}{2x_{i}} + \frac{ \partial_{i} \sigma_{mi} \sigma_{mi}}{2x_{i}} -  \frac{\sigma_{mi} \sigma_{mi}}{2x_{i}^2}\right] \\
        &= \frac{b_{i}}{x_{i}}+ \sum_{m}\left[\frac{ \partial_{i'} \sigma_{mi}  \sigma_{mi'}}{2x_{i}} + \frac{ [x_{i}\partial_{i} \sigma_{mi}  - \sigma_{mi}] \sigma_{mi}}{2x_{i}^2}\right].
    \end{align*}
    By Assumptions~\ref{req:smoothness-2d} and~\ref{req:invar-2d}, we have $\partial_{i'} \sigma_{mi} (x)= O(x_{i})$, $\sigma_{mi'}(x) = O(x_i')$,  $\sigma_{mi}(x) = x_{i}\partial_{i} \sigma_{mi}(x) + o(x_{i})$, $\sigma_{mi}(x) = O(x_{i})$ and $b_i(x) = x_{i}\partial_{i} b_{i}(0,0) + o(x_{i}) + O(x_{i'})O(x_{i})$ as $d(x,\Or)\to 0$.
    Therefore, 
    \begin{align*}
        (\cL(\ln{}\circ x_{i}))(x) 
        &= \frac{x_{i}\partial_{i} b_{i}(0,0) + o(x_{i}) + O(x_{i'})O(x_{i})}{x_{i}} + \sum_{m}\left[\frac{ O(x_{i}) O(x_i')}{2x_{i}} + \frac{ o(x_{i}) O(x_{i})}{2x_{i}^2}\right] \\
        &= \partial_{i} b_{i}(0,0) + o(1)
    \end{align*}
    as $d(x,\Or)\to 0$.
\end{proof}

\begin{corollary}
\label{cor:Lyap-corner-pre}
    Under Assumptions~\ref{req:smoothness-2d},~\ref{req:invar-2d}, and~\ref{req:hyperbolicity-2d}, there exists $r\in(0,\tfrac 14)$ such that 
    \begin{align*}
        d(x,\Or)<r \qquad\Rightarrow\qquad
            |(\cL(\alpha \ln{}\circ x_1 + \beta \ln{}\circ x_2))(x) - (\alpha\eivalh + \beta\eivalv)| \leq \tfrac 12 |\alpha||\eivalh| + \tfrac 12 |\beta||\eivalv|.
    \end{align*}
    for all constants $\alpha$ and $\beta$.
\end{corollary}

\begin{proof}
    This follows from Lemma~\ref{lem:Lyap-corner-pre} since, under Assumption~\ref{req:hyperbolicity-2d}, $\eivalh,\eivalv\ne 0$. 
\end{proof}

Our next goal is to obtain similar properties near the edge~$\Ed$ {using expansion of the coefficients of~$\cL$ near $\Ed$}.
\begin{lemma}
\label{lem:Lyap-edge-pre}
    Under Assumptions~\ref{req:smoothness-2d} and~\ref{req:invar-2d}, 
    as $d(x,\Ed)\to 0$,
    \begin{align}
        \label{eq:L-log-near-E}
        (\cL(\ln{}\circ x_2))(x) = \oldlambda(x_1,0) + o(1),
    \end{align}
    where $\oldlambda$ is defined by~\eqref{eq:new-def-oldlambda}.
\end{lemma}

\begin{proof}
    As in the proof of Lemma~\ref{lem:Lyap-corner-pre}, we have
    \begin{align*}
        \cL(\ln{}\circ x_2)
        &= \frac{b_2}{x_2}+ \sum_{m}\left[\frac{ \partial_{1} \sigma_{m 2}  \sigma_{m 1}}{2x_2} + \frac{ [x_2\partial_{2} \sigma_{m 2}  - \sigma_{m 2}] \sigma_{m 2}}{2x_2^2}\right].
    \end{align*}
    Using Assumptions~\ref{req:smoothness-2d} and~\ref{req:invar-2d}, we find that
    \begin{align*}
        \cL(\ln{}\circ x_2)(x)
        &= \frac{x_2\partial_2b_2(x_1,0) + o(x_2)}{x_2} \\
        &\qquad\qquad{} + \sum_{m}\left[\frac{ [x_2 \partial_{21} \sigma_{m 2}(x_1,0) + o(x_2)] [\sigma_{m 1}(x_1,0) + O(x_2)]}{2x_2} + \frac{ o(x_2) O(x_2)}{2x_2^2}\right] \\
        &= \partial_2b_2(x_1,0) + \frac 12 \sum_{m} \partial_{21} \sigma_{m 2}(x_1,0) \sigma_{m 1}(x_1,0) + o(1)
    \end{align*}
    as $d(x,\Ed)\to 0$.
\end{proof}

Unlike the right-hand sides in~\eqref{eq:logL-near-O}, the right-hand side of~\eqref{eq:L-log-near-E} is not guaranteed to have a definite sign near~$\Ed$. Thus  there is no useful corollary to Lemma~\ref{lem:Lyap-edge-pre}  stated in terms of $\ln x_2$ that would be a perfect counterpart to Corollary~\ref{cor:Lyap-corner-pre}. However, following the strategy of~\cite{FK22}, we will
obtain such a counterpart for a corrected logarithmic distance. Namely, we will show  that if
\begin{equation}
\label{eq:Lyap-for-edge-attr-rep}
    \Phi(x)=\factorLyap(\ln x_2+\psi(x_1)),\quad x\in \bXo,
\end{equation}
for a suitably chosen {\it corrector}~$\psi$, then $\cL\Phi$ will be close to
$\factorLyap\oldbarlambda$  near $\Ed$. Recalling that  $\oldbarlambda$ 
was defined in~\eqref{eq:average-repulsion-over-edge} and required to be nonzero due to Assumption~\ref{req:Lyap-of-edge}, we will be able to choose the constant $\factorLyap\in\R$ to ensure an estimate like $\cL\Phi<-1$.
Thus $\Phi$ can be viewed as a Lyapunov function near $\Ed$; see Corollaries~\ref{cor:pre-mean-H-process} and~\ref{cor:mean-H-process} below for precise statements.

Before implementing this plan, let us see what Lemmas~\ref{lem:Lyap-corner-pre} and~\ref{lem:Lyap-edge-pre} say about our diffusion in a special useful coordinate system. To define the new coordinates, we fix a pair of smooth increasing functions~$y_1 : (0,1) \to \rr$ and $y_2 : (0,1) \to \rr$ with the property that
$y_i(x_i) = \ln x_i$ for $x_i < \tfrac 14$ and $y_i(x_i) = - \ln (1-x_i)$ for $x_i > \tfrac 34$. Let us also set $y = (y_1,y_2)$. 
The following lemma follows from Lemmas~\ref{lem:Lyap-corner-pre} and~\ref{lem:Lyap-edge-pre} and the It\^o formula~\eqref{eq:first-explicit-Ito}.

\begin{lemma}
    \label{lem:Y}
    Under Assumptions~\ref{req:smoothness-2d}, \ref{req:invar-2d} and \ref{req:hyperbolicity-2d}, if  $X$ satisfies~\eqref{eq:maineq}, then $Y(t)=y(X(t))$ solves  a system of It\^o SDEs
    \begin{subequations}
    \label{eq:alt-SDE-first}  
    \begin{align}
        \dd Y_{1}(t) &= \ell_1(Y(t)) \dd t + \sum_{m}s_{m 1}(Y(t)) \dd W_{m}(t), \label{eq:alt-SDE-first-h}
        \\
        \dd Y_{2}(t) &= \ell_2(Y(t)) \dd t + \sum_{m}s_{m 2}(Y(t)) \dd W_{m}(t), 
    \end{align}
    \end{subequations}
    where $\ell_1$, $\ell_2$, $s_{11}$, $s_{12}$, $s_{21}$ and $s_{22}$ are bounded and in $C^2(\rr^2)$.
    Moreover,
    there exists $y_* \in (-\infty,-1)$ such that 
    \begin{subequations}
    \label{eq:exp-near-infinities}
    \begin{align}
        y_1\wedge y_2 &< y_* 
        \qquad\Rightarrow\qquad 
        |\ell_1(y) - \eivalh| \leq \tfrac 12|\eivalh|
        \quad\text{ and }\quad
        |\ell_2(y) - \eivalv| \leq \tfrac 12|\eivalv| 
        ,
        \\
        y_2 &< y_* 
        \qquad\Rightarrow\qquad 
        |\ell_2(y) - \oldlambda(x_1(y_1),0)| \leq \tfrac 14|\oldbarlambda|, 
    \end{align}
    \end{subequations}
    where $\oldbarlambda$ is defined in \eqref{eq:average-repulsion-over-edge}
\end{lemma}

\begin{remark}
\label{rem:strat-ito-combo}
    Stratonovich equations are more geometric while It\^o equations allow for direct martingale analysis. Thus, depending on the context, it will be convenient to use~\eqref{eq:maineq},~\eqref{eq:alt-SDE-first}, or some combination of them.
\end{remark}

Now let us find a function $\psi$ that would be useful in 
the definition~\eqref{eq:Lyap-for-edge-attr-rep} of the Lyapunov function~$\Phi$. Following  \cite[Section~{3}]{FK22}, we define
\begin{equation}
    \label{eq:def-psi}
        \psi(x_1)  = \int_0^\infty \ee g(X_1^{(x_1,0)}(t)) \dd t,\quad x_1\in(0,1),
    \end{equation}
where 
\begin{equation}
    \label{eq:g_is_centered_Lambda}
    g(x_1)=\oldlambda(x_1,0)-\oldbarlambda,\quad x_1\in[0,1].
\end{equation}  
By Remark~\ref{rem:mixing}, 
the integral on the right-hand side of~\eqref{eq:def-psi} converges absolutely
for any continuous function~$g : [0,1] \to \rr$ with zero mean with respect to~$\invm_{\Ed}$, in particular for $g$ defined in~\eqref{eq:g_is_centered_Lambda}. 

\begin{lemma} 
\label{lem:psi}
    Suppose Assumptions~\ref{req:smoothness-2d},~\ref{req:invar-2d},~\ref{req:irreducibility} and \ref{req:hyperbolicity-2d} hold. 
    Suppose that 
    $\Ed \in \Supt$ 
    and that $\psi: (0,1)\to\rr$ is defined by \eqref{eq:def-psi}  
    for a continuous function~$g : [0,1] \to \rr$ with zero mean with respect to~$\invm_{\Ed}$. If $g$ is differentiable on~$(0,1)$ and has a bounded first derivative, then 
    the function $\varphi:\bXo\to\rr$ defined by $\varphi(x)=\varphi(x_1,x_2)=\psi(x_1)$ satisfies
    \begin{align*}
        \cL\varphi(x) 
        &= -g(x_1) + o(1) 
    \end{align*}
    as $d(x,\Ed)\to 0$.
\end{lemma} 
We prove this lemma in Appendix~\ref{app:fk-function}. Note that 
in contrast with the situation studied in~\cite{FK22}, 
we need to take additional care of the fact that 
the set $\Ed$ is not compact, and the corrector $\psi$ defined by~\eqref{eq:def-psi}--\eqref{eq:g_is_centered_Lambda}
typically fails to be bounded, growing logarithmically near the endpoints of $\Ed$.  

\begin{corollary}
\label{cor:pre-mean-H-process}
    Suppose that Assumptions~\ref{req:smoothness-2d},~\ref{req:invar-2d},~\ref{req:irreducibility} and \ref{req:hyperbolicity-2d} hold and that
    $\Ed \in \Supt$. 
    Let $\Phi$ be defined 
    by~\eqref{eq:Lyap-for-edge-attr-rep}, where  
    the corrector $\psi$ is defined by~\eqref{eq:def-psi} for $g$ given in \eqref{eq:g_is_centered_Lambda}.
    Then, as $d(x,\Ed)\to 0$,
    \begin{align}
    \label{eq:pre-mean-H-process}
        \cL \Phi(x) = \factorLyap\oldbarlambda + o(1).
    \end{align}    
\end{corollary}
\begin{proof}
    Due to the definition of~$\oldbarlambda$ and Assumptions~\ref{req:smoothness-2d} and~\ref{req:invar-2d}, the function $g$ defined by~\eqref{eq:g_is_centered_Lambda} satisfies the conditions of Lemma~\ref{lem:psi}. 
    The estimate~\eqref{eq:pre-mean-H-process} follows from Lemmas~\ref{lem:Lyap-edge-pre} and Lemma~\ref{lem:psi}. 
\end{proof}

\begin{corollary}
\label{cor:mean-H-process}
    Suppose Assumptions~\ref{req:smoothness-2d},~\ref{req:invar-2d},~\ref{req:irreducibility} and \ref{req:hyperbolicity-2d}--\ref{req:Lyap-of-edge} 
    hold.
    If $\Ed \in \Supt$ and $\Phi$ is defined 
    by~\eqref{eq:Lyap-for-edge-attr-rep} and~\eqref{eq:def-psi} for 
    $g$ given in \eqref{eq:g_is_centered_Lambda}, then there is $r\in(0,\tfrac 14)$, independent of the choice of~$\factorLyap$ in~\eqref{eq:Lyap-for-edge-attr-rep}, such that
    \begin{align*}
        d(x,\Ed) < r 
        \qquad\Rightarrow\qquad 
        |\cL\Phi(x) - \factorLyap\oldbarlambda| \leq \tfrac 12 {|\factorLyap|}|\oldbarlambda|.
    \end{align*}
\end{corollary}

\begin{proof}
    This follows from Corollary~\ref{cor:pre-mean-H-process} since, under Assumption~\ref{req:Lyap-of-edge}, the mean~$\oldbarlambda$ is nonzero.
\end{proof} 

On a different note, let us now make more precise our remark that Assumption~\ref{req:irreducibility} allows us to control exit times from compact sets. 
We begin with a lemma on the transitions from compact sets to open sets, of which exits from compact sets are a special case.

\begin{lemma}
\label{lem:uniform-irred}
    Suppose that Assumptions~\ref{req:smoothness-2d} and~\ref{req:irreducibility} hold. 
    Then, for every compact set $R \subset \bXo$ and every open set $V \subset \bXo$, there exists $T \geq 0$ such that the stopping times 
    \begin{align*}
        \hit_V^x = \inf\{t \geq 0 : X^x(t) \in V\},\quad x\in R
    \end{align*}
    satisfy
    \begin{equation*}
        \inf_{x \in R} \pp\{\hit_V^x \leq T \} > 0.
    \end{equation*} 
\end{lemma}

\begin{proof}
    By Assumption~\ref{req:irreducibility}, for every $x \in R$, there is $t_x \geq 0$ such that 
    $q(x)>0$, where  
    $q(x)=\pp\{X^x(t_x) \in V\}$. Recall that we have the Feller property and thus that, for each $x$, the map $x' \mapsto \pp\{X^{x'}(t_x) \in V\}$ is lower semicontinuous, and there is an open set $U_x\subset \bXo$ such that if~$x' \in U_x$, then $\pp^{x'}\{X(t_x) \in V\} > \tfrac 12 q(x) > 0$. These sets form an open cover of the compact set $R$, so we can find a finite set $\{x^{(j)}\}_{j=1}^M$ such that $R\subset \bigcup_{j=1}^M U_{x^{(j)}} $. Setting $T = \max_{j=1,\dotsc,M} t_{x^{(j)}}$, we have
    \begin{align*}
        \pp \{\hit_V^{x'} \leq T \} 
        &\geq \max_{j=1,\dotsc,M}  \pp\{X^{x'}(t_{x^{(j)}}) \in V \} 
        > \min_{j=1,\dotsc,M} \tfrac 12 q(x^{(j)}) > 0,\quad x' \in R,
    \end{align*}
    as desired.
\end{proof}

\begin{lemma}
\label{lem:exit-compact}
    Suppose that Assumptions~\ref{req:smoothness-2d} and~\ref{req:irreducibility} hold. Then, for every compact set $R \subset \bXo$ and every $x\in\bXo$, the exit time 
    \begin{equation*}
        \exit_{R}^x = \inf\{ t \geq 0 : X^x(t) \notin R \}
    \end{equation*}
    is almost surely finite, and has a finite exponential moment that is bounded uniformly in $x$.
\end{lemma}

\begin{proof}
    By Lemma~\ref{lem:uniform-irred}, there exists $\delta > 0$ and $T > 0$ such that
    $
        \sup_{x \in R} \pp\{\exit_R^x > T \} < 1-\delta.
    $
    Since $X^x(t) \in R$ while $t < \exit_R^x$, we may use the strong Markov property repeatedly 
    to deduce that  
    $
        \sup_{x \in R} \pp\{\exit_R^x > nT \} < (1-\delta)^n
    $ 
    for all $n\in\nn$.
    We conclude that $\exit_R^x$ is almost surely finite and has a finite exponential moment, uniformly in~$x \in R$. It is clear however that the same property also holds for $x \notin R$.
\end{proof}

We end these preliminaries with a more abstract, measure-theoretic lemma that will be used in a variety of situations.

\begin{lemma}
\label{lem:levy-conseq}
    Let $(\eta_n)_{n\in\nn}$ be a sequence of almost surely finite stopping times, and let $(\cF_{\eta_n})_{n\in\nn}$ be the filtration of $\sigma$-algebras containing information up to and including time~$\eta_n$. If the event $A$ is measurable with respect to the  $\sigma$-algebra $\sigma(\cF_{\eta_n} : {n\in\nn})$, and if there exists $\delta > 0$ such that the almost sure inequality $\ee[\one_A|\cF_{\eta_n}]\geq\delta$ holds for all~$n$, then $\pp(A) = 1$.
\end{lemma}

\begin{proof}
    By L\'evy's 0-1 law, we almost surely have $\ee[\one_A|\sigma(\cF_{\eta_n} : {n\in\nn})] = \lim_{n\to\infty} \ee[\one_A|\cF_{\eta_n}] \geq \delta$. This also follows from Doob's martingale convergence theorem. By the measurability assumption, we have the almost sure identity $\ee[\one_A|\sigma(\cF_{\eta_n} : {n\in\nn})] = \one_A$. To conclude, note that the almost sure inequality $\one_A\geq\delta$ implies the almost sure identity $\one_A = 1$ since~$\one_A$ takes values in~$\{0,1\}$.
\end{proof}

\subsection{Conditional convergence results}
\label{ssec:local-abs}

Here, under Assumptions~\ref{req:smoothness-2d}--\ref{req:irreducibility} and \ref{req:hyperbolicity-2d}--\ref{req:Lyap-of-edge}, 
we state and prove local results near attracting edges or vertices. 
Those are at the heart of Theorem~\ref{thm:2d-attr} (and thus case~\ref{it:attractor-case} of Theorem~\ref{thm:main2d}), whose proof will only be concluded in Section~\ref{ssec:conclu-main-attr}.
Recall that, for $x\in \bXo$ and $A\in\Attr$, the event $B^x_A$ and its probability~$p^x_A$ were defined in~\eqref{eq:def-conv-prob}. 
We also define 
\begin{align*}
    \tilde{B}^x_A = \left\{ d(X^x(t), \overline{A}) \to 0 \right\}
    \qquad\text{and}\qquad
    \tilde{p}^x_A = \pp(\tilde{B}^x_A).
\end{align*}
Since $\convEvent^x_A \subset \tilde{\convEvent}^x_A$, we have $0 \leq p^x_A \leq \tilde{p}^x_A \leq 1$. When convenient, we will use the notation~$\pp^x$ for the probability measure describing the joint law of $W_1 = (W_1(t))_{t\geq 0}$, $W_2 = (W_2(t))_{t\geq 0}$ and $X = (X(t))_{t \geq 0}$, where the latter is the solution to~\eqref{eq:main-eq-vec} with deterministic initial condition $X(0) = x$.

\begin{lemma}
\label{lem:corners-attract}
    Suppose that $\Or \in \Attr$. 
    Then, 
    \begin{equation}
    \label{eq:attr-corner}
        \lim_{r \downarrow 0 }\,\inf_{ x\in \overline{U_{r,r}}} \tilde{p}^x_{\Or} = 1,
    \end{equation}
    where 
    \[ 
        U_{r,s} = (0,r) \times (0,s).
    \]
\end{lemma}

\begin{proof}
    By Lemma~\ref{lem:Y}, there exists $y^*\in(-\infty,-1)$ with the following property: for every initial condition~$x$ with $y_1(x) \vee y_2(x) \leq y^*$ for the process~\eqref{eq:maineq}, there are martingales $M,N\in\cM$ (see the definition of $\cM$ in Appendix~\ref{app:martingale})    
    such that, while $Y_1\vee Y_2< y^*$, we have
    \begin{align*}
        Y_1(t)-y_1 &\leq \tfrac 12 \lambda_1 t + M(t), \\
        Y_2(t)-y_2 &\leq\tfrac 12 \lambda_2 t + N(t).
    \end{align*}
    Although the precise martingales depend on the initial condition~$x$, an upper bound on the growth of quadratic variation can be chosen uniformly $x$. Since $\eivalh<0$ and $\eivalv<0$ due to $\Or \in \Attr$, 
    Lemma~\ref{lem:drift-wins-over-mart-bound} and Corollary~\ref{cor:drift-wins-over-mart-div} imply that, for every~$\delta>0$, there is $y_* < y^* $ with the following property:
    if $y_1(x) \vee y_2(x) \leq y_*$, then 
    \[
        \pp^x\Big\{\lim_{t\to\infty}Y_1(t) \vee Y_2(t) = -\infty \text{ and } \sup_{t\ge 0} Y_1(t) \vee Y_2(t) <y^*\Big\}\ge 1-\delta.
    \]
    Since $\lim_{r\to 0}\sup_{x\in \overline{U_{r,r}}} y_1(x) \vee y_2(x) = -\infty$, our claim follows.
\end{proof}    

An analogous result holds for attracting edges, albeit with a slightly more involved proof.
\begin{lemma}
\label{lem:edge-abrsob-prob}
    Suppose $\Ed \in \Attr$. 
    Then, for every $r\in(0,\tfrac 12)$,
    \begin{align}
    \label{eq:attr-edge}
        \lim_{s \downarrow 0 } \inf_{x \in \overline{G_{r,s}}} \tilde{p}^x_{\Ed}  = 1.
    \end{align}
    where
    $$
        G_{r,s} = (r,1-r)\times (0,s).
    $$ 
    Also, for all $r\in(0,\tfrac 12)$ and $s\in(0,1)$, $\inf_{x \in \overline{G_{r,s}}} \tilde{p}^x_{\Ed}>0$.
\end{lemma}

\begin{proof}
    We first prove~\eqref{eq:attr-edge} under the extra assumption that the function~$g = \oldlambda-\oldbarlambda$ satisfies $g(0)> 0$ and $ g(1)> 0$. 
    Under this assumption, Lemmas~\ref{lem:psi-properties-i} and \ref{lem:psi-as-time} imply that
    $\psi$ is bounded below on~$(0,1)$.
    Therefore,
    if $\Phi(x)\to-\infty$, where 
    $\Phi$ is given in \eqref{eq:Lyap-for-edge-attr-rep} with $\factorLyap=1$, then $x_2 \to 0$. 
    Thus, using Corollary~\ref{cor:mean-H-process}, we can choose $h^*$ such that the following is true for the process defined by $H(t) = \Phi(X(t))$: for every initial condition $x$ with $\Phi(x) < h^*$ for the process~\eqref{eq:maineq}, there exists a martingale $M\in\cM$ such that, while $H(t) < h^*$, we have
    \begin{align*}
        H(t)- \Phi(x) &\leq \tfrac 12 \oldbarlambda t + M(t)
    \end{align*}
    Although the precise martingale depends on the initial condition~$x$, an upper bound on the growth of quadratic variation can be chosen uniformly in $x$. Since $\oldbarlambda < 0$ due to~$\Ed \in \Attr$, Lemma~\ref{lem:drift-wins-over-mart-bound} and Corollary~\ref{cor:drift-wins-over-mart-div} imply that, for every~$\delta>0$, there is $h_*<h^*$ with the following property:
    if $\Phi(x) \leq h_*$, then 
    \[
        \pp^x\Big\{\lim_{t\to\infty}H(t) = -\infty;\ \sup_{t\ge 0} H(t)<h^* \Big\}\ge 1-\delta.
    \]
    Since $\lim_{s\to 0}\sup_{x\in \overline{G_{r,s}}} \Phi(x) =-\infty$, our claim follows.

    In the general case, 
    one can, without changing the sign of the mean with respect to~$\invm_{\Ed}$, increase~$\oldlambda(x_1)$ near $x_1=0$ and $x_1=1$ in order to obtain a new function $g^+$ that can be used in place of~$g$ in \eqref{eq:def-psi} to construct a function~$\psi^+$ that is bounded from below, and then define $\Phi^+$ accordingly.
    We can then introduce 
    $H^+(t) = \Phi^+(X(t))$ and rely on the inequality
    \begin{align*}
        H^+(t)-\Phi^+(x)
        &\leq \tfrac 12 \oldbarlambdastar t + M(t),
    \end{align*}
    where $\oldbarlambdastar$ is the\,---\,greater but still negative\,---\,mean of this modification of $\oldlambda$.

    The second part of the lemma follows from the first one, Lemma~\ref{lem:uniform-irred}, and the strong Markov property.
\end{proof}

\begin{remark}\label{rem:Lambda_as_attr_rate}
    This proof shows that the (signed) exponential rate of attraction to $\Ed$ is bounded by $\tfrac12 \oldbarlambda$. It is easy to modify this proof to show that this Lyapunov exponent actually equals $\oldbarlambda$.  
\end{remark}

We now turn to conditional characterization of the limit points and convergence of empirical distributions.
We begin with an obvious statement  for the case $\Or \in \Attr$.

\begin{lemma}
\label{cor:empi-sink} 
    Suppose $\Or\in\Attr$.
    For all $x\in\bXo$, on the event~$\tilde{\convEvent}^x_\Or$, we have $X^x(t) \to \Or$ and $\empi{x}_t \to \delta_{\Or}$ as ${t\to\infty}$.
    In particular, the same conclusions hold on $\convEvent^x_\Or$ and $p^x_\Or = \tilde{p}^x_\Or$.
\end{lemma} 

Let us now prove a similar statement for $\Ed\in\Attr$, with $\invm_\Ed$ the probability measure as in part~\ref{it:1d-averaging} of Theorem~\ref{thm:1d} applied to the restriction of~\eqref{eq:main-eq-vec} to~$\overline{\Ed}$.

\begin{lemma}
\label{lem:attr-edge-empir}
    Suppose $\Ed \in \Attr$.  For all $x\in\bXo$, on the event~$\tilde{\convEvent}^x_\Ed$, we almost surely have 
    $
        X^x(t) \to \overline{\Ed}
    $
    and
    $
    \empi{x}_t \to \invm_\Ed
    $
    as $t\to\infty$.
    In particular, the same conclusions hold on $\convEvent^x_\Ed$ and $p^x_\Ed = \tilde{p}^x_\Ed$.
\end{lemma}

\begin{proof}
    Fix $x\in\bXo$. First, note that $\empi{x}_t \to \invm_\Ed$ implies $X^x(t) \to \overline{\Ed}$ since $\operatorname{supp} \invm_\Ed = \overline{\Ed}$.
    Thus, by tightness, it suffices to show that, on $\tilde{\convEvent}^x_\Ed$, almost surely, the empirical measures $(\empi{x}_t)_{t > 0}$ have no weak limit point other than~$\invm_\Ed$.
    By Theorem~\ref{thm:benaim}, almost surely, every such limit point is invariant. The only invariant probability measures concentrated on $\overline{\Ed}$ are convex combinations of $\delta_{\Or^0}$, $\delta_{\Or^1}$ and~$\invm_{\Ed}$. In order to rule out contributions from $\delta_{\Or^0}$ or $\delta_{\Or^1}$, it suffices to note that the asymptotic relative time spent in a neighborhood of $\Or^0$ and $\Or^1$ of radius $r$ converges to $0$ as $r\to 0$. 
    The latter  
    follows from Lemma~\ref{lem:ee-ap-bound} since, say near $\Or^0$, the increments of the process $Y_1(t)=\ln X_1(t)$ are bounded below by the increments of $Z(t)=\nu t + M(t)$ for some $\nu>0$ and a martingale $M\in\cM$.
\end{proof}

\section{Heteroclinic behavior in dimension 2}
\label{sec:main-heteroc-proof}

Our goal for this section is to prove Theorem~\ref{thm:2d-cycle} and Proposition~\ref{prop:quadri}.  The theorem describes the cycling that the system goes through visiting vertices and traveling along edges sequentially. The evolution naturally decomposes into epochs spent near the vertices and epochs of transitions between small neighborhoods of the vertices. Thus our analysis naturally consists of several parts.

The first part is a study of the behavior of the system in a small neighborhood of a single saddle vertex. 
The second part studies transitions between such neighborhoods connected by an edge.   These results describe the typical behavior during the respective epochs and estimate how improbable the deviations from the typical scenario are.  
Namely, we need to control the direction of exit, the times that the associated epochs take, and how the distance from $X$ to $\partial \X$ changes over those epochs.  The third part of the analysis combines the first two and describes what happens over the combined epoch of a visit to a corner and a transition to the next corner. 
The last part sequentially integrates the combined epochs of the third part into an iterative process of cycling along the boundary accompanied by asymptotic absorption.

In the next four subsections, we describe the main four stages of analysis and give the corresponding results, assuming throughout that Assumptions~\ref{req:smoothness-2d}--\ref{req:irreducibility} and~\ref{req:hyperbolicity-2d} hold.
The proofs are given in Section~\ref{sec:proof-heteroc}.

Our description of the typical and exceptional scenarios is given in terms of high-probability events and rare events. These notions are most succinctly introduced in terms of the probability measure~$\pp^x$ describing the joint law of $W_1 = (W_1(t))_{t\geq 0}$, $W_2 = (W_2(t))_{t\geq 0}$ and $X = (X(t))_{t \geq 0}$, where the latter is the solution to~\eqref{eq:main-eq-vec} with deterministic initial condition $X(0) = x$.
Given events $A(z)$ and initial conditions $x(z)\in\X$ depending on a parameter $z\in(-\infty,z_0)$ for some~$z_0$, 
we will say that $A(z)$ is {\it rare} under $\pp^{x(z)}$ as $z\to-\infty$
if there are
constants $c_1,c_2,c_3>0$ such that
\[
    \pp^{x(z)}(A(z))\le c_1\Exp{-c_2 |z|^{c_3}}.
\]
If the complement of $A(z)$ is rare, we 
say that  $A(z)$ is a {\it high-probability event} (as~$z\to-\infty$) and  write: $A(z)$ w.h.p.  We say that $B(z)$ holds on $A(z)$  w.h.p.\ if $A(z)\setminus B(z)$ is a rare event.
Finally, we say that $B(z)$ is a rare event conditionally on $A(z)$ 
if 
\[
  \pp^{x(z)}(B(z)|A(z))\le c_1\Exp{-c_2 |z|^{c_3}}.
\]
The role of $z$ is usually played by the logarithm of the distance to the boundary expressed in an appropriate chart.

\subsection{Evolution near a single saddle} 
\label{sec:near-a-saddle-results}

We begin with the behavior near a single positively oriented saddle point, which we place at the origin~$\Or$ without loss of generality.
We analyze exits from the corner neighborhood $U_{r,2r} = (0,r) \times (0,2r)$ using the segments
\begin{align*}
    \gamma_{0,r}  &= (0,r)\times\{2r\},
    &
    \gamma_{2,r}  &=\{r\}\times(0,r), 
    \\
    \gamma_{1,r}  &= (0,r)\times\{r\}, 
    &
    \gamma'_{2,r}  &=\{r\}\times(0,2r),
\end{align*}    
see Figure~\ref{fig:U_r2r}. 
We will start the process on $\gamma_{1,r}$ and consider 
stopping times
\begin{align*}
    \exituv & =\inf\{t\ge 0: X(t)\notin U_{r,2r} \},
    &
    \tau_{0,r} & = \inf \{ t
    \ge 0 : X(t) \in \gamma_{0,r} \}, 
    \\ 
    && \tau_{2,r} & = \inf \{ t \ge 0 : X(t) \in \gamma'_{2,r} \}.
\end{align*}
While the segment $\gamma'_{2,r}\setminus \gamma_{2,r}$ will rarely be of concern, it is convenient to rely on the almost sure identity $\exituv = \tau_{0,r} \wedge \tau_{2,r}$.

\begin{figure}
    \centering
    \footnotesize
    \hfill
    \begin{tikzpicture}\begin{axis}[
            axis lines=none,
            xmin=-.15, xmax=2.2*\err,
            ymin=-.15, ymax=2.15*\err,
            zmax=12,
            unit vector ratio=1 1,
            unit rescale keep size=false,
            view={0}{90},
            scale=1
        ]
       
        \addplot3[mark=none] coordinates {(0,0,0)};
    
        \draw (0,2.5*\err) -- (0,0) -- (2.5*\err,0);

        \draw (0,2*\err) -- (\err,2*\err);
        \node[anchor=north] at (axis cs: 0.5*\err,2*\err) {$\gamma_0$};
        
        \draw (0,\err) -- (\err,\err);
        \node[anchor=north] at (axis cs: 0.5*\err,\err) {$\gamma_1$};
    
        \draw (\err,0) -- (\err,\err-0.00301);
        \node[anchor=east] at (axis cs: \err,0.5*\err) {$\gamma_2$};
    
        \draw (\err+.00301,0) -- (\err+.00301,2*\err);
        \node[anchor=west] at (axis cs: \err,\err) {$\gamma'_2$};

        \node[anchor=north] at (axis cs: 0,0) {$0$};
        \node[anchor=north] at (axis cs: \err,0) {$r$};
        \node[anchor=east] at (axis cs: 0,\err) {$r$};
        \node[anchor=east] at (axis cs: 0,2*\err) {$2r$};

        \swsaddlesmall
    
        \end{axis}
    \end{tikzpicture}
    \hfill
    \begin{tikzpicture}
        \begin{axis}[
            axis lines=none,
            xmin=-.15, xmax=2.2*\err,
            ymin=-.15, ymax=2.15*\err,
            zmax=12,
            unit vector ratio=1 1,
            unit rescale keep size=false,
            view={0}{90},
            scale=1
        ]
        
        \addplot3[mark=none] coordinates {(0,0,0)};

        \draw (0,2.5*\err) -- (0,0) -- (2.5*\err,0);

        \draw (0,\err) -- (\err-.00301,\err);
        \node[anchor=north] at (axis cs: 0.5*\err,\err) {$\gamma_1$};

        \draw (\err-.00301,0) -- (\err-.00301,\err-0.00301);
        \node[anchor=east] at (axis cs: \err,0.5*\err) {$\gamma_2$};
        
        \draw (2*\err+.00301,0) -- (2*\err+.00301,\err);
        \node[anchor=west] at (axis cs: 2*\err,0.5*\err) {$\gamma_3$};

        \draw (0,\err+.00301) -- (2*\err,\err+.00301);
        \node[anchor=south] at (axis cs: \err,\err) {$\gamma'_1$};

        \node[anchor=north] at (axis cs: 0,0) {$0$};
        \node[anchor=north] at (axis cs: \err,0) {$r$};
        \node[anchor=north] at (axis cs: 2*\err,0) {$2r$};
        \node[anchor=east] at (axis cs: 0,\err) {$r$};

        \swsaddlesmall

        \end{axis}
    \end{tikzpicture}
    \hfill{}
    \caption{The setup for Propositions~\ref{prop:corner} and~\ref{lem:2-->1,3}: neighborhoods $U_{r,2r}$ and $U_{2r,r}$ of the origin~$\Or$ with segments introduced to define stopping times. Some subscripts have been omitted for readability.}
    \label{fig:U_r2r}
    \label{fig:U_2rr}
\end{figure}
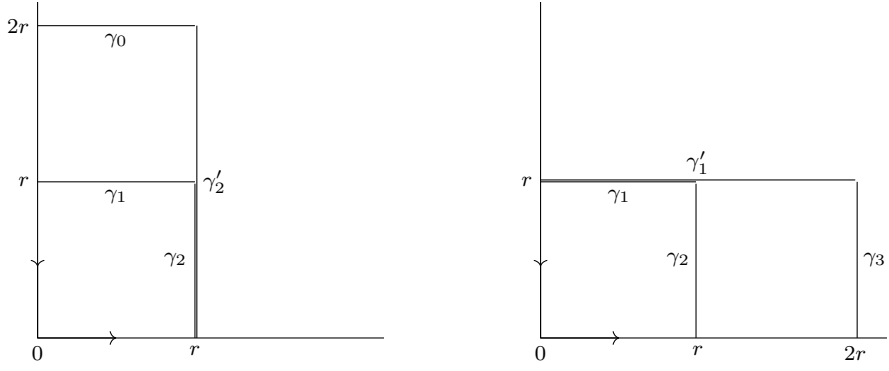

\begin{lemma}
\label{lem:proto-corner-prob}
    Suppose that there is a positively oriented saddle at the vertex~$\Or$. Then, there exists $p>0$ such that, for all sufficiently small $r>0$, we have
    \begin{equation}
        \inf_{x \in \gamma_{1,r}} \pp^{x} \{ \exituv = \tau_{2,r} \} \ge  p.
    \end{equation}
\end{lemma}

To further describe $\exituv$ and $X(\exituv)$, it is convenient to use notation from Section~\ref{sec:coord-prelim} such as
the transformed coordinates $(y_1,y_2)$ and transformed process $(Y_1,Y_2)$. For $r$ small enough, the coordinate change in $U_{r,2r}$ is given by $y_1 = \ln x_1$ and $y_2 = \ln x_2$, or equivalently $x_1 = \Exp{y_1}$ and $x_2 = \Exp{y_2}$.

\begin{proposition} 
\label{prop:corner}
    Suppose that there is a positively oriented saddle at the vertex~$\Or$.
    Then, for every $\eps > 0$ and $r_0 > 0$, there exists $r\in(0,r_0)$ such that
    \begin{enumerate}[(a)]
        \item 
        conditionally on $\exituv = \tau_{2,r}$, the following holds  w.h.p.\ under~$\pp^{(\Exp{y_1},r)}$ as $y_1\to-\infty$ : 
        \[ 
            \exituv \in \left[ \left(\frac{1}{\eivalplus} - \eps\right) |y_1|, \left(\frac{1}{\eivalplus} + \eps\right) |y_1| \right]
            \text{\rm\quad and\quad } 
            Y_2 (\exituv) \in \left[(\rho + \eps)
            y_1, (\rho - \eps) y_1\right].
        \]

        \item
        conditionally on $\exituv \ne \tau_{2,r}$
        , the following holds  w.h.p.\ under~$\pp^{(\Exp{y_1},r)}$ as $y_1\to-\infty$:
        \[
            \exituv \le | y_1 |^{1 / 3} \quad 
            \text{\rm and\quad} 
            Y_1 (\exituv) \in \big [y_1 - | y_1 |^{1 / 2}, y_1 + | y_1 |^{1 / 2}\big].
        \]
    \end{enumerate}
\end{proposition}

This proposition describes two possible scenarios of exit from $U_{r,2r}$ starting on~$\gamma_{1,r}$, with probabilities given some a priori control in the lemma. 
If the process exits through $\gamma'_{2,r}$, the exit time is typically about $|y_1|/\eivalplus$, and the logarithm of the distance to the boundary typically changes by a factor close to~$\rho$, implying in particular that the exit is actually through~$\gamma_{2,r}$. 
In comparison, the alternative typically happens fast and leads only to a small change in the distance to the boundary.
The proofs are postponed to Section~\ref{ssec:proof-heteroc-aux}.

Next, we analyze exits from $U_{2r,r} = (0,2r) \times (0,r)$ using, in addition to $\gamma_{1,r}$ and $\gamma_{2,r}$, the segments
\begin{align*}
  \gamma_{3,r}  = \{ 2 r \} \times (0, r)\qquad\text{ and }\qquad
  \gamma'_{1,r} =(0,2r)\times\{r\},
\end{align*} 
see Figure~\ref{fig:U_2rr}.
We are particularly interested in starting points on $\gamma_{2,r}$ (i.e., what happens after case~(a) in Proposition~\ref{prop:corner}),
but we will consider more general starting points in our statements. 
We define $\exituh$ to be the first exit time 
from~$U_{2r,r}$. We then have the almost sure identity $\exituh = \tau_{1,r} \wedge \tau_{3,r}$, where $\tau_{1,r}$ and $\tau_{3,r}$ are the hitting times for $\gamma'_{1,r}$ and $\gamma_{3,r}$, respectively.

\begin{proposition}
\label{lem:2-->1,3}
    Suppose that there is a positively oriented saddle at the vertex~$\Or$. 
    For sufficiently small $r > 0$, the following holds w.h.p.~under $\pp^{(x_1,\Exp{y_2})}$ as $y_2\to-\infty$, uniformly in $x_1\in[r,2r)$: 
    \[
        \exituh = \tau_{3,r}\le |y_2|^{1/2} 
        \quad\text{\rm and}\quad 
        Y_2 (\exituh) \in \big [y_2 - | y_2 |^{1 / 2}, y_2 + | y_2 |^{1 / 2}\big]
    \]
\end{proposition}

This result says that the process typically exits $U_{2r,r}$ through~$\gamma_{3,r}$ (it is a rare event to go against the drift near the saddle), this does not take too long, and the distance to the boundary does not change much. Although the process is still close to the saddle, the typical behavior here is to move away from the saddle along the unstable manifold. The proof is postponed to Section~\ref{ssec:proof-heteroc-aux}.

\subsection{Evolution along a single edge}

We now analyze exits from~$G_{r,r} = (r,1-r) \times (0,r)$, using the additional notation
\begin{align*}
  G'_{r,r'} & = [2r,1-r) \times (0, r'),
  & 
  \tilde\gamma_{1,r} & = \rotat\gamma_{1,r} = \{ 1-r \}\times (0, r),\\
  &
  &
  \gamma_{4,r} & = (r,1-r)\times \{r \},
\end{align*}
see Figure~\ref{fig:FG}. We are mostly interested in starting points on $\gamma_{3,r}$ (i.e., what happens after the typical scenario in Proposition~\ref{lem:2-->1,3}), but we will consider more general starting points in our statements.
We define $\exitg$ as the exit time from $G_{r,r}$. We then have the almost sure identity
$\exitg = \tau_{2,r} \wedge \tilde \tau_{1,r} \wedge \tau_{4,r}$, where $\tau_{2,r}$, $\tilde \tau_{1,r}$ and $\tau_{4,r}$ are hitting times for $\gamma_{2,r}$, $\tilde \gamma_{1,r}$, and $\gamma_{4,r}$, respectively.

\begin{figure}
    \centering
    \footnotesize
    \begin{tikzpicture}

    \begin{axis}[
        axis lines=none,
        xmin=-.15, xmax=\xmaximum+.05,
        ymin=-.15, ymax=1.25*\errbutsmaller,
        zmax=12,
        legend style={at={(axis cs: \xmaximum,\ymaximum, 0)},anchor=north west},
        unit vector ratio=1 1,
        unit rescale keep size=false,
        view={0}{90},
        scale=2
    ]

    \fill[cpatternthree] (\errbutsmaller,0) rectangle (1-\errbutsmaller,\errbutsmaller);
    \addlegendimage{area legend, cpatternthree}

    \fill[cpatterntwo] (2*\errbutsmaller,0) rectangle (1-\errbutsmaller,.5*\errbutsmaller);
    \addlegendimage{area legend, cpatterntwo}

    \legend{$G_{r,r}$, $G'_{r,r'}$}

    \addplot3[mark=none] coordinates {(0,0,0)};

    \draw (0,\ymaximum) -- (0,0) -- (1,0) -- (1,\ymaximum);    

    \draw (\errbutsmaller-.00301,0) -- (\errbutsmaller-.00301,\errbutsmaller-0.00301);
    \node[anchor=east] at (axis cs: \errbutsmaller,0.5*\errbutsmaller) {$\gamma_2$};

    \draw (1-\errbutsmaller+.00301,0) -- (1-\errbutsmaller+.00301,\errbutsmaller);
    \node[anchor=west] at (axis cs: 1-\errbutsmaller,0.5*\errbutsmaller) {$\tilde{\gamma}_1$};

    \draw (\errbutsmaller,\errbutsmaller+.00301) -- (1-\errbutsmaller,\errbutsmaller+.00301);
    \node[anchor=south] at (axis cs: 0.5,\errbutsmaller) {$\gamma_4$};

    \node[anchor=north] at (axis cs: 0,0) {$0$};
    \node[anchor=north] at (axis cs: \errbutsmaller,0) {$r$};
    \node[anchor=north] at (axis cs: 2*\errbutsmaller,0) {$2 r$};
    \node[anchor=north] at (axis cs: 1-\errbutsmaller,0) {$1-r$};
    \node[anchor=north] at (axis cs: 1,0) {$1$};
    \node[anchor=east] at (axis cs: 0,\errbutsmaller) {$r$};
    \node[anchor=east] at (axis cs: 0,0.5*\errbutsmaller) {$r'$};

    \swsaddlesmall
    \swsaddletildesmall

    \end{axis}
    \end{tikzpicture}
    \caption{The setup for Proposition~\ref{lem:exit-from-G}. Some subscripts have been omitted for readability.}
    \label{fig:FG}
\end{figure}
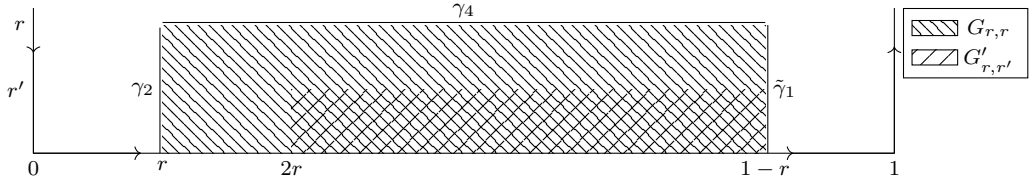

\begin{lemma}
\label{lem:proto-exit-from-G-time}
    For every $r>0$ small enough, there are $c_1, c_2 > 0$ such that 
    \[ \pp^x \{ \exitg > t \} \le c_1 \Exp{- c_2 t} \]
    for all $x \in G_{r,r}$ and all $t > 0$.
\end{lemma}

\begin{lemma}
\label{lem:proto-exit-from-G-prob-easier}
    For every $r>0$, there exists $q>0$ such that, if $r' \in (0,r)$ is small enough, then $\pp^x\{\exitg=\tilde\tau_{1,r}\}\ge q$ for all $x\in G'_{r,r'}$.
\end{lemma}

\begin{proposition}
\label{lem:exit-from-G} 
    For every $r>0$ small enough, the following holds w.h.p.\ under $\pp^{(x_1,\Exp{y_2})}$ as $y_2\to-\infty$, uniformly in $x_1\in[r,1-r]$:
    \[
        \exitg = \tilde\tau_{1,r}\wedge \tau_{2,r} \le |y_2|^{1/2}
        \quad\text{\rm and}\quad 
        Y_2 (\exitg) \in \big[y_2 - | y_2
        |^{1 / 2}, y_2 + |  y_2 | ^{1 / 2}\big]
        . 
    \]
\end{proposition}

Lemma~\ref{lem:proto-exit-from-G-time} and Proposition~\ref{lem:exit-from-G} say that the exit time from $G_{r,r}$ has exponential tails and single out a typical scenario in which the process exits towards one of the two adjacent corners and not into the bulk, spending a short time near an edge but away from the corners, leading to small changes in the distance to the boundary compared to what happens when traversing a corner, i.e., compared to case~(a) in Proposition~\ref{prop:corner}. 
The proofs, which are postponed to Section~\ref{ssec:proof-heteroc-aux}, do not use any assumptions on the local behavior of the drift near any vertices, so the vertices do not have to be saddles for this result to hold true.

\subsection{Integrating the behavior near a vertex and an incident edge}
\label{ssec:corner-with-edge}

In this section, we state the result combining our studies of  behaviors near one saddle and between two saddles.
In the next section, this central result will be applied iteratively to a sequence of saddles to describe the scenario of cycling near the boundary accompanied by asymptotic absorption.

We need more notation: let us define $\rotat$ to be the rotation of the square sending $\Or^k$ to~$\Or^{k+1}$, and 
\begin{align}
\label{eq:F_r}
\begin{aligned}
  U_r&=U_{r,r}
  &\qquad\qquad& & 
  \tilde U_r & = \rotat U_r =(1-r,r)\times(0,r),\\
  F_r & = (0,1)\times(0,r),
  &\qquad\qquad& & 
  \tilde\gamma_{2,r} & =\rotat\gamma_{2,r}=(1-r,1)\times\{r\},
\end{aligned}
\end{align}
see Figure~\ref{fig:FGU}. We are mostly interested analyzing the transition from $\gamma_{2,r}$ to~$\tilde{\gamma}_{2,r}$ (suitable for iteration),
but will consider a more general setup in our statement.
We will also use $\tilde {\Or}=\rotat \Or=\Or^1$,

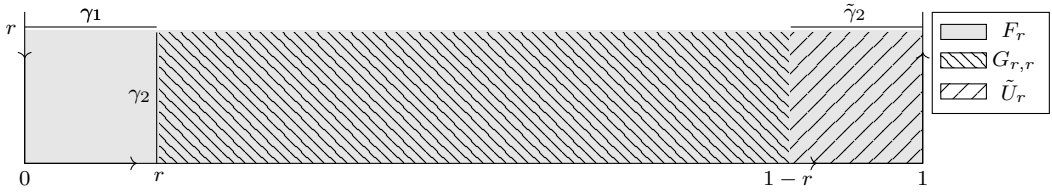
\begin{figure}
    \centering 
    \footnotesize
    \begin{tikzpicture}

    \begin{axis}[
        axis lines=none,
        xmin=-.15, xmax=\xmaximum,
        ymin=-.15, ymax=1.25*\errbutsmaller,
        zmax=12,
        legend style={at={(axis cs: \xmaximum,\ymaximum, 0)},anchor=north west},
        unit vector ratio=1 1,
        unit rescale keep size=false,
        view={0}{90},
        scale=2.
    ]

    \fill[cpatternone] (0,0) rectangle (1,\errbutsmaller);
    \addlegendimage{area legend, cpatternone}

    \fill[cpatternthree] (\errbutsmaller,0) rectangle (1-\errbutsmaller,\errbutsmaller-0.00301);
    \addlegendimage{area legend, cpatternthree}

    \fill[cpatterntwo] (1-\errbutsmaller+0.00301,0) rectangle (1,\errbutsmaller-0.00301);
    \addlegendimage{area legend, cpatterntwo}

    \legend{$F_r$, $G_{r,r}$, $\tilde{U}_r$}
    
    \addplot3[mark=none] coordinates {(0,0,0)};

    \draw (0,\ymaximum) -- (0,0) -- (1,0) -- (1,\ymaximum);    

    \draw (0,\errbutsmaller+0.00301) -- (\errbutsmaller-0.00301,\errbutsmaller+0.00301);
    \node[anchor=south] at (axis cs: 0.5*\errbutsmaller,\errbutsmaller) {$\gamma_1$};

    \node[anchor=south] at (axis cs: 0.5*\errbutsmaller,\errbutsmaller) {$\gamma_1$};

    \draw (\errbutsmaller-0.00301,0) -- (\errbutsmaller-0.00301,\errbutsmaller-0.00301);
    \node[anchor=east] at (axis cs: \errbutsmaller,0.5*\errbutsmaller) {$\gamma_2$};

    \draw (1-\errbutsmaller+0.00301,\errbutsmaller+0.00301) -- (1,\errbutsmaller+0.00301);
    \node[anchor=south] at (1-0.5*\errbutsmaller,\errbutsmaller) {$\tilde{\gamma}_2$};

    \node[anchor=north] at (axis cs: 0,0) {$0$};
    \node[anchor=north] at (axis cs: \errbutsmaller,0) {$r$};
    \node[anchor=north] at (axis cs: 1-\errbutsmaller,0) {$1-r$};
    \node[anchor=north] at (axis cs: 1,0) {$1$};
    \node[anchor=east] at (axis cs: 0,\errbutsmaller) {$r$};

    \swsaddlesmall
    \swsaddletildesmall

    \end{axis}
    \end{tikzpicture}
    \caption{Visual reference for Proposition~\ref{lem:highprob-events}.  Some subscripts have been omitted for readability.}
    \label{fig:FGU}
\end{figure}

\begin{proposition}
\label{lem:highprob-events}
    Suppose that  vertices~$\Or$ and $\tilde{\Or}$ are positively oriented saddles, and let 
    \begin{eqnarray*}
        \eta_{r}  = \inf \{ t \ge  0 : X(t)  \in \tilde\gamma_{2,r} \}.
    \end{eqnarray*} 
    Then, for every~$\eps > 0$ and $r_0>0$, there exists $r\in(0,r_0)$ such that the following hold w.h.p.\ under $\pp^{(x_1,\Exp{y_2})}$ as $y_2\to\infty$, uniformly over $x_1\in [r,1-r]$:
    $X(t) \in F_r$ for all $t \in [0,\eta_r)$, and the approximations
    \begin{align*}
        \int^{\eta_r}_0 \one_{\{ X (t) \notin \tilde U_r
        \}} \dd t &\approx 0,
        & 
        \int^{\eta_r}_0 \one_{\{ X (t) \in \tilde U_r\}} \dd t &\approx \frac{|y_2|}{\eivalplustilde},
        & 
        \ln \tilde X_2 (\eta_r) &\approx \rho y_2 
    \end{align*}
    hold with absolute error at most~$\eps|y_2|$.
\end{proposition}

This proposition shows that once the process $X$ is in $G_{r,r}$, it typically spends not too much time outside of $\tilde U_{r}$ before spending a long  time in~$\tilde U_{r}$ and exiting from it along the next edge. In addition, the logarithmic distances from the changes approximately by a factor of $\rho$. 
The proof of this proposition, which combines all previous results from this section, is given in Section~\ref{sec:proof-combined-saddle-edge}.

\subsection{Proof of  Theorem~\ref{thm:2d-cycle}}

The results below rely on Proposition~\ref{lem:highprob-events} to exhaust all claims in Theorem~\ref{thm:2d-cycle} and Proposition~\ref{prop:quadri}.
We will need several auxiliary coordinate systems or charts. One way to introduce them is to start with $(x^0_1,x^0_2) = (x_1,x_2)$ and iteratively define rotations $(x_1^{k}, x_2^{k})= (x_2^{k-1},1-x_1^{k-1})$ for $k\in\N$. Due to cyclicity, it often makes sense to use $x^k$ restricting $k$ to $\zz_4$=\{0,1,2,3\}. Viewing~$x^k$ as a function of $x\in\X$, we define the representation of our diffusion in the $k$-th chart by $X^k(t)=x^k(X(t))$. Then, $Y^k(t)$ is defined as in Section~\ref{sec:coord-prelim}. In particular, close enough to $\Ed^k$, we have $Y_2^k(t) = \ln X_2^k(t)$.
The assumption of positive orientation of the stochastic cycle means that, for each $k\in\{0,1,2,3\}$,  one can choose vectors $(1,0)$ and $(0,1)$ in coordinates $(x^k_1,x^k_2)$ as eigenvectors of the linearization of~$b$ near~$\Or^k$ associated with eigenvalues~$\eivalplus^k > 0$ and~$\eivalmoins^k < 0$, respectively.
We also recall that the stability index~$\Index$ is defined in~\eqref{eq:product-of-rhos} as the product of saddle stability indices~$\rho^k$ defined in~\eqref{eq:rhos}, and assumed to satisfy
\begin{equation}
\label{eq:stabiity-assumption1}
    \Index  > 1.
\end{equation}

Some statements rely on sequences of stopping times defined recursively as follows:
\begin{subequations}    
\label{eq:def-eta-heteroclinic}
\begin{align}
    \eta_{0,r,r'}
    &= \inf \{ t > 0 : X(t) \in F^*_{r'}\}, \\
    \eta_{n,r,r'} ,
    &= \inf \{ t > \eta_{n-1,r,r'} : \text{ there exists } k \text{ such that } X(t) \in \gamma_{2,r}^k \text{ and } X(\eta_{n-1,r,r'})\notin \gamma_{2,r}^k\}
\end{align}
\end{subequations}
for $n\in\nn$. Here, $F^*_{r'}$ is used as a shorthand for $\bigcup_{k=0}^3 F_{r'}^k$ where $F_{r'}^k = \rotat^k F_{r'}$.

\begin{lemma}
\label{lem:proto-eventual-cycling}
    For every initial condition $x \in \bXo$ and small enough parameters $0 < r' < r $, the stopping time $\eta_{n,r,r'}$ is almost surely finite for every~$n\in\nn$.
\end{lemma}
We give the proof of this lemma in Section~\ref{sec:proof-heteroc}.
It allows us to conclude that there is a full-measure set
on which we can introduce an infinite sequence $(K_{n,r})_{n\in\nn}$ of random variables so that
\begin{align*}
    K_{n,r,r'} &= k  \qquad\Leftrightarrow\qquad  
    X(\eta_{n,r,r'}) \in \gamma_{2,r}^k.
\end{align*}
The following proposition relies on iterated applications of Proposition~\ref{lem:highprob-events} in order to yield an almost sure statement.

\begin{proposition}
\label{prop:eventual-cycling}
    Suppose that there are positively oriented saddles at all vertices and that~\eqref{eq:stabiity-assumption1} holds.
    For every $x \in \bXo$, $\eps > 0$ and $r_0 \in (0,\tfrac 12)$, there exist $r$ and $r'$ such that $0 < r' < r < r_0$ and  
    \begin{align*}
        \pp^x\left(\bigcup_{N=1}^\infty \bigcap_{n=N}^\infty E_{n,r,r',\eps} \right) = 1,
    \end{align*}
    where\,---\,with some subscripts omitted for readability\,---
    \begin{align*}
    E_{n,r,r',\eps} 
        &= \left\{K_{n+1} = K_{n} + 1  \pmod{4}\right\} 
        \cap \left\{ \max_{t \in [\eta_n, \eta_{n+1}]} Y_2^{K_n}(t) = \ln r \right\}
        \\
        &\qquad \cap \left\{ \left|\int_{\eta_{n}}^{\eta_{n+1}} \one_{\{ X (t) \notin U_r^{K_{n}+1}
        \}} \dd t\right| < \eps |Y^{K_{n}}_2(\eta_{n})|\right\} \\
        &\qquad \cap \left\{\left|\int_{\eta_{n}}^{\eta_{n+1}} \one_{\{ X (t) \in U_r^{K_{n}+1}
        \}} \dd t - \frac{| Y^{K_{n}}_2(\eta_{n}) |}{\eivalplus^{K_{n}+1}} \right| <  \eps | Y^{K_{n}}_2(\eta_{n}) | 
        \right\} \\
        &\qquad \cap \left\{\left|Y_2^{K_{n}+1}(\eta_{n+1}) - \rho^{K_{n+1}} Y^{K_{n}}_2(\eta_{n})\right| < \eps |Y^{K_{n}}_2(\eta_{n})|\right\}.
    \end{align*}
\end{proposition}

\begin{proof}
    Fix $\eps > 0$ and $r_0\in(0,\tfrac 12)$, let $r \in (0,r_0)$ be as in Proposition~\ref{lem:highprob-events} and let $r' \in (0,r)$ be arbitrary for the time being. 
    Due to monotonicity of $E_{n,r,r',\eps}$ with respect to $\eps$, there is no loss of generality in assuming that $\eps$ is small enough to ensure the following perturbation of the condition~\eqref{eq:stabiity-assumption1}:
    \begin{equation}
    \label{eq:perturbed-cycle-stability}
        \prod_{k=0}^3 (\rho^k - \eps) > 1.
    \end{equation}
    With $R>0$, let
    \begin{align*}
        A_{0,r,r',\eps}
        &=
        \left\{
            Y_2^{K_{1,r,r'}}(\eta_{1,r,r'}) \leq 
            R \ln r'
        \right\}.
    \end{align*}
    The arguments used in the proofs of Lemma~\ref{lem:proto-corner-prob} and~\ref{lem:proto-exit-from-G-prob-easier} can be adapted to show that, for a suitably chosen~$R>0$, there exists~$p > 0$ such that $\pp^x(A_{0,r,r',\eps}) \geq p$ for all $r'$ small enough, independently of $x \in \bXo$.
    
    In view of~\eqref{eq:perturbed-cycle-stability} and the definition of~$A_{0,r,r',\eps}$ and $E_{n,r,r',\eps}$, there exists a positive sequence $(a_n)_{n\in\nn}$, also independent of~$x$ and $r'$, such that 
    \begin{equation}
    \label{eq:growth-an}
        \liminf_{n\to\infty} \frac{\ln a_n}{n} > 0    
    \end{equation}
    and
    \begin{align}
    \label{eq:a-priori-bound-on-An}
        Y_2^{K_n,r}(\eta_{n,r,r'}) &\leq  a_n \ln r' &&\text{on}& 
        A_{n,r,r',\eps} = A_{0,r,r',\eps} \cap E_{1,r,r',\eps} \cap \dotsb \cap E_{n,r,r',\eps}
    \end{align}
    for all $n\in\nn$.
    Combining~\eqref{eq:a-priori-bound-on-An} and the fact that $r$ is chosen as in Proposition~\ref{lem:highprob-events}, we deduce that
    \begin{align*}
        \pp^x(A_{n,r,r',\eps}|A_{n-1,r,r',\eps})
        &\geq 1 - c_1 \Exp{-c_2 |a_n|^{c_3} |{\ln r'}|^{c_3}}
    \end{align*}
    for all $n$, where the right-hand side is positive for $r'$ small enough, independently of~$x$.
    Therefore, for $r'$ small enough, we have
    \[ 
        \pp^x(A_{N,r,r',\eps}) \geq p \prod_{n=1}^N (1 - c_1 \Exp{-c_2 |a_n|^{c_3} |{\ln r'}|^{c_3}})
    \]
    for all $N\in\nn$ and~$x\in\bXo$.
    Since the product on the right-hand side is convergent as $N\to\infty$ thanks to~\eqref{eq:growth-an}, we conclude that there exists $\delta > 0$, independent of~$x$ and $r'$ small enough, such that 
    \[ 
        \pp^x\left( \bigcap_{n=1}^\infty E_{n,r,r',\eps}\right) 
        \geq 
        \lim_{N\to\infty} \pp^x\left(  A_{N,r,r',\eps}\right)
        \geq \delta.
    \]
    By the strong Markov property, this in turn implies
    \begin{align*}
        \ee^x\left[\one_{\bigcup_{M=1}^\infty \bigcap_{m=M}^\infty  E_{m,r,r',\eps}} \middle| \cF_{\eta_{n,r,r'}}\right] &= \ee^x\left[\one_{\bigcup_{M=n}^\infty \bigcap_{m=M}^\infty  E_{m,r,r',\eps}} \middle| \cF_{\eta_{n,r,r'}}\right] \\
        &\geq \inf_{\tilde{x}\in\bXo} \pp^{\tilde{x}}\left( \bigcap_{m=1}^\infty E_{m,r,r',\eps}\right) \\
        &\geq \delta
    \end{align*}
    for all $n$. Thus, we can conclude using Lemma~\ref{lem:levy-conseq}.
\end{proof}

With this proposition at hand, we can complete the proof of Theorem~\ref{thm:2d-cycle}.

\begin{proof}[Proof of Theorem~\ref{thm:2d-cycle}]
    First, we note that for every $x \in \bXo$,
    with probability $1$,
    $
       d(X(t),\partial\bX) \to 0
    $ as  $t\to\infty$. This   follows  
    since $E_{n,r,r',\eps} \subset \big\{\max_{t \in [\eta_n, \eta_{n+1}]} d(X(t),\partial\bX) \leq r_0\big\}$, and we can take $r_0\to 0$ in Proposition~\ref{prop:eventual-cycling}.
    In addition, Proposition~\ref{prop:eventual-cycling} implies that the diffusion makes transitions along the edges cyclically, thus every point on each edge is a limit point, and part~\ref{it:conv-to-bdry} is proved. 
    
    To prove parts~\ref{it:mostly-corners} and \ref{it:recent-corner-is-essential}, we 
    fix $x \in \bXo$ and $r_0$ and let $\eps > 0$ be arbitrary. In view of Proposition~\ref{prop:eventual-cycling}, there are $0 < r' < r < r_0$ for which there is a measurable partition $\Omega_{0,r,r', \eps}, \dotsc, \Omega_{3,r,r', \eps}$ of the full-measure set $\bigcup_{m=1}^\infty \bigcap_{n=m}^\infty E_{n,r,r',\eps}$, with the property that, almost surely on~$\Omega_{\ell,r,r', \eps}$, we have $K_{4n+\ell,r,r'} = 0$ for all $n$ large enough.
    For each $k$, almost surely on $\Omega_{\ell,r,r', \eps}$, we have that (with some subscripts omitted for readability) 
    $|Y_2^{4n+\ell+k}(\eta_{4n+\ell+k})|$ grows geometrically in~$n$ and so does
    $\eta_{4n+\ell+k+1} - \eta_{4n+\ell+k}$, and in addition we have
    \begin{align*}
       \liminf_{n\to\infty} \frac{\int_{\eta_{4n+\ell+k}}^{\eta_{4n+\ell+k+1}}  \one_{\{ X (t) \in U_{r_0}^{k+1}\}} \dd t}{\eta_{4n+\ell+k+1} - \eta_{4n+\ell+k}} > 1 - O(\eps)
    \end{align*}
    and 
    \begin{align*}
        \liminf_{n\to\infty} \frac{\int_{\eta_{4(n-1)+\ell}}^{\eta_{4n+\ell}}  \one_{\{ X (t) \in U_{r_0}^*\}} \dd t}{\eta_{4n+\ell} - \eta_{4(n-1)+\ell}} 
        &\geq 
        \sum_{k=-3}^0 \liminf_{n\to\infty}
        \frac{\int_{\eta_{4n+\ell+k}}^{\eta_{4n+\ell+k+1}}  \one_{\{ X (t) \in U_{r_0}^{k+1}\}} \dd t}{\eta_{4n+\ell+1} - \eta_{4(n-1)+\ell+1}} 
        > 1 - O(\eps).
    \end{align*}
    Considering all $\ell$ and taking $\eps \to 0$ we complete the proof.
\end{proof}

\subsection{Proof of Proposition~\ref{prop:quadri}}
We already know that every limit point of $(\empi{x}_t)_{t > 0}$ (at least one exists by compactness) must be invariant, and thus, in the present situation, a convex combination of $\delta_{\Or^0}$, $\delta_{\Or^1}$, $\delta_{\Or^2}$ and $\delta_{\Or^3}$. Proposition~\ref{prop:eventual-cycling} allows us to be more precise.

Let us define
\begin{align*}
    f_{k,j}
    = 
    \frac{\prod_{\ell = k-3}^{j}\rho^{\ell}}
    {\eivalplus^{j}}>0,
\end{align*}
where indices $k \in \{0,1,2,3\}$ and $j = \{k-3,k-2,k-1,k\}$ are understood modulo 4.

\begin{corollary}
\label{cor:empirical-limit-points-cycle}
    Under the assumptions of Proposition~\ref{prop:eventual-cycling},  
 for every $x\in\bXo$, almost surely, the measure
    \begin{equation}
    \label{eq:empirical-limit-points-cycle}
        \overline{\mu}_k = \frac{\sum_{j=k-3}^{k} f_{k,j} \delta_{\Or^j}}{\sum_{j'=k-3}^{k} f_{k,j'}} 
    \end{equation}
    is a limit point of $(\empi{x}_t)_{t >0}$. 
\end{corollary}

\begin{proof}
    Fix $k \in \{0,1,2,3\}$. Now, fix $x \in \bXo$ and let $\eps > 0$ and $r_0 > 0$ be arbitrary. We use the same partition of a  full-measure set using Proposition~\ref{prop:eventual-cycling} as in the proof of Theorem~\ref{thm:2d-cycle} at the end of the preceding subsection. 
    We first restrict our attention to the first element of this partition. To wit, on~$\Omega_{0,r,r',\eps}$, the intervals 
    \begin{align*}
        I_{n,k} = \left[{\eta_{4n+k}}, {\eta_{4(n+1)+k}}\right]
    \end{align*}
    are such that, with some subscripts omitted for readability, 
    $\liminf_{n\to\infty} |I_{n,k}| > 0$ and
    \begin{align*}
      \frac{|Y_2^{k}(\eta_{4n+k})|(f_{k,j} - O(\eps))}{\sum_{j'= k-3}^{k} |Y_2^{k}(\eta_{4n+k})|(f_{k,j'} + O(\eps))}  
        &< \liminf_{n\to\infty} \frac{\int_{I_{n,k}}  \one_{\{ X (t) \in U_{r_0}^{j}\}} \dd t}{|I_{n,k}|}  \\
        &\leq \limsup_{n\to\infty} \frac{\int_{I_{n,k}}  \one_{\{ X (t) \in U_{r_0}^{j}\}} \dd t}{|I_{n,k}|} \\
        &< \frac{|Y_2^{k}(\eta_{4n+k})|(f_{k,j} + O(\eps))}{\sum_{j'= k-3}^{k} |Y_2^{k}(\eta_{4n+k})|(f_{k,j'} - O(\eps))}  
    \end{align*}
    for every~$j$, where the times estimates follow from the definition of sets $E_{n,r,r',\eps}$.
    Using the standard metrization of weak convergence of measures via tests against 1-Lipschitz functions, we conclude that, almost surely on~$\Omega_{0,r,r'\eps}$, at all but finitely many times of the form~$\eta_{4n+k}$, the empirical measure is $[O(\eps)+O(r_0)]$-close to $\overline{\mu}_k$.
    Since this is easily adapted to other $\Omega_{\ell,r,r',\eps}$ by merely shifting $k$ by $\ell$ in the appropriate places, the same is true on a full-measure set. And since $\eps$ and $r_0$ were arbitrary, we can conclude that, almost surely, the measure $\overline{\mu}_k$ is a limit point.
\end{proof}

\begin{corollary}
\label{cor:dyn-conseq-of-high-probab}
    Under the assumptions of Proposition~\ref{prop:eventual-cycling}, for every $x \in \bXo$ and $r>0$, the empirical measures $(\empi{x}_t)_{t > 0}$ almost surely fail to converge.
\end{corollary}

\begin{proof}
    Due to Lemma~\ref{cor:empirical-limit-points-cycle}, it suffices to pick any $k$ and show that $\overline{\mu}_k \neq \overline{\mu}_{k+1}$.
    Since 
    \begin{align*}
        f_{k+1,j} &= \frac{\prod_{\ell = k-2}^{j}  
        \rho^{\ell}}{\eivalplus^{j}}
        = \frac{\frac{1}{\rho^{k-3}} \prod_{\ell = k-3}^{j}  
        \rho^{\ell}}{\eivalplus^{j}}
        = \frac{1}{\rho^{k-3}} f_{k,j}
    \end{align*}
    for $j \neq k+1$, and 
    \begin{align*}
        f_{k+1,k+1} &= \frac{\prod_{\ell = k-2}^{k+1}  
        \rho^{\ell}}{\eivalplus^{k+1}}
        = \frac{ \frac{\prod_{\ell = k-2}^{k+1}  
        \rho^{\ell}}{\rho^{k-3}} \rho^{k-3}}{\eivalplus^{k+1}}
        = \frac{\Index}{\rho^{k-3}} f_{k,k-3}
    \end{align*}
    we have
    \begin{equation}
    \label{eq:increase-k-empi}
        \rho^{k-3}f_{k+1,j}
        = 
        \begin{cases}
            f_{k,j} & \text{for } j\neq k+1 \\
            \Index f_{k,j} & \text{for } j = k+1.
        \end{cases}
    \end{equation}
    Since $\Index > 1$, we obtain the desired relation $\overline{\mu}_k \neq \overline{\mu}_{k+1}$.
\end{proof}

The following corollary directly implies Proposition~\ref{prop:quadri} and provides a precise description of the set of limit points for empirical measures.

\begin{corollary}
    Under the assumptions of Proposition~\ref{prop:eventual-cycling}, for every $x \in \bXo$, the set of limit points of the empirical measures $(\empi{x}_t)_{t > 0}$ is the nonplanar quadrilateral cyclically connecting the measures  
    $(\overline{\mu}_k)_{k=0}^3$ given by \eqref{eq:empirical-limit-points-cycle}.
\end{corollary}

\begin{proof}
    We resume with the setup and notation of the proof of Corollary~\ref{cor:empirical-limit-points-cycle} for a fixed~$x\in\bXo$, which, we recall, relied on Proposition~\ref{prop:eventual-cycling}.
    However, we now use a partition into time intervals of the form 
    \begin{align*}
        I_{n,k}(s) = \left[{\eta_{4n+k}} + \frac{s}{\eta_{4n+k+1}-\eta_{4n+k}} ,{\eta_{4(n+1)+k}} + \frac{s}{\eta_{4(n+1)+k+1}- \eta_{4(n+1)+k}}\right].
    \end{align*}
    with $s \in [0,1)$. At fixed $k$ and $s$, varying $n$ induces a partition of~$[\eta_{5},\infty)$.
    Now, almost surely on $\Omega_{0,r,r',\eps}$,
    \begin{align*}
        \frac{f_{k,j} - O(\eps)}{\sum_{j'= 0}^{3} f_{k,j'} + s(\Index-1) + O(\eps)}
        &< \liminf_{n\to\infty} \frac{\int_{I_{n,k}(s)}  \one_{\{ X (t) \in U_{r_0}^{j}\}} \dd t}{|I_{n,k}(s)|}  \\
        &\leq \limsup_{n\to\infty} \frac{\int_{I_{n,k}(s)}  \one_{\{ X (t) \in U_{r_0}^{j}\}} \dd t}{|I_{n,k}(s)|} 
        < \frac{f_{k,j} + O(\eps)}{\sum_{j'= 0}^{3} f_{k,j'}+ s(\Index-1)- O(\eps)}
    \end{align*}
    for $j\neq k+1$,
    and
    \begin{align*}
    \frac{f_{k,k+1} + s(\Index-1)}{\sum_{j'= 0}^{3} f_{k,j'} + s(\Index-1)} - O(\eps) 
        &< \liminf_{n\to\infty} \frac{\int_{I_{n,k}(s)}  \one_{\{ X (t) \in U_{r_0}^{k+1}\}} \dd t}{|I_{n,k}(s)|}  \\
        &\leq \limsup_{n\to\infty} \frac{\int_{I_{n,k}(s)}  \one_{\{ X (t) \in U_{r_0}^{k+1}\}} \dd t}{|I_{n,k}(s)|} 
        < \frac{f_{k,k+1} + s(\Index-1)}{\sum_{j'= 0}^{3} f_{k,j'}+ s(\Index-1)} + O(\eps).
    \end{align*}
    In view of~\eqref{eq:increase-k-empi}, we conclude that, almost surely $\Omega_{0,r,r',\eps}$, for all but finitely many~$n$, independently of~$k$ and~$s$, the empirical measure at the end of $I_{n,k}(s)$ is $[O(\eps)+O(r_0)]$-close to a convex combination of $\overline{\mu}_k$ and $\overline{\mu}_{k+1}$ when tested against a 1-Lipschitz function. 
    As in the proof of Corollary~\ref{cor:empirical-limit-points-cycle}, we can use the fact that $r_0$ and $\eps$ are arbitrary and adapt the argument to other $\Omega_{\ell,r,r',\eps}$ in order to conclude that varying $s \in [0,1)$ indeed corresponds to generating all convex combinations of $\overline{\mu}_k$ and $\overline{\mu}_{k+1}$ as the limit points. Since every late enough time is the right endpoint of one and only one interval of the form $I_{n,k}(s)$ (varying $k$, $s$ and $n$), no other limit point is possible. Nonplanarity can be deduced from~\eqref{eq:increase-k-empi}.
\end{proof}

\subsection{Proof details}
\label{sec:proof-heteroc}

\label{ssec:proof-heteroc-aux}

\begin{proof}[Proof of Lemma~\ref{lem:proto-corner-prob}]
    Let us use the shorthand $\exit_r$ for the exit time from $U_{r,2r}$.
    A straightforward modification of Lemma~\ref{lem:Y}
    implies that, given any $\delta > 0$, the parameter $r$ can be taken
    sufficiently small that the following property holds: for every $x \in \gamma_{1,r}$, there are martingales~$M,N\in\cM$ such that
    \begin{align}
        (\eivalplus - \delta) t + M (t) &\le Y_1 (t) - y_1(x) \le (\eivalplus + \delta) t + M (t), 
        \label{eq:upper-Y1} \\
        (\eivalmoins - \delta) t + N (t) &\le Y_2 (t) - \ln r \le (\eivalmoins + \delta) t + N (t) 
        \label{eq:upper-Y2}
    \end{align}
    for all $t \leq \exit_r$. Throughout this proof, we will only consider $\delta < (|\eivalmoins|\wedge  \eivalplus)/2$. In particular, $\lambda_+ - \delta$ is still positive and $\lambda_- + \delta$ is still negative. Although the precise martingales depend on the initial condition~$x$, an upper bound on the growth of quadratic variation can be chosen uniformly in $x$.

    The lower bound in~\eqref{eq:upper-Y1} and Corollary~\ref{cor:finite-mean} imply that $\exit_r$ is almost surely finite. The upper bound in~\eqref{eq:upper-Y2} and Lemma~\ref{lem:drift-wins-over-mart-bound} imply that $\tau_{0,r} > \exit_r$ with a probability that has a positive lower bound depending only on~$\delta$ and the quadratic variation of~$N$.
\end{proof}

\begin{proof}[Proof of Proposition~\ref{prop:corner}]
    We continue with the setup and notation of the proof of Lemma~\ref{lem:proto-corner-prob}.
    If $r$ is small enough to ensure~\eqref{eq:upper-Y1}, then Lemmas~\ref{lem:lb-tauplus-prob} and~\ref{lem:ub-tauplus-prob} applied to $Y_1 - y_1$ show that the events
    \begin{align*}
        \exit_r < (1-\delta)\frac{\ln r - y_1}{\lambda_+ + \delta}
        \quad\text{and}\quad
        \exit_r > \frac{1}{1-\delta}\cdot\frac{\ln r - y_1}{\lambda_+ - \delta}
    \end{align*}
    are both rare conditioned on $\exit_r=\tau_{2,r}$. Hence, given $\eps>0$, both
    \begin{align*}
        \exit_r < \left(\frac{1}{\lambda_+} - \eps \right)|y_1|
        \quad\text{and}\quad
        \exit_r > \left(\frac{1}{\lambda_+} + \eps \right)|y_1|
    \end{align*}
    are rare conditioned on $\exit_r=\tau_{2,r}$, provided that~$r$ small enough. 
    Applying Theorem~\ref{thm:Bass-martingale} to the martingale $N$, we see that on the complement to these rare events,
    $$ 
        Y_2(\exit_r)
        < \ln r + (\lambda_- - \delta) \left(-\frac{1 +
        \delta}{\eivalplus} y_1\right) +\delta y_1
    $$
    and 
    $$ 
        Y_2(\exit_r)
        > \ln r + (\lambda_-+\delta) \left(-\frac{1 -
        \delta}{\eivalplus} y_1\right) - \delta y_1.
    $$
    Hence, given $\eps>0$, both
    $$ 
        Y_2(\exit_r) < \frac{-\eivalmoins}{\eivalplus} y_1 + \eps y_1
    \qquad\text{ and }\qquad
        Y_2(\exit_r) > \frac{-\eivalmoins}{\eivalplus} y_1 - \eps y_1
    $$
    are rare for all $r$ small enough. This proves the claims in~(a).

    Now note that, for every $T>0$, we have 
    \begin{align*}
        &\{\exit_r = \tau_{0,r} \text{ and } \exit_r > T \} \\
        & \qquad \subset \big\{ \ln r + (\eivalmoins +  \delta ) t + N
        (t) \ge \ln 2r \text{ for some } t > T \big\}
        \\
        & \qquad \subset \Big\{ (\eivalmoins +  \delta ) s + [N (T+s) - N(T)] \ge  \ln 2 - (\eivalmoins+\delta)T - N^*(T)  \text{ for some } s > 0 \Big\} \\
        & \qquad \subset \big\{N^*(T) > -\tfrac 12 (\eivalmoins+\delta)T \big\} \\
            &\qquad\qquad \cup \Big\{ (\eivalmoins +  \delta ) s + [N (T+s) - N(T)] \ge \ln2 - \tfrac12(\eivalmoins+\delta)T \text{ for some } s > 0 \Big\}
    \end{align*}
    The right-hand side decays exponentially in~$T$ (Theorem~\ref{thm:Bass-martingale} and Lemma~\ref{lem:drift-wins-over-mart-bound}), i.e., as a stretched exponential in $|y_1|$ if $T=|y_1|^{1/3}$.  Finally, in view of~\eqref{eq:upper-Y1}, if $\exit_r\leq|y_1|^{1/3}$, then $Y_1 (\exit_r) \notin [y_1 - | y_1 |^{1/2}, y_1 + | y_1 |^{1/2}]$ is rare (Theorem~\ref{thm:Bass-martingale} again). This proves the claims in~(b).
\end{proof}

\begin{proof}[Proof of Proposition~\ref{lem:2-->1,3}]
    Albeit with $\exit_r$ defined in a way that is different from that in Propositions~\ref{prop:corner}, we still have bounds of the form~\eqref{eq:upper-Y1}--\eqref{eq:upper-Y2} for $t \leq \exit_r$ and $r$ small enough.
    We can thus follow a strategy similar to that in the proof of Proposition~\ref{prop:corner}.
\end{proof}

\begin{proof}[Proof of Lemma~\ref{lem:proto-exit-from-G-time}]
    Thanks to Lemma~\ref{lem:Y} (and Remark~\ref{rem:strat-ito-combo}), we have a process of the form~\eqref{eq:alt-SDE}, and the proposed bound follows from Lemma~\ref{lem:finite-2d-mean}. 
\end{proof}

\begin{proof}[Proof of Lemma~\ref{lem:proto-exit-from-G-prob-easier}]
    Fix $r>0$.
    On the one hand, thanks to Lemmas~\ref{lem:Y} and~\ref{lem:lb-tauplus-prob} (applied to $Y_2$), 
    \[ 
    \lim_{r' \to 0} \sup_{x \in G'_{r,r'}} \pp^x\{\tau_{4,r} \leq 2 \} = 0.
    \]
    On the other hand, thanks to Lemmas~\ref{lem:Y} and~\ref{lem:positive-prob-to-escape-right} (applied to $Y_1$), there exists $c>0$ such that  
    \begin{align*}
        \inf_{x \in G'_{r,r'}} \pp^x\{ \tilde{\tau}_{1,r} \leq \tau_{2,r} \text{ and } \tilde{\tau}_{1,r} \leq 1 \} \geq c
    \end{align*}
    for all $r' \in (0,r)$. Hence, the proposed bound holds with $q = \tfrac 12 c$ if $r'$ is small enough.
\end{proof}

\begin{proof}[Proof of Proposition~\ref{lem:exit-from-G}]
    Let us show that, with $\exit_r$ as a shorthand for the exit time from $G_{r,r}$,
    we have $\{ \exit_r = \tau_{4,r} \}$ is a rare event under $\pp^{(x_1,\Exp{y_2})}$
    as $y_2\to-\infty$, uniformly in $x_1\in[r,1-r]$. We will in some places omit keeping explicit track of the dependence on $r$ to lighten the notation.
    By Lemma~\ref{lem:Y}, there is a constant $C > 0$
    and a martingale $M\in\cM$ such that
    \[ 
        |Y_2 (t) - y_2| \le Ct + M^* (t) 
    \]
    for $t<\exit$. Although the precise martingale depends on the initial condition $(x_1,\Exp{y_2})$, a constant $C$ and an upper bound on the growth of quadratic variation can be chosen uniformly in $(x_1,\Exp{y_2}) \in G$.
    We then see that
    \[ \{ \exit = \tau_4 \} \subset \{ \exit > t \} \cup \{ M^{\ast} (t) > \ln r - y_2 -
        C t \}. \]
    Choosing $t = | y_2 |^{1 / 3}$, we use Lemma~\ref{lem:proto-exit-from-G-time} to obtain
    \[ \pp^{(x_1,\Exp{y_2})} \{ \exit = \tau_4 \}\le c_1 \Exp{- c_2  | y_2 |^{1 / 3}} +
        \pp^{(x_1,\Exp{y_2})} \left\{ M^* (| y_2 |^{1 / 3}) > \ln R - y_2 - C | y_2 |^{1 /
        3}  \right\}, \]
    and the desired estimate follows from the exponential martingale inequality (Lemma~\ref{thm:Bass-martingale}).

    The fact that $\exit > |y_2|^{1/2}$ is rare under $\pp^{(x_1,\Exp{y_2})}$,  uniformly in $x_1\in[r,1-r]$, is a direct consequence of Lemma~\ref{lem:proto-exit-from-G-time}, so it remains only to show that 
    \[B(y_2) = \big\{ \exit  = \tilde\tau_{1} \wedge \tau_{2} \text{ and }|Y_2 (\exit)-y_2|>
    |  y_2 | ^{1 / 2} \big\}\]      
    is rare under $\pp^{(x_1,\Exp{y_2})}$,  uniformly in $x_1\in[r,1-r]$.
    To this end, note that, for arbitrary $t>0$,
    \begin{align*}
        \pp^{(x_1,\Exp{y_2})} (B(y_2)) \le \pp^{(x_1,\Exp{y_2})} \{ \exit > t \} +
        \pp^{(x_1,\Exp{y_2})} \left\{ Ct + M^* (t) \ge | y_2  |^{1 / 2} \right\}.
    \end{align*}
    Plugging in $t = | {y_2} |^{1 / 3}$, applying Lemma~\ref{lem:proto-exit-from-G-time} to the first term, and applying the exponential martingale inequality (Lemma~\ref{thm:Bass-martingale}) to the second term,
    we complete the proof.
\end{proof}

\label{sec:proof-combined-saddle-edge}
The typical scenario at the heart of the proof of Proposition~\ref{lem:highprob-events} is the following: starting in~$G_{r,r}$, the process may make several short transitions between the curves $\gamma_{2,r},\gamma_{3,r},\tilde\gamma_{0,r},\tilde\gamma_{1,r}$ before it makes a transition from $\tilde\gamma_{1,r}$ to $\tilde\gamma_{2,r}$; see Figure~\ref{fig:FGU-proof}. We know that with the exception of this last transition, the coordinate $y_2$ changes at most by $|y_2|^{1/2}$ over each transition. Before we delve into the proof of Proposition~\ref{lem:highprob-events}, we provide an estimate used to control the cumulative effect of those transitions. The proof is straightforward and therefore omitted. 

\begin{figure}
  \centering 
  \footnotesize
    \begin{tikzpicture}
    \begin{axis}[
        axis lines=none,
        xmin=-.15, xmax=\xmaximum,
        ymin=-.15, ymax=1.25*\errbutsmaller,
        zmax=12,
        legend style={at={(axis cs: \xmaximum,\ymaximum, 0)},anchor=north west},
        unit vector ratio=1 1,
        unit rescale keep size=false,
        view={0}{90},
        scale=2.
    ]

    \fill[cpatternone] (0,0) rectangle (1,\errbutsmaller);
    \addlegendimage{area legend, cpatternone}

    \fill[cpatternthree] (\errbutsmaller,0) rectangle (1-\errbutsmaller,\errbutsmaller-0.00301);
    \addlegendimage{area legend, cpatternthree}

    \legend{$F$, $G$}

    \addplot3[fill=black!25, opacity=.75, mark=none] 
    coordinates {(0,0,0)};

    \draw (0,\ymaximum) -- (0,0) -- (1,0) -- (1,\ymaximum);    

    \draw (0,\errbutsmaller+0.00301) -- (\errbutsmaller-0.00301,\errbutsmaller+0.00301);
    \node[anchor=south] at (axis cs: 0.5*\errbutsmaller,\errbutsmaller) {$\gamma_1$};

    \draw (\errbutsmaller,\errbutsmaller+0.00301) -- (1-\errbutsmaller,\errbutsmaller+0.00301);
    \node[anchor=south] at (axis cs: 0.5,\errbutsmaller) {$\gamma_4$};

    \node[anchor=south] at (axis cs: 0.5*\errbutsmaller,\errbutsmaller) {$\gamma_1$};

    \draw (\errbutsmaller-0.00301,0) -- (\errbutsmaller-0.00301,\errbutsmaller-0.00301);
    \node[anchor=east] at (axis cs: \errbutsmaller,0.5*\errbutsmaller) {$\gamma_2$};

    \draw (1-\errbutsmaller+0.00301,\errbutsmaller+0.00301) -- (1,\errbutsmaller+0.00301);
    \node[anchor=south] at (1-0.5*\errbutsmaller,\errbutsmaller) {$\tilde{\gamma}_2$};

    \node[circle, inner sep=0pt, anchor=west, fill=black!10, opacity=.6] at (axis cs: 2.05*\errbutsmaller,0.5*\errbutsmaller) {\phantom{$\gamma_3$}};
    \node[anchor=west] at (axis cs: 2*\errbutsmaller,0.5*\errbutsmaller) {$\gamma_3$};
    \draw (2*\errbutsmaller+.00301,0) -- (2*\errbutsmaller+.00301,\errbutsmaller-0.00301);

    \node[circle, inner sep=0pt, anchor=west, fill=black!10, opacity=.6] at (axis cs: 1-1.95*\errbutsmaller,0.5*\errbutsmaller) {\phantom{${\gamma}_0$}};
    \node[anchor=west] at (axis cs: 1-2*\errbutsmaller,0.5*\errbutsmaller) {$\tilde{\gamma}_0$};
    \draw (1-2*\errbutsmaller+.00301,0) -- (1-2*\errbutsmaller+.00301,\errbutsmaller-.00301);

    \draw (1-\errbutsmaller+.00301,0) -- (1-\errbutsmaller+.00301,\errbutsmaller-0.00301);
    \node[anchor=west] at (axis cs: 1-\errbutsmaller,0.5*\errbutsmaller) {$\tilde{\gamma}_1$};

    \node[anchor=north] at (axis cs: 0,0) {$0$};
    \node[anchor=north] at (axis cs: \errbutsmaller,0) {$r$};
    \node[anchor=north] at (axis cs: 2*\errbutsmaller,0) {$2 r$};
    \node[anchor=north] at (axis cs: 1-2*\errbutsmaller,0) {$1-2r$};
    \node[anchor=north] at (axis cs: 1-\errbutsmaller,0) {$1-r$};
    \node[anchor=north] at (axis cs: 1,0) {$1$};
    \node[anchor=east] at (axis cs: 0,\errbutsmaller) {$r$};

    \swsaddlesmall
    \swsaddletildesmall
    
    \end{axis}
    \end{tikzpicture}
    \caption{The setup for the proof of Proposition~\ref{lem:highprob-events}. }
    \label{fig:FGU-proof}
\end{figure}
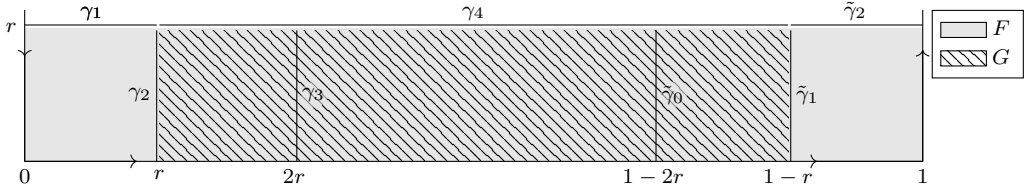

\begin{lemma}
\label{lem:cumulative-displacement}
    With $f^n$ denoting the $n$-th iteration of a function $f$, i.e.\ $f^n (y) = f \circ f \circ \cdots \circ f (y)$, and with 
    \begin{align*}
        f : (0, + \infty) &\rightarrow (0, + \infty) 
        & 
        g : (0, + \infty) &\rightarrow \mathbb{R} 
        \\
        y &\mapsto y + y^{1 / 2} 
        & 
        y &\mapsto y - y^{1 / 2},
    \end{align*}
    we have
    \begin{enumerate}
        \item There is $y_0 > 0$ such that for all $y > y_0$
        and for all $n \in \mathbb{N}$, 
        \[ f^n (y) \le y + n^2 y^{2 / 3} . \]

        \item There is \ $y_0 > 0$ such that for all $y > y_0$ and all $n \in \mathbb{N}$ satisfying $n < y^{1 / 6} - 1$, $g^n$ is
        well-defined and
        \[ g^n (y) \ge y - n^2 y^{2 / 3}. \]
    \end{enumerate}
\end{lemma}

\begin{proof}[Proof of Proposition~\ref{lem:highprob-events}]
    Let $\eps > 0$ and $r_0 \in (0,\tfrac 12)$ be arbitrary and let $r \in (0,r_0)$ be small enough for applications of Proposition~\ref{lem:2-->1,3} near~$\Or$, Lemmas~\ref{lem:proto-exit-from-G-time} and~\ref{lem:proto-exit-from-G-prob-easier} and Proposition~\ref{lem:exit-from-G} between~$\Or$ and~$\tilde{\Or}$, and Lemma~\ref{lem:proto-corner-prob} and Proposition~\ref{prop:corner} near~$\tilde{\Or}$. 
    In order to lighten the notation, we omit references to~$r$ for the rest of the argument.
    We define 
    \begin{align*}
    \Gamma &= \gamma_1\cup\gamma_2\cup \gamma_3 \cup \gamma_4\cup \tilde\gamma_0\cup\tilde\gamma_1\cup\tilde\gamma_2
    &&\text{and}&
    \Gamma' &= \gamma_2\cup\gamma_3\cup\tilde\gamma_0\cup \tilde\gamma_1,
    \end{align*}
    see Figure~\ref{fig:FGU-proof}.
    We also define the following stopping times: $\chi_0 = 0$, and then,
    inductively, 
    \begin{align*}
        \chi_n = \inf\{t > \chi_{n-1} : &{} \text{ there exists } \gamma \in \{\gamma_1,\gamma_2,\gamma_3,\gamma_4, \tilde\gamma_0,\tilde\gamma_1,\tilde\gamma_2\} \\
        &{} \text{ such that } X(t) \in \gamma \text{ and } X(\chi_{n-1})\notin \gamma\}.
    \end{align*}
    We let $\chi = \inf \{ t > 0 : X (t) \in \Gamma \setminus \Gamma' \}$ and introduce the random variable $\nu=\min\{n: \chi_{n} = \chi\}$.
    These stopping times describe the
    following high-probability scenario: starting in $G$,
    the process visits intermittently different curves constituting~$\Gamma'$ at times $\chi_1, \chi_2, \dotsc$ and, eventually, upon a visit to~$\tilde\gamma_1$, exits the
    rectangle~$F$ through $\tilde\gamma_2$ at time~$\chi = \chi_\nu$. 

    A transition from $\gamma_2$
    to $\gamma_1 \cup\gamma_4$ is a rare event due to Proposition~\ref{lem:2-->1,3};
    a transition from $\gamma_3\cup\tilde\gamma_0$ to $\gamma_4$, due to  Proposition~\ref{lem:exit-from-G}; a transition from $\tilde\gamma_1$ to $\gamma_4$, due to Proposition~\ref{prop:corner}. 
    Each of the other types of transition between the curves of $\Gamma'$
    rarely takes long due to Propositions~\ref{lem:2-->1,3},~\ref{lem:exit-from-G} and~\ref{prop:corner}, and Lemma~\ref{lem:proto-exit-from-G-prob-easier} implies an
    exponential estimate on the tail of the number of those transitions. Propositions~\ref{lem:2-->1,3},~\ref{lem:exit-from-G} and~\ref{prop:corner} also guarantee that  
    the distance to the boundary does not change much during each of those
    transitions, and the cumulative change of that distance is estimated by Lemma~\ref{lem:cumulative-displacement}.
    The terminal transition from~$\tilde\gamma_1$ to~$\tilde\gamma_2$ is described by case~(a) of Proposition~\ref{prop:corner}.

    Let us be slightly more precise. 
    Using Propositions~\ref{prop:corner}, \ref{lem:2-->1,3},~\ref{lem:exit-from-G} and Lemma~\ref{lem:cumulative-displacement}, we can show
    inductively that the following happens w.h.p.\ as $y_2\to-\infty$: for each $n < |y_2|^{1 / 8}\wedge \nu$,
    \begin{subequations}
    \label{eq:before-chi_nu}
    \begin{gather}
        |Y_2 (\chi_n) - y_{2} |
        \le
        n^2 | y_2 |^{2 / 3} < | y_2 |^{11 / 12},
        \intertext{and}
        \chi_n 
        \le n (|y_2|+n^2|y_2|^{2/3})^{1/2}\le 2 |y_2|^{5/8}\le |y_2|^{2/3}.
    \end{gather}
    \end{subequations}
    However, using~\eqref{eq:before-chi_nu}, Lemmas~\ref{lem:proto-corner-prob} and~\ref{lem:proto-exit-from-G-prob-easier},
    we obtain that $\nu \le |y_2|^{1/8}$ w.h.p.\ as $y_2 \to -\infty$. 
    Therefore, the bounds~\eqref{eq:before-chi_nu} imply that
    $|Y_2 (\chi_{\nu-1})
    - y_2 | < | y_2 |^{11 / 12}$
    and 
    $\chi_{\nu-1} < | y_2 | ^{2/3}$
    w.h.p. Finally, this last transition from $ \tilde\gamma_1$ to $\tilde\gamma_2$ is described by Proposition~\ref{prop:corner} near $\tilde{\Or}$, keeping in mind that $\tilde{y}_1 = y_2$: 
    \begin{align*}
        \chi_\nu - \chi_{\nu-1} &\in \left[ \left(\frac{1}{\eivalplus} - \eps\right) (|y_2|-| y_2 |^{11 / 12}), \left(\frac{1}{\eivalplus} + \eps\right) (|y_2|+| y_2 |^{11 / 12}) \right], \\
        \tilde{Y}_2 (\chi_\nu) &\in \left[(\rho + \eps)
        (y_2-| y_2 |^{11 / 12}), (\rho - \eps) (y_2 + | y_2 |^{11 / 12})\right].
    \end{align*}
    The conclusion of Proposition~\ref{lem:highprob-events} follows, albeit with a slightly larger~$\eps$ to accommodate the lower powers of $|y_2|$.
\end{proof}

\begin{proof}[Proof of Lemma~\ref{lem:proto-eventual-cycling}]
    First, we note that since {$\bXo\setminus F^*_{r'}$ is  compact}, Lemma~\ref{lem:exit-compact} guarantees that $\eta_{0,r,r'}$ is almost surely finite.
    More generally, the stopping times $\zeta_{0,r'} = \eta_{0,r,r'}$ and 
    $$ 
        \zeta_{m,r'} = \inf\{t > \zeta_{m-1,r'} + 1 : X(t) \in F^*_{r'}\}
    $$
    are diverging as $m\to \infty$ but all almost surely finite.
    In particular, the event $A = \{\eta_{1,r,r'} < \infty\}$ is measurable with respect to the $\sigma$-algebra generated by $(\cF_{\zeta_{m,r'}})_{m\in\nn}$.
    By the strong Markov property and Lemma~\ref{lem:proto-corner-prob} and Proposition~\ref{lem:highprob-events}, provided that $r$ and $r'$ are small enough, there exists $\delta > 0$ such that
    $
        \ee^x[\one_A|\cF_{\zeta_{m,r'}}] \geq \delta
    $
    for all $m$, independently of~$x$.
    Hence, we can appeal to Lemma~\ref{lem:levy-conseq} to deduce that $\pp^x(A) = 1$, i.e., that $\eta_{1,r,r'}$ is almost surely finite. A simple induction argument using the strong Markov property then shows that~$\eta_{n,r,r'}$ is almost surely finite for every $n$.
\end{proof}

\section{A General hitting-time estimate}
\label{sec:YM}

\subsection{The Setting and main result}
\label{sapp:YM-setting}

In this section, we prove a general estimate on the time
it takes for a Markov process $X$ to hit a set~$\Rset$.
The classical Foster--Lyapunov hitting-time estimates 
assume that there is a {\it Lyapunov function} $\Phi$ such that, outside of $\Rset$, the function $\Phi$ is bounded from below and  $\Phi(X(t))$  tends to decay. 
Although our estimate, 
which is a version of a result in \cite{YM12}, 
is also based on the Foster--Lyapunov approach, it is targeted at situations where there also may be regions where the Lyapunov function is allowed to grow on average. 
However, we will assume that the process does not spend too much time in those regions, so that the decay of $\Phi(X(t))$ dominates outside of~$\Rset$. 
The resulting estimate will be used in Section~\ref{sec:proofs-Lyap} to prove recurrence of certain regions for the diffusions defined in Section~\ref{sec:2d-setup-result}.

We consider a homogeneous Markov process $X$ with continuous paths on an open state space $\bU \subset \rr^n$, with the strong Markov property; see, e.g., \cite[Section 2.6B]{KaSh}. The distribution on paths started at $x\in\bU$ will be denoted by $\pp^x$, and the expectation with respect to $\pp^x$ will be denoted by~$\ee^x$. We let $(\cF_t)_{t\ge 0}$ be the natural filtration of the process. 

To impose a regularity assumption on the process, we will need a sequence of open precompact sets $(\bU_m)_{m\in\N}$ such that $\bU_m\subset \bU_{m+1}$ for all $m$, $\bigcup_{m\in\N}\bU_m=\bU$. We define $\BB$ to be the space of functions locally bounded above on $\bU$, i.e., functions $f:\bU\to\R$ such that for every $m\in\N$, $\sup\{f(x):\, x\in\bU_m\}<\infty$.

We assume that there is an operator $\cL:C^2(\bU)\to\BB$ such that for every $f\in C^2(\bU)$ and every starting point $x\in\bU$, there is a continuous local martingale $(M(t))_{t\ge 0}$ w.r.t.\ $(\cF_t)_{t\geq 0}$ such that the following ``It\^o formula'' holds $\pp^x$-almost surely:  
\begin{equation}
    \label{eq:Ito-with-generator}
f(X(t))=f(x)+\int_{0}^t \cL f(X(s)) \dd s + M(t), \quad t\ge 0.
\end{equation} 
The local martingale property means that there is a sequence $(\eta_m)_{m\in\N}$ of stopping times such that 
(i) $\eta_m\uparrow+\infty$ and 
(ii) for each $m\in\N$, the process $M^{\eta_m}(t)=M(\eta_m\wedge t)$ is a martingale w.r.t.\ $(\cF_t)_{t\geq 0}$. 
In our setup, we can equivalently require in addition that  for all $m\in\N$ the following holds:
(iii) $\eta_m\le m$;  
(iv) $X(t)\in \bU_m$ for all $t\le \eta_m$; 
(v) $|M(t)|\le m$ for all $t\le \eta_m$.

\begin{remark}
We introduce this setup mainly with diffusion processes in mind. Namely all these properties hold for a diffusion driven by an It\^o SDE with Lipschitz coefficients under the assumption that $\bU$ is forward invariant for this diffusion (the latter can be expressed in terms of behavior of the coefficients near $\partial \bU$). In the SDE setting,
the role of $\cL$  is played by the usual second order differential operator, and \eqref{eq:Ito-with-generator} is the actual It\^o  formula with $M$ being an It\^o integral; see~\eqref{eq:def-cL-as-diff}--\eqref{eq:first-explicit-Ito}.
\end{remark}

To state the main result of this section, in addition to the setup described above, we need closed sets~$\Rset,\Ssetin\subset \bU$,  an open set  $\Ssetout\subset\bU$ containing $\Ssetin$, and $\Qset = (\Rset\, \cup\, \Ssetout)^\complement$ (throughout this section, complements are taken with respect to $\bU$, i.e., $A^\complement=\bU\setminus A$ for any $A\subset \bU$).   
Continuity of  $X$ implies that the hitting times 
$\hit_\Rset=\inf\{t \geq 0 : X_t \in \Rset\}$ and similarly defined 
$\hit_{\Ssetin \cup \Rset}$ and $\hit_{\Ssetout^\complement}$ are stopping times. 
We view $\Ssetin$ and $\Ssetout$ as an ``inner'' set  and an ``outer'' one.

\begin{theorem}
\label{prop:YM12}
    Suppose that a  $C^2$ function $\Phi  : \Ssetout \cup \Qset \to [1,\infty)$
    and constants  $K \geq 1$ and $T > 0$ satisfy the following conditions:
    \begin{enumerate}[\rm (a)]
        \item\label{it:usual-Lyap-req}
        for all $x \in \Ssetout$,
        \begin{equation}
            \label{eq:typical-Lyap-decay}
                (\cL \Phi)(x) \leq -1;
            \end{equation}
        \item\label{it:Lyap-on-Q}
        $(\cL \Phi )(x) \leq K$ for all~$x \in \Qset$;      
        \item\label{it:hit-Si-R}
        $\ee^x \hit_{\Ssetin \cup \Rset} \leq T$ for all $x \in \Qset$;
        \item\label{it:exit-So}
        $\ee^x \hit_{\Ssetout^\complement}  \geq T(2K + 1)$ for all $x \in \Ssetin$.
\end{enumerate}
    Then, for all~$x \notin \Rset$, 
    \begin{equation}
    \label{eq:YM-goal}
        \ee^x \hit_{\Rset} \leq 3\Phi (x) + 2KT.
    \end{equation} 
    
\end{theorem}

\begin{remark}
\label{rem:usual-Lyap}
    If \eqref{eq:typical-Lyap-decay} holds on the entire $\Rset^\complement$, then the classical Foster--Lyapunov estimate mentioned in the beginning of this section holds: 
    \begin{equation}
    \label{eq:typical-Lyap-conclu}
        \ee^x \hit_\Rset \leq \Phi(x),\quad x\in \Rset^\complement,
    \end{equation}
    see, e.g.,~\cite[Section~{3.7}]{Kha}.  We will recall a simple derivation of  this estimate after stating Corollary~\ref{cor:Dynkin}. In Theorem~\ref{prop:YM12}, imposing less restrictive assumptions on $\cL$ and $\Phi$ on $\Rset^\complement$, we still can control the expected return time to $\Rset$.
\end{remark}

\begin{remark}
    In applications of this result, the set $\Ssetout$
    may have several connected components each supplied 
    with a separate Lyapunov function satisfying~\eqref{eq:typical-Lyap-decay} just on that component. Thus, the main issue we will need to solve in this case may be viewed as the problem of a common extension of these separate Lyapunov functions into $\Ssetout \cup \Qset$ satisfying the remaining three requirements.
\end{remark}

\subsection{Proof of Theorem~\ref{prop:YM12}.}

We begin with several auxiliary lemmas.
Our setup (including the requirements (i)--(v) on the localizing stopping time sequence $(\eta_m)_{m\in\nn}$), the optional stopping theorem, and Fatou's lemma 
imply the following Dynkin formula:
\begin{lemma}\label{lem:Dynkin} For all $x\in\bU$, all $f\in C^2(\bU)$  and every stopping time $\tau$,  all $m\in\N$,
\begin{equation}
    \label{eq:Dynkin1}
        \ee^x f(X(\tau^m))= f(x)+\ee^x \int_0^{\tau^m} \cL f(X(s)) \dd s,
\end{equation} 
where $\tau^m=\tau\wedge\eta_m$.
If, in addition, $f\ge 0$,  then
\begin{equation}
    \label{eq:Dynkin2}
       \ee^x f(X(\tau))\le  f(x)+\liminf_{m\to\infty} \ee^x \int_0^{\tau^m} \cL f(X(s)) \dd s.
\end{equation}
\end{lemma}

\begin{corollary}
    \label{cor:Dynkin}
     Suppose that $x\in\bU$ and $\Phi\in C^2(\bU)$ is nonnegative. Assume that a stopping time $\tau$ and a constant $c\in\R$ satisfy
    \[
       \pp^x\bigg\{\sup_{s\in[0,\tau)} \cL \Phi(X(s)) \le c \bigg\}=1.
    \]  
  Then, 
    \begin{equation}
        \label{eq:main_appl_of_Lyapunov}
        \ee^x \Phi(X(\tau))\le \Phi(x)+ c \ee^x \tau.
    \end{equation}
\end{corollary}

If one assumes that~\eqref{eq:typical-Lyap-decay} holds for ll $x\in\Rset^\complement$, this corollary with $c=-1$ and $\tau=\hit_R$ directly implies~\eqref{eq:typical-Lyap-conclu}. Our situation is more complex and requires further analysis.

\begin{lemma}
\label{lem:zeroth-transit}
    Under assumptions \eqref{it:usual-Lyap-req}--\eqref{it:hit-Si-R}, the inequality
    \begin{align*}
        \ee^x \Phi (X(\hit_{\Ssetin \cup \Rset})) \leq \Phi (x) + KT
    \end{align*}
    holds for every $x \notin \Ssetin$.
\end{lemma}

\begin{proof}
    If $x \notin \Ssetout$, then 
    Corollary~\ref{cor:Dynkin} with $c=K$ and assumptions~\eqref{it:Lyap-on-Q},\eqref{it:hit-Si-R} 
    imply
    \begin{equation}
        \label{eq:outside_So}
        \ee^x \Phi (X(\hit_{{\Ssetin \cup \Rset}})) 
        \le \Phi (x) + KT.
    \end{equation}
    If $x \in \Ssetout \setminus \Ssetin$, then, similarly,  Corollary~\ref{cor:Dynkin} with $c=-1$ and assumption~\eqref{it:usual-Lyap-req}  yield
    \begin{align*}
        \ee^x \Phi (X(\hit_{\Ssetin \cup \Rset} \wedge \hit_{\Ssetout^\complement})) 
        \le \Phi (x).
    \end{align*}
    Combining this with~\eqref{eq:outside_So} and the strong Markov property, we complete the proof.
\end{proof}

To prove Proposition~\ref{prop:YM12}, we define  $\tau_0 = 0$ and, recursively, for $n\in\N$,
\begin{align}
\label{eq:taus}
    \tau_{n+1} &= 
        \begin{cases}
            \inf\Big\{t \geq \tau_{n} : X(t) \in \Ssetin \cup \Rset \text{ and}\  X(t') \notin \Ssetout  \text{ for some}\  t' \in [\tau_{n},t] \Big\} 
               &\text{if } X(\tau_n) \in \Ssetin \setminus \Rset, \\
            \inf\big\{t \geq \tau_{n} : X(t) \in \Ssetin \cup \Rset\big\} 
                &\text{otherwise.}
        \end{cases}
\end{align}
By definition, 
\begin{equation}
\label{eq:hitting-triviality}
    \hit_R \leq \tau_\nu,
\end{equation}
where 
\begin{equation}
\label{eq:first-visit-to-R}
    \nu= \inf\{n \geq 0 : X(\cT_n) \in \Rset\}.
\end{equation}

\begin{lemma}
\label{lem:pre-YM12-b}
    Under the assumptions of Proposition~\ref{prop:YM12},   
    \begin{align*}
            \ee^x {\tau_{1}} \leq \Phi (x) + T,\quad x \notin \Ssetin.
    \end{align*}
\end{lemma}

\begin{proof}
    Let
    $  
        \tau  = \inf\{t \geq 0 : X(t) \in \Qset \cup \Rset \cup \Ssetin \}
    $.
    Using Corollary~\ref{cor:Dynkin} with $c=-1$ and the fact that $\ee^x \Phi(X(\tau)) \geq 0$, we obtain $\ee^x \tau \leq \Phi(x)$. Our claim now follows from $\ee^x[\tau_1 - \tau] \leq T$ which can be derived
    from assumption~(c) and the strong Markov property applied on $\{\tau_1 - \tau> 0\}$.
\end{proof}

\begin{lemma}
\label{lem:pre-YM12}
    Under the assumptions of Proposition~\ref{prop:YM12},   for all~$x \in \Ssetin$,
    \begin{align}
        \label{eq:pre-YM12}
            \ee^x \Phi (X({\tau_{1}})) \leq \Phi (x) - \frac 12 \ee^x \tau_{1}.
    \end{align}
\end{lemma}

\begin{proof}
    If $x \in \Rset \cap \Ssetin$, then the definition~\eqref{eq:taus} specifies that $\tau_1=0$, and~\eqref{eq:pre-YM12} trivially holds. From now on we consider a point $x \in \Ssetin \setminus R$.
    Let     $\theta_1 = \hit_{\Ssetout^\complement}$ and $\theta_2=\hit_{\Ssetin \cup \Rset}$.
    Denoting $E_1(x)=\ee^x \theta_1$ and $E_2(y)=\ee^y\theta_2$,  $\Psi(y)= \ee^{y}\Phi(X(\theta_2))$
    for $y\in \Ssetout^\complement$,
    we obtain 
    \begin{align}
        \label{eq:E_Phi_strong_Markov}
        \ee^x \Phi(X(\tau_1)) &= \ee^x \Psi(X(\theta_1)),
    \\
    \ee^x \tau_1 &= E_1(x)+ \ee^x E_2(X(\theta_1)). 
    \label{eq:strong_Markov_for_Etau1}
    \end{align}    
    Corollary~\ref{cor:Dynkin} and assumption~\eqref{it:usual-Lyap-req} imply 
        \begin{equation}
            \label{eq:bound_on_Phi_X_theta_1}
            \ee^{x}[\Phi(X(\theta_1))] \leq \Phi(x) - E_1(x). 
        \end{equation}
        In particular, $E_1(x)< \infty$.
    Corollary~\ref{cor:Dynkin} and assumption~\eqref{it:Lyap-on-Q} imply
    \begin{equation}
        \label{eq:bound_on_Psi_y}
        \Psi(y)\leq \Phi(y) + KE_2(y),\quad y \in \Ssetout^\complement,
    \end{equation}where $E_2(y) \leq T$ according to assumption~\eqref{it:hit-Si-R}.
    Therefore all the expectations in the following chain of inequalities are finite:
    \begin{align*}
        \ee^x \Psi(X(\theta_1))
        &\leq \ee^x [\Phi(X(\theta_1)) + KE_2(X(\theta_1))] \\
        &\leq \Phi(x) - E_1(x) + K\ee^x E_2(X(\theta_1)),\\
        &\le   \Phi(x) - \tfrac 12 (E_1(x) + \ee^x E_2(x)) - \tfrac 12 E_1(x)+ (K+\tfrac 12)\ee^x E_2(X(\theta_1)),
        \\ 
        &\le  \Phi(x) - \tfrac 12 (E_1(x) + \ee^x E_2(x)) 
        \\
        &\le \Phi(x)-\tfrac{1}{2}\ee^x \tau_1.
    \end{align*}
    We used~\eqref{eq:bound_on_Psi_y} in the first inequality, \eqref{eq:bound_on_Phi_X_theta_1} in the second one.   The third line is just a rearrangement, the fourth one  follows from  assumptions~\eqref{it:hit-Si-R} and~\eqref{it:exit-So}, the fifth one follows from \eqref{eq:strong_Markov_for_Etau1}. Using \eqref{eq:E_Phi_strong_Markov}, we complete the proof of the lemma.
\end{proof}

\begin{proof}[Proof of Theorem~\ref{prop:YM12}]
    We denote $H(x)=\ee^x \cT_\nu$.
    We obviously have
    \begin{equation}
        \label{eq:if_x_in_R}
        H(x)=0,\quad x\in R,
    \end{equation}    
    so we only need to focus on the case $x\notin\Rset$.
    First note that, in view of~\eqref{eq:taus}--\eqref{eq:first-visit-to-R},  it suffices to show $H(x) \leq 3\Phi (x) + 2KT$.
    The latter follows from 
    \begin{equation}
    \label{eq:sigma-bound-R}
        H(x) \leq 2\Phi(x), \quad x \in \Ssetin.
    \end{equation}
    To see this, consider $x \notin \Ssetin$ and recall that, by Lemma~\ref{lem:pre-YM12-b}, we have $\ee^x \cT_1 \leq \Phi(x)$. Hence, using the strong Markov property, we have 
    \begin{align*}
        H(x)
            &= \ee^x\cT_1 + \ee^x[\cT_\nu - \cT_1] 
            \leq \Phi(x) + \ee^x H(X(\cT_1)),
    \end{align*}
    Now, since $X(\cT_1) \in \Ssetin \cup \Rset$ by definition, \eqref{eq:if_x_in_R},  \eqref{eq:sigma-bound-R}, and Lemma~\ref{lem:zeroth-transit} imply
    \begin{align*}
        H(x)
            &\leq \Phi(x) + \ee^x[2\Phi(X(\cT_1))]
            \leq \Phi(x) + 2(\Phi (x) + KT) 
            = 3\Phi(x) + 2KT.
    \end{align*}

    Let us now prove that~\eqref{eq:sigma-bound-R} holds following the strategy of~\cite[Section~{V.A}]{YM12}.
    We set
    \begin{align*}
        M_0 &= \Phi (x) = \Phi (X(\cT_0)),\\   
        M_{n} &= \Phi (X(\cT_{n})) + \frac 12 \sum_{n' = 0}^{n-1}\ee^x [\cT_{n'+1} - \cT_{n'}|\cF_{\cT_{n'}}],
        \quad  n\in\nn.
    \end{align*}
    Lemma~\ref{lem:pre-YM12} and the strong Markov property imply the following supermartingale property:
    \begin{align*}
        \ee^x[M_{n+1}|\cF_{\tau_n}] \leq M_n.
    \end{align*}
    Now, for all $m\in\nn$, we define $\nu^m  = \nu \wedge m$, where $\nu$ has been defined in~\eqref{eq:first-visit-to-R}. Then,
    \begin{align*}
        \ee^x \cT_{\nu^m} 
            = \ee^x  
                \sum_{n=0}^{\nu^m - 1}\ee^x[\cT_{n+1} - \cT_{n}|\cF_{\cT_{n}}]  
            = 
            2\ee^x\left[ M_{\nu^m} - \Phi (X(\cT_{\nu^m}))\right] \leq 2\ee M_{\nu^m} 
            \leq 2M_0=2\Phi(x).
    \end{align*}
    We have used $\Phi\ge 0$ in the first inequality and Doob's optional stopping theorem in the second one. Taking $m\to\infty$ and using Lebesgue monotone convergence, we deduce~\eqref{eq:sigma-bound-R}, as desired.
\end{proof}

\section{Completing the proof of Theorem~\ref{thm:main2d}}
\label{sec:proofs-Lyap}

\subsection{A Case study} 
\label{ssec:conclu-main-attr}

For concreteness, let us first consider the following situation, which falls within case~\ref{it:attractor-case}: $\Or^0 \in \Attr$ is the only attracting vertex, $\Or^1$ and~$\Or^3$ are positively oriented saddles and~$\Or^2$ is a source; $\Ed^1 \in \Supt \cap \Attr$ is the only attracting edge and $\Ed^2 \in \Supt \setminus \Attr$ is the only other edge on which an invariant probability measure is concentrated.
In such a situation, our claim is that, starting at $x\in\bXo$, almost every trajectory eventually converges to either~$\Or^0$ or $\Ed^1$, with a positive probability for both: 
\begin{equation}
\label{eq:particular-case-1-preconclu}
    p^x_{\Or^0}>0
    \qquad\text{ and }\qquad
    p^x_{\Ed^1}>0,
\end{equation}
and
\begin{equation}
\label{eq:particular-case-1-conclu}
    p^x_{\Or^0} + p^x_{\Ed^1} = 1.
\end{equation}

Recall that, by Lemmas~\ref{lem:corners-attract} and~\ref{lem:edge-abrsob-prob}, we already have
\begin{align}
\label{eq:local-abs-prob}
    \inf_{x \in \overline{U_r^0}} p^x_{\Or^0} > 0
    \qquad\text{ and }\qquad 
    \inf_{x \in \overline{G_r^1}} p^x_{\Ed^1} > 0
\end{align}
for all $r>0$ small enough; see Figure~\ref{fig:basic-split-attr} for a visual reminder of the shapes of the sets~$\overline{U_r^0}=\overline{U_{r,r}^0}$ and~$\overline{G_r^1}=\overline{G_{r,r}^1}$ used throughout Sections~\ref{sec:local-attr} and~\ref{sec:main-heteroc-proof}. 
Assumption~\ref{req:irreducibility}, the Markov property and~\eqref{eq:local-abs-prob} readily ensure that~\eqref{eq:particular-case-1-preconclu} holds for all $x \in \bXo$, and we will momentarily see how adding a large enough central compact to the closures of~${U_r^0}$ and ${G_r^1}$ yields a set~$\Rset$ such that 
\begin{equation}
\label{eq:uniform-abs-p}
    \inf_{x \in R} \left(p^x_{\Or^0} + p^x_{\Ed^1} \right) > 0.
\end{equation}
It is this set~$\Rset$ for which we want to establish the recurrence property in Theorem~\ref{prop:YM12}, making it straightforward to obtain~\eqref{eq:particular-case-1-conclu} from~\eqref{eq:uniform-abs-p}. The idea is to construct the nested sets $\Ssetin \subsetneq \Ssetout$ near $\Or^1, \Or^2, \Or^3$ and $\Ed^2$.
    Because the drift and diffusion are well controlled, nesting the inner set deeply enough into the outer set gives a lower bound as in condition~(\ref{it:exit-So}).
     Because $\Or^1, \Or^2, \Or^3$ and $\Ed^2$ are not attracting, we can construct, on each connected component of~$\Ssetout$, a positive Lyapunov function as in condition~(\ref{it:usual-Lyap-req}).
The set $\Qset$ will then consist of what is left out of $\Ssetout$ and $\Rset$. 
    Because the drift is consistent enough, we can give an upper bound on exit times from $\Qset$, which is an important step towards condition~(\ref{it:hit-Si-R}) if $\Ssetin$ is not too deeply nested into~$\Ssetout$.
    We will want to arrange things so that there is sufficient room for us to extend the different local definitions of Lyapunov functions in components of~$\Ssetout$ in such a way that condition~(\ref{it:Lyap-on-Q}) holds.

\medskip

Now, let us fix a value of~$r > 0$ as in~\eqref{eq:local-abs-prob}, with the additional constraint that the smallness of~$r$ warrants applications of Corollaries~\ref{cor:Lyap-corner-pre} and~\ref{cor:mean-H-process} (and all their rotated counterparts). 
Consider the set $(F^*_{r'})^\complement=\bXo\setminus F^*_{r'}$ obtained by removing strips of small width $r'>0$ about each edge from~$\bXo$. 
Since it is a compact subset of $\bXo$, 
Lemma~\ref{lem:uniform-irred} yields that, for every $r'\in(0,\tfrac 12)$, there exists $T \geq 0$ such that 
\begin{equation}
\label{eq:intermediate-trans}
    \inf_{x \in (F^*_{r'})^\complement} \pp^x \left\{ \hit_{{U_r^0} \cup {G_r^1}} \leq T \right\} > 0.
\end{equation}
Then, yet another standard argument using the strong Markov property shows that, indeed, for every $r'\in(0,\tfrac 12)$, the lower bounds~\eqref{eq:local-abs-prob} and~\eqref{eq:intermediate-trans} imply the existence of $\delta \in (0,1]$ such that 
\begin{equation}
\label{eq:conv-lb-from-Rset}
    x \in \Rset := \overline{U_r^0} \cup \overline{G_r^1} \cup (F^*_{r'})^\complement \qquad\Rightarrow\qquad p^x_{\Or^0} + p^x_{\Ed^1} \geq \delta.
\end{equation}

\begin{figure}
    \centering
    \footnotesize
    \begin{subcaptionblock}{.49\textwidth}
        \centering
        \begin{tikzpicture}
            \begin{axis}[
                xmin=0, xmax=1,
                xtick=\empty,
                ymin=0, ymax=1,
                ytick=\empty,
                clip=false,
                unit vector ratio=1 1,
                ]
        
                \fill[cpatternone] (0,0) rectangle (\err,\err);
                    \node[anchor=north west] at (\err,\err) {$\overline{U^0_r}$};
                \fill[cpatternone] (1-\err,\err) rectangle (1,1-\err);
                    \node[anchor=east] at (1-\err,.5) {$\overline{G^1_r}$};
        
                \draw[->] (.5,0) -- (\near,0);
                \draw[->] (0,.5) -- (0,\near);
                \draw[->] (0,.5) -- (0,1-\near);
                \draw[->] (.5,0) -- (1-\near,0);
        
                \draw[->] (1,1) -- (1,1-\near);
                \draw[->] (1,1) -- (1-\near,1);
                \draw[->] (0,1) -- (\near,1);
                \draw[->] (1,0) -- (1,\near);
        
                \node[below] at (.5,1) {$\downarrow\downarrow$};
                \node[left] at (1,.5) {\rotatebox{90}{$\downarrow\downarrow$}};

                \node[below] at (0.2*\err,1) {$F^*_{r'}$};
                \draw[dashed] (1.05*\err*\errplus,1.05*\err*\errplus) rectangle (1-1.05*\err*\errplus,1-1.05*\err*\errplus);
                \draw (0,0) rectangle (1,1);
            \end{axis}
        \end{tikzpicture}
        \caption{There are regions near attracting vertices and edges in which we have already proved positive probability of convergence to said attracting object in Lemmas~\ref{lem:corners-attract} and~\ref{lem:edge-abrsob-prob}.}
        \label{fig:basic-split-attr}
    \end{subcaptionblock}
    \hfill
    \begin{subcaptionblock}{.49\textwidth}
        \centering
        \begin{tikzpicture}
            \begin{axis}[
                xmin=0, xmax=1,
                xtick=\empty,
                ymin=0, ymax=1,
                ytick=\empty,
                clip=false,
                unit vector ratio=1 1,
                ]
        
                \fill[cpatterntwo] (1,0) rectangle (1-\err,\err);
                    \node[anchor=south west] at (1-\err,\err) {$U^1$};
                \fill[cpatterntwo] (1,1) rectangle (1-\err*\errmoins*\errmoins, 1-\err);
                    \node[anchor=north west] at (1-\err*\errmoins*\errmoins, 1-\err) {$U^2$};
                \fill[cpatterntwo] (\err*\errmoins,1) rectangle (1-\err*\errmoins, 1-\err);
                    \node[anchor=north] at (0.5, 1-\err) {$G^2$};
                \fill[cpatterntwo] (0,1) rectangle (\err*\errmoins*\errmoins,1-\err);
                    \node[anchor=north east] at (\err*\errmoins*\errmoins,1-\err) {$U^3$};
        
                \draw[->] (.5,0) -- (\near,0);
                \draw[->] (0,.5) -- (0,\near);
                \draw[->] (0,.5) -- (0,1-\near);
                \draw[->] (.5,0) -- (1-\near,0);
        
                \draw[->] (1,1) -- (1,1-\near);
                \draw[->] (1,1) -- (1-\near,1);
                \draw[->] (0,1) -- (\near,1);
                \draw[->] (1,0) -- (1,\near);
        
                \node[below] at (.5,1) {$\downarrow\downarrow$};
                \node[left] at (1,.5) {\rotatebox{90}{$\downarrow\downarrow$}};
            \end{axis}
        \end{tikzpicture}
        \caption{There are regions in which it is straightforward to construct, locally, a Lyapunov function with suitable decay using Corollaries~\ref{cor:Lyap-corner-pre} and~\ref{cor:mean-H-process}.}
        \label{fig:basic-split-non-attr}
    \end{subcaptionblock}
    \caption{Vertices and edges are separated into two categories: those that are attracting are associated with closed sets that will become part of~$\Rset$ (panel~\ref{fig:basic-split-attr}); the others, to open sets that will become part of~$\Ssetout$ (panel~\ref{fig:basic-split-non-attr}). The arrows near the vertices encode the signs of the linearization coefficients.}
\end{figure}

Building on Corollaries~\ref{cor:Lyap-corner-pre} and~\ref{cor:mean-H-process}, we start by making local definitions of the form
\begin{subequations}
\label{eq:local-def-ABC}
\begin{align}
    \Phi &= -\alpha_k \ln{}\circ x_1^k - \beta_k \ln{}\circ x_2^k 
        & \text{ on } U^k & \text{ with } \Or^k \notin \Attr, \\
    \Phi &= -\gamma_k\, \psi^k \circ x^k -\gamma_k \ln{}\circ x_2^k 
        & \text{ on } G^k & \text{ with } \Ed^k \in \Supt \setminus \Attr,
\end{align}
\end{subequations}
where $\alpha_k$, $\beta_k$ and $\gamma_k$ are scalar parameters, and~$U^k$ and~$G^k$ are open sets of the same form as was used throughout Sections~\ref{sec:local-attr} and~\ref{sec:main-heteroc-proof}; see Figure~\ref{fig:basic-split-non-attr}.
The following is a 
direct consequence of Corollaries~\ref{cor:Lyap-corner-pre} and~\ref{cor:mean-H-process}: provided that these sets are small enough, the local definitions~\eqref{eq:local-def-ABC} ensure that $\cL \Phi \leq -1$ whenever 
\begin{subequations}
\label{eq:decay-constraints-ABC}
\begin{align}
    \alpha_k\eivalh^k + \beta_k\eivalv^k &> 2 
    & \text{ for } k & \text{ with } \Or^k \notin \Attr, \\
    \gamma_k \oldlambda^k &> 2
    & \text{ for } k & \text{ with } \Ed^k \in \Supt \setminus \Attr.
\end{align}
\end{subequations}

\begin{figure}
    \centering 
    \footnotesize
    \begin{subcaptionblock}[T]{.49\textwidth}
        \centering
        \begin{tikzpicture}[scale=1]
            \footnotesize
            \begin{axis}[
                xmin=0, xmax=1,
                xtick={\err,1-\err},
                xticklabels={{},{}},
                ymin=0, ymax=1,
                ytick={\err,1-\err},
                yticklabels={{},{}},
                clip=false,
                unit vector ratio=1 1,
                ]
        
                \fill[cpatternone] (\err*\errplus*\errmoins*\errmoins,\err*\errplus*\errmoins*\errmoins) rectangle (1-\err*\errplus*\errmoins*\errmoins,1-\err*\errmoins*\errmoins*\errplus);
                \fill[cpatternone] (0,0) rectangle (\err,\err);
                \fill[cpatternone] (1-\err,\err) rectangle (1,1-\err);

                \fill[cpatterntwo] (1,0) rectangle (1-\err,\err);
                \fill[cpatterntwo] (1,1) rectangle (1-\err*\errmoins*\errmoins, 1-\err);
                \fill[cpatterntwo] (\err*\errmoins,1) rectangle (1-\err*\errmoins, 1-\err);
                \fill[cpatterntwo] (0,1) rectangle (\err*\errmoins*\errmoins,1-\err);

                \node[rotate=90,right] at (\err*\errmoins,1) {$\rmoins r$};
                \node[rotate=90,right] at (0.9*\err*\errmoins*\errmoins,1) {$\rmoins^2 r$};
                \node[rotate=90,right] at (\err*\errmoins*\errmoins*\errplus,1) {\phantom{$\rplus \rmoins^2 r$}};

                \node[right] at (1,\err) {$r$};
                \node[right] at (1,\err*\errplus*\errmoins*\errmoins) {$r'$};

                \draw[->] (.5,0) -- (\near,0);
                \draw[->] (0,.5) -- (0,\near);
                \draw[->] (0,.5) -- (0,1-\near);
                \draw[->] (.5,0) -- (1-\near,0);
        
                \draw[->] (1,1) -- (1,1-\near);
                \draw[->] (1,1) -- (1-\near,1);
                \draw[->] (0,1) -- (\near,1);
                \draw[->] (1,0) -- (1,\near);

                \draw (0,0) rectangle (1,1);
            \end{axis}
        \end{tikzpicture}
        \caption{Between adjacent edge and corner regions where we have made local definitions of the Lyapunov function, we keep a buffer region corresponding to a logarithmic increment of~$|{\ln\rmoins}|$ along the edge. These buffer regions will be part of the set~$\Qset$ and support the derivatives of transition functions.}
        \label{fig:buffers}
    \end{subcaptionblock}
    \hfill
    \begin{subcaptionblock}[T]{.49\textwidth}
        \centering
        \begin{tikzpicture}[scale=1]
            \begin{axis}[
                xmin=0, xmax=1,
                xtick={\err,1-\err},
                xticklabels={{},{}},
                ymin=0, ymax=1,
                ytick={\err,1-\err},
                yticklabels={{},{}},
                clip=false,
                unit vector ratio=1 1,
                legend pos=outer north east,legend style={xshift=0 pt}
                ]

                \addlegendimage{area legend, cpatternthree}
                \addlegendimage{area legend, cpatterntwo}
                \addlegendimage{area legend, cpatternone}
                \addlegendimage{area legend, fill=white}

                \legend{$\Ssetin$, $\Ssetout$,$\Rset$,$\Qset$}
        
                \fill[cpatternone] (\err*\errplus*\errmoins*\errmoins,\err*\errplus*\errmoins*\errmoins) rectangle (1-\err*\errplus*\errmoins*\errmoins,1-\err*\errplus*\errmoins*\errmoins);
                \fill[cpatternone] (0,0) rectangle (\err,\err);
                \fill[cpatternone] (1-\err,\err) rectangle (1,1-\err);
                
                \fill[cpatterntwo] (1,0) rectangle (1-\err,\err);
                \fill[cpatternthree] (1,0) rectangle (1-\err*\errmoins,\err*\errplus);
        
                \fill[cpatterntwo] (1,1) rectangle (1-\err*\errmoins*\errmoins, 1-\err);
                \fill[cpatternthree] (1,1) rectangle (1-\err*\errmoins*\errmoins*\errplus, 1-\err*\errplus);
        
                \fill[cpatterntwo] (\err*\errmoins,1) rectangle (1-\err*\errmoins, 1-\err);
                \fill[cpatternthree] (\err,1) rectangle (1-\err, 1-\err*\errplus);
        
                \fill[cpatterntwo] (0,1) rectangle (\err*\errmoins*\errmoins,1-\err);
                \fill[cpatternthree]  (0,1) rectangle (\err*\errmoins*\errmoins*\errplus,1-\err*\errmoins);

                \node[rotate=90,right] at (\err,1) {$r$};
                \node[rotate=90,right] at (\err*\errmoins,1) {$\rmoins r$};
                \node[rotate=90,right] at (0.9*\err*\errmoins*\errmoins,1) {$\rmoins^2 r$};
                \node[rotate=90,right] at (\err*\errmoins*\errmoins*\errplus,1) {$\rplus \rmoins^2 r$};

                \node[right] at (1,\err) {$r$};
                \node[right] at (1,\err*\errplus) {$\rplus r$};

                \draw[->] (.5,0) -- (\near,0);
                \draw[->] (0,.5) -- (0,\near);
                \draw[->] (0,.5) -- (0,1-\near);
                \draw[->] (.5,0) -- (1-\near,0);
        
                \draw[->] (1,1) -- (1,1-\near);
                \draw[->] (1,1) -- (1-\near,1);
                \draw[->] (0,1) -- (\near,1);
                \draw[->] (1,0) -- (1,\near);

                \draw[stealth-stealth, green, thick] (\err,0.3*\err*\errplus) -- (1-\err*\errmoins,0.3*\err*\errplus);

                \draw[stealth-, densely dotted, red, thick] (1-\err,0.7*\err*\errplus) -- (1-\err*\errmoins,0.7*\err*\errplus);

                \draw[stealth-,densely dotted, red, thick] (1-0.5*\err*\errplus*\errmoins*\errmoins,\err) -- (1-0.5*\err*\errplus*\errmoins*\errmoins,\err*\errplus);

                \draw[stealth-stealth, green, thick] (0.3*\err*\errplus*\errmoins*\errmoins,\err) -- (0.3*\err*\errplus*\errmoins*\errmoins,1-\err*\errmoins);

                \draw[-stealth, green, thick] (\err*\errmoins*\errmoins,1-0.3*\err*\errplus) -- (\err,1-0.3*\err*\errplus);

                \draw[stealth-, densely dotted, red, thick] (0.7*\err*\errplus*\errmoins*\errmoins,1-\err) -- (0.7*\err*\errplus*\errmoins*\errmoins,1-\err*\errmoins);
                
                \draw[stealth-,densely dotted, red, thick] (\err*\errmoins*\errmoins,1-0.5*\err*\errplus) -- (\err*\errplus*\errmoins*\errmoins,1-0.5*\err*\errplus);

                \draw[-stealth, densely dotted, red, thick] (0.5,1-\err*\errplus) -- (0.5,1-\err);

                \draw[stealth-,densely dotted, red, thick] (1-0.5*\err*\errplus*\errmoins*\errmoins,1-\err) -- (1-0.5*\err*\errplus*\errmoins*\errmoins,1-\err*\errplus);

                \draw[stealth-,densely dotted, red, thick] (1-\err*\errmoins*\errmoins,1-0.5*\err*\errplus) -- (1-\err*\errplus*\errmoins*\errmoins,1-0.5*\err*\errplus);

                \draw[-stealth, green, thick] (1-\err*\errmoins*\errmoins, 1-0.3*\err*\errplus) -- (1-\err,1-0.3*\err*\errplus);

                \draw[-stealth, red, densely dotted, thick] 
                (1-\err,1-0.7*\err*\errplus) -- (1-\err*\errmoins, 1-0.7*\err*\errplus);

                \draw[-stealth, red, densely dotted, thick] 
                (\err,1-0.7*\err*\errplus) -- (\err*\errmoins, 1-0.7*\err*\errplus);

                \draw (0,0) rectangle (1,1);
            \end{axis}
        \end{tikzpicture}
        \caption{The logarithmic increment by which $\Ssetin$ is inset into $\Ssetout$ in a given direction depends on the linearized drift: if exiting from $\Ssetout$ in that direction goes against the drift, we choose $|{\ln \rmoins}|$; if it goes along the drift, we choose $|{\ln\rplus}|$.}
        \label{fig:insets}
    \end{subcaptionblock}
    \caption{A key feature of the construction is that hitting $\Ssetin \cup \Rset$ starting in~$\Qset$ can be achieved relatively quickly, taking advantage of a moderate logarithmic increment $|{\ln\rmoins}|$ going along the drift (solid, green arrows), while exits from~$\Ssetout$ starting in~$\Ssetin$ are unlikely to be quick either because they go against the drift or because they correspond to a large logarithmic increment $|{\ln \rplus}|$ (dotted, red arrows).}
    \label{fig:buffers-and-insets}
\end{figure}

The problem would be simpler if we were able to choose parameters $\alpha_k$, $\beta_k$ and $\gamma_k$ satisfying~\eqref{eq:decay-constraints-ABC} in such a way that the local definitions~\eqref{eq:local-def-ABC} admit an extension that still satisfies $\cL \Phi \leq -1$ on all of~$\bXo\setminus\Rset$; see Remark~\ref{rem:usual-Lyap}. We are only able to make a weaker claim that is still sufficient for recurrence.  Namely, we are able to choose parameters $\alpha_k$, $\beta_k$ and $\gamma_k$ satisfying~\eqref{eq:decay-constraints-ABC} in such a way that the local definitions~\eqref{eq:local-def-ABC} admit an extension that still satisfies the hypotheses of the more subtle Theorem~\ref{prop:YM12} for some choices of sets $\Ssetin$, $\Ssetout$ and $\Qset$. 
Given the form of the differential operator~$\cL$ and the local definitions of~$\Phi$ in~\eqref{eq:local-def-ABC}, 
it is reasonable to impose an additional nonrestrictive assumption 
\begin{subequations}
\label{eq:pre-gluing-constraints-ABC}
\begin{align}
    \beta_k &= \alpha_{k+1}, \\
    \gamma_k &= \beta_k & \text{ for } k & \text{ with } \Ed^k \in \Supt \setminus \Attr
\end{align}
\end{subequations}
(recall that we enumerate vertices modulo $4$).
Although~\eqref{eq:pre-gluing-constraints-ABC} couples constraints that were independent in~\eqref{eq:decay-constraints-ABC}, we can prove that such a choice is possible.

\begin{lemma}
\label{lem:chooseing-betas-case-I}
    Parameters~$(\beta_k)_{k=0}^3$ can be chosen to ensure~\eqref{eq:decay-constraints-ABC} and~\eqref{eq:pre-gluing-constraints-ABC}.
\end{lemma} 
\begin{proof}   
    It suffices to find positive parameters $(\tilde\beta_k)_{k=0}^3$ satisfying
    \begin{equation}
    \label{eq:decay-and-gluing-constraints-ABC}    
        \tilde{\beta}_{k-1}\eivalh^k + \tilde{\beta}_k\eivalv^k > 0
    \end{equation}
    and to set $\beta_k = M \tilde{\beta}_k$ for some large enough constant~$M$. To this end, choose $\tilde{\beta}_0 =1$, then choose a positive number $\tilde{\beta}_{1}$ satisfying $\tilde{\beta}_{0} \eivalh^{1} + \tilde{\beta}_{1}\eivalv^{1} > 0$, then choose a positive number $\tilde{\beta}_{2}$ satisfying $\tilde{\beta}_{1} \eivalh^{2} + \tilde{\beta}_{2}\eivalv^{2} > 0$, and finally choose a positive number  $\tilde{\beta}_{3}$ satisfying $\tilde{\beta}_{2} \eivalh^{3} + \tilde{\beta}_{3}\eivalv^{3} > 0$. 
\end{proof}

Let us now to extend the local definitions~\eqref{eq:local-def-ABC} to $\bXo\setminus\Rset$ and choose the sets $\Ssetin$ and $\Ssetout$ in such a way that we can apply Theorem~\ref{prop:YM12}. Recalling that thin strips $F^k$ were used throughout Section~\ref{sec:main-heteroc-proof} and that $x_1^k$ and $x_2^k$ are functions defined on $\X$, we set
\begin{subequations}
\label{eq:Lyap-ext}
\begin{align}
    \Phi
    &= (\xi\circ x_1^k)\cdot(-\beta_{k-1}\ln{}\circ x_1^k - \beta_k \ln{}\circ x_2^k) 
       \notag \\
    &\quad{} + (1-\xi \circ x_1^k) \cdot(-\beta_k \ln{}\circ x_2^k - \beta_{k+1}\ln{}\circ(1-x_1^k)) 
    & \text{on } F_{r'}^k & \text{ with } \Ed^k \notin \Supt,
    \label{eq:Lyap-ext-direct}
    \\
    \Phi
    &= 
    (\chi_k \circ x_1^k) \cdot (-\beta_{k-1}\ln{}\circ x_1^k - \beta_k \ln{}\circ x_2^k)  \notag \\
    &\quad{} + (1-\chi \circ x_1^k) \cdot (1-\chi\circ(1-x_1^k)) \cdot(- \beta_k \psi_k \circ x_1^k - \beta_k \ln{}\circ x_2^k) \notag \\
    &\quad{} + (\chi\circ(1-x_1^k)) \cdot (-\beta_{k+1}\ln{}\circ(1-x_1^k)
    -\beta_k \ln{}\circ x_2^k)
    & \text{on } F_{r'}^k & \text{ with } \Ed^k \in \Supt \setminus \Attr,
    \label{eq:Lyap-ext-indirect}
\end{align}
\end{subequations}
where $\xi:[0,1]\to[0,1]$ is a smooth, nonincreasing transition function 
satisfying $\xi(r)=1$ and $\xi(1-r)=0$, 
 and 
$\chi:[0,1]\to[0,1]$ is a smooth nonincreasing transition function  satisfying
$\chi(\rmoins^2 r)=1$ and $\chi(\rmoins r)=0$ 
for some parameter $\rmoins\in(0,1)$.
The subsets of the strips $F^k_{r'}$ where~$\xi \circ x_1^k$, $\xi \circ (1-x_1^k)$, $\chi \circ x_1^k$ or~$\chi \circ (1-x_1^k)$ is not constant make up the region~$\Qset$ in which the mean decay of $\Phi(X(t))$ may fail; see Figure~\ref{fig:buffers}. To ensure that these definitions are consistent on their intersections and indeed provide extensions of the definitions~\eqref{eq:local-def-ABC}, we require $r'\leq \rmoins^2 r$; see Figure~\ref{fig:buffers}.
The following lemma ensures that requirement~\eqref{it:Lyap-on-Q} in Theorem~\ref{prop:YM12} is still satisfied. 

\begin{lemma}
\label{lem:pre-YM-i-particular}
    There exists $K>0$ such that, for every choice of $\rmoins$ small enough and $r'\in(0,\rmoins^2 r]$,
    there is a choice of transition functions~$\xi$ and $\chi$ for which 
    the formulas~\eqref{eq:Lyap-ext} provide an extension $\Phi : \Ssetout \cup \Qset \to \rr$ satisfying $\cL\Phi \leq K$ on~$\Qset$.
\end{lemma}

\begin{proof}
    A transition function $\xi$ with the desired properties can be constructed independently of~$\rmoins$. We can then use the product and chain rules, and Corollary~\ref{cor:Lyap-corner-pre} to estimate $\cL\Phi$ for~\eqref{eq:Lyap-ext-direct} independently of~$\rmoins$. 
    
    We now turn to the construction of~$\chi$, which, in contrast, needs to be tailored to the value of~$\rmoins$. 
    We construct the function $\chi$ by specifying $\eta = \chi \circ y_1^{-1}$, where $y_1: (0,1) \to \R$ is as in Lemma~\ref{lem:Y}.
    Take $\eta$ to be a smooth monotone functions that takes value~1 on~$(-\infty, 2\ln \rmoins + \ln r]$; takes value $0$ on~$[\ln \rmoins + \ln r, \infty)$; has first and second derivatives that are of order $|{\ln \rmoins}|^{-1}$ and supported on $[2\ln \rmoins + \ln r, \ln \rmoins + \ln r]$\,---\,for example, we can mollify a piecewise linear function. 
    Lemma~\ref{lem:psi} implies that $\psi \circ y_1^{-1}$ is globally Lipschitz and twice differentiable. 
    The same is true of $\ln{}\circ  y_1^{-1}$ and $\ln (1-\,\cdot\,) \circ y_1^{-1}$.  
    Hence, we can use the product and chain rules, and Corollaries~\ref{cor:Lyap-corner-pre} and~\ref{cor:mean-H-process} to bound $\cL\Phi$  for~\eqref{eq:Lyap-ext-indirect} independently of~$\rmoins$ small enough.
\end{proof}

We now construct $\Ssetin \subsetneq \Ssetout$ as in Figure~\ref{fig:insets} using an additional parameter $\rplus \in (0,1)$. 
The following two lemmas respectively ensure that requirements~\eqref{it:hit-Si-R} and~\eqref{it:exit-So} in Theorem~\ref{prop:YM12} are satisfied, with
\begin{equation}
\label{eq:T-choice}
    T = 2C_1 |{\ln\rmoins}|,
\end{equation}
where $C_1$ is the maximum of the constants in Lemmas~\ref{lem:ee-ap-bound},~\ref{lem:lin-ub-different-extremes} or \ref{lem:lin-ub-same-extremes} ($\nu = \min_{i,k} |\lambda_i^k|/2$, $A = \max_i \|s_i\|_\infty$, $b = -\ln r$). 

\begin{lemma}
\label{lem:return-Rset-particular-c}
    For all choices of~$\rplus$, $\rmoins$, and $r'\in(0,\rplus\rmoins^2 r]$, the sets $\Ssetin$, $\Rset$ and $\Qset$ are such that 
    $\ee^x \hit_{\Ssetin \cup \Rset} \leq T$ 
    for all $x \in \Qset$.
\end{lemma}

\begin{proof}
   There are two types of components of $\Qset$: those bridging two corner sets contained in $\Rset\cup\Ssetout$ and those between a corner set and an edge set  
   contained in $\Rset\cup\Ssetout$. We treat starting points in those separately.

   Suppose the starting point~$x$ is located between two corner sets, say $\overline{U^0}$ and $U^1$  in Figure~\ref{fig:buffers-and-insets}. Using the transformed process $(Y(t))_{t\geq 0}$ from Lemma~\ref{lem:Y}, we can say that $Y_1(0)\in (\ln r, -\ln r]$ and
      to reach $\Ssetin\cup\Rset$, it suffices for
   $Y_1$ to exit $(\ln r + \ln \rmoins, -\ln r - \ln\rmoins)$\,---\,provided that $r' \leq \rplus r$; see Figure~\ref{fig:insets}. The direction of the exit through the endpoints of this interval coincides with the direction of the drift. Thus, 
    we can apply Lemma~\ref{lem:lin-ub-different-extremes} to conclude that~$\ee^x \hit_{\Ssetin \cup \Rset} \leq C_1|{\ln \rmoins}|$. This argument is easily adapted to a starting point between $U^3$ and $\overline{U^0}$\,---\,provided that $r' \leq \rplus \rmoins^2 r$; see Figure~\ref{fig:insets}.
        
    Suppose the starting point~$x$ is located between a corner set and an edge set, say~$U^2$ and~$G^2$ in Figure~\ref{fig:buffers-and-insets}. We will use the process $(Y^2(t))_{t\geq 0}$ given by $Y^2(t)=y(1-X_1(t), 1-X_2(t))$.  
    To reach $\Ssetin\cup \Rset$, it suffices for
    $Y_1^2$ to increase by $2|{\ln \rmoins}|$\,---\,provided that $r' \leq \rplus \rmoins^2 r$; see Figure~\ref{fig:insets}.
    As long as the process has not reached $\Rset$ or $\Ssetin$ via some other route, 
    this increase is aligned with the drift. Thus, we can apply  
    Lemma~\ref{lem:ee-ap-bound} to conclude that $\ee^x \hit_{\Ssetin \cup \Rset} \leq2C_1|{\ln \rmoins}|$. This argument is easily adapted to a starting point between~$G^2$ and~$U^3$.
\end{proof}

\begin{proposition}
\label{prop:return-Rset-particular-d}
    For every $K>0$,
    there exists a choice of~$\rplus$ and $\rmoins$ with the property that the sets~$\Ssetin$ and $\Ssetout$ are such that 
    $\ee^x \hit_{\Ssetout^\complement}  \geq T(2K + 1)$ 
    for all $x \in \Ssetin$.
\end{proposition}

\begin{proof}
    In addition to $K$ and the constant $C_1$ identified above, we let $C_2$ be as in Corollary~\ref{cor:lin-lb-tauplus-prob} ($\nu = 2\max_{i,k} \lambda_i^k$, $A = \max_i \|s_i\|_\infty$) and $C_3$ as in Lemma~\ref{lem:exp-lb-suppt-case} ($\nu = \tfrac 12 \min_{i,k} |\lambda_i^k|$, $A = \max_i \|s_i\|_\infty$, $b= -\ln r$).
    Throughout the proof, we assume that $|{\ln \rmoins}|$ is greater than $L$ in Corollary~\ref{cor:drift-wins-over-mart-bound} and Lemma~\ref{lem:exp-lb-suppt-case} ($\nu = \min_{i,k}|\lambda_i^k|$, $A = \max_i \|s_i\|_\infty$, $b= -\ln r$). 
    There are two types of components of~$\Ssetin$: those in corners, and those along edges. We treat starting points in those separately.
    
    Suppose that the starting point~$x$ is located in a corner set, say the component of~$\Ssetin$ within $U^1$ in Figures~\ref{fig:buffers-and-insets}. We will use the process $(Y^1(t))_{t\geq 0}$ defined by $Y^1(t) = y(X_2(t),1-X_1(t))$. To exit $\Ssetout$, one of two things needs to happen: 
    either $Y_2^1$ increases by at least $|{\ln \rplus}|$ going along the drift, or $Y_1^1$ increases by at least $|{\ln \rmoins}|$ going against the drift.  
    Appealing to Corollary~\ref{cor:lin-lb-tauplus-prob} and Corollary~\ref{cor:drift-wins-over-mart-bound},
    we deduce that, with probability at least $\tfrac 12$, this takes at least $C_2^{-1} |{\ln \rplus}|$. This argument is easily adapted to a starting point in the component within~$U^2$ or the component within~$U^3$.
    
    Suppose that the starting point is located in the edge set that is the component of~$\Ssetin$ within $G^2_{r\rmoins, r}$ in Figure~\ref{fig:buffers-and-insets}. We will use the process $(Y^2(t))_{t \geq 0}$. To exit $\Ssetout$, one of two  things needs to happen: 
    either 
    $Y_2^2$ increases by $|{\ln \rplus}|$, possibly going along the drift, or $Y_1^2$ goes from being in $[\ln r, -\ln r]$ to being outside of $(\ln\rmoins + \ln r, -\ln \rmoins - \ln r)$, going against the drift. 
    By Corollary~\ref{cor:lin-lb-tauplus-prob} and Lemma~\ref{lem:exp-lb-suppt-case} and a union bound, we deduce that, with probability at least $\tfrac 12$, this takes at least $C_2^{-1} |{\ln \rplus}| \wedge C_3^{-1} \exp(C_3^{-1}|{\ln \rmoins}|)$. 

    Hence, 
    $$\ee^x \hit_{\Ssetout^\complement}  \geq \tfrac 12 C_2^{-1} |{\ln \rplus}| \wedge \tfrac 12 C_3^{-1} \exp(C_3^{-1}|{\ln \rmoins}|)$$ for all $x \in \Ssetin$.
    Therefore, it suffices to pick $\rmoins$ small enough that
    $
        \tfrac 12 C_3^{-1} \exp(C_3^{-1}|{\ln \rmoins}|) > (2K+1) 2C_1 |\ln\rmoins|,
    $
    and then
    $\rplus$ small enough that 
    $
        \tfrac 12 C_2^{-1} |{\ln \rplus}| > (2K+1) 2C_1 |\ln\rmoins|.
    $
    Referring back to the definition~\eqref{eq:T-choice} of the time~$T$, we have the desired bound.
\end{proof}

Now that we have shown that Proposition~\ref{prop:YM12} applies, and hence that the set $R$ is recurrent, we can consider the event
\[  
    {\convEvent_\Attr} = 
    \left\{ X(t) \to \Or^0  \text{ or } X(t) \to \Ed^1 \text{ as } t\to\infty \right\},
\] 
and the stopping times  $\tau_0 = 0$ and
\[ 
    \tau_n  = \{t \geq \tau_{n-1} + 1 : X(t) \in \Rset \}.
\]
Since the set $\Rset$ is recurrent, these stopping times are all almost surely finite, and the event ${\convEvent_\Attr}$ is measurable with respect to the $\sigma$-algebra~$\cF_\infty$ generated by $(\cF_{\tau_n})_{n=1}^\infty$. In view of~\eqref{eq:conv-lb-from-Rset} and the strong Markov property,  $\ee[\one_{\convEvent_\Attr}|\cF_{\tau_n}] \geq \delta$ for all $n$. 
Hence, the equality~\eqref{eq:particular-case-1-conclu} follows from Lemma~\ref{lem:levy-conseq}.

\begin{remark}
    While the intermediate step using the strong Markov property and~\eqref{eq:intermediate-trans} would be avoided by directly showing that $\overline{U_r^0} \cup \overline{G_r^1}$ is recurrent rather than showing that $R = \overline{U_r^0} \cup \overline{G_r^1} \cup (F^*_{r'})^\complement$ is recurrent, it turns out that the use of Proposition~\ref{prop:YM12} for the former is technically more tedious than for the latter. Moreover, using $R$ instead of $\overline{U_r^0} \cup \overline{G_r^1}$ parallels more closely what needs to de done in Section~\ref{ssec:proof-it-recurrent-case}. 
\end{remark}

\subsection{Proofs of cases~\ref{it:attractor-case} and~\ref{it:recurrent-case} of Theorem~\ref{thm:2d-trich}}
\label{ssec:proof-it-recurrent-case}

While we used a particular situation for illustration purposes,
the local definitions~\eqref{eq:local-def-ABC}--\eqref{eq:pre-gluing-constraints-ABC} and their extensions~\eqref{eq:Lyap-ext}, 
as well as the proofs of Lemmas~\ref{lem:pre-YM-i-particular} and~\ref{lem:return-Rset-particular-c} and Proposition~\ref{prop:return-Rset-particular-d}, 
apply to all situations that fall within case~\ref{it:attractor-case}, provided that one adapts the definition of~$\Rset$ and the proof of Lemma~\ref{lem:chooseing-betas-case-I} to whatever are the elements of~$\Attr$.

\medskip

We now turn to case~\ref{it:recurrent-case}, which we recall is the case where there are no attracting vertex or edge and there is no stable stochastic cycle. 
Our claim in this case is that there is a unique invariant probability measure with a density with respect to the 2-dimensional Lebesgue measure, to which the empirical measures converge  with probability~$1$.

As in case~\ref{it:attractor-case}, a key step is to prove positive recurrence to 
a set $\Rset$. However, now this set is simply {$\bXo\setminus F^*_{r'}$}.  We keep the local definitions~\eqref{eq:local-def-ABC} and~\eqref{eq:pre-gluing-constraints-ABC} and their extensions~\eqref{eq:Lyap-ext}. These definitions depend on parameters $({\beta}_k)_{k=0}^3$ or, equivalently, $(\tilde{\beta}_k)_{k=0}^3$ that we will choose (see below) to satisfy~\eqref{eq:decay-and-gluing-constraints-ABC} for all $k$. We also carry over the choice of parameters $\rmoins$ and $\rplus$ and sets $\Ssetout$ and $\Ssetin$ ensuring that we can apply Lemmas~\ref{lem:pre-YM-i-particular} and~\ref{lem:return-Rset-particular-c} and check the conditions of Theorem~\ref{prop:return-Rset-particular-d}  
guaranteeing
the following recurrence property of the compact set $\Rset = \bXo\setminus F^*_{r'}$: 
\begin{equation}
\label{eq:YM-goal-done}
    \ee^x \inf\{t \geq 0 : X(t) \in R \} \leq 3\Phi (x) + 2KT.
\end{equation} 
Before we make these choices precise depending on the situation, let us explain how~\eqref{eq:YM-goal-done} allows us to deduce case~\ref{it:recurrent-case} using a result of Meyn and Tweedie. To be more precise, special cases of Theorems~3.2.i, 4.1.i, 4.3.ii, 4.4 and 8.1.i in~\cite{MT93} can be combined to obtain Theorem~\ref{thm:extracted-from-MT} below. While its assumptions can be checked by adapting different arguments from the earlier literature, we will for simplicity make references to the textbook exposition in~\cite{Benaim-Hurth:MR4559704}.

\begin{theorem}[Meyn--Tweedie, 1993]
\label{thm:extracted-from-MT}
    Consider a Feller, time-homogeneous Markov process $(X_t)_{t\geq 0}$ on an open set $\bU \subset \rr^n$. Suppose the process has continuous paths and that the following two conditions hold:
    \begin{description}
        \item[Minorization] There exists a probability measure~$m$ with the property that, for every compact set $R \subset \bU$, there exists $\delta > 0$ such that 
        \begin{align*}
            \inf_{x \in R} \int_0^\infty \Exp{-t} \pp^x\{X(t) \in A\} \geq \delta m(A)
        \end{align*}
        for every Borel set $A \subset \bU$.
        \item[Strong recurrence] There exists a compact set $R$ and a positive time~$s>0$ such that
        \begin{align*}
            \ee^x \inf\{ t \geq s : X(t) \in R \} < \infty
        \end{align*}
        for every $x\in\bU$.
    \end{description}
    Then, there exists a probability measure~$\invm$ on~$\bU$ such that $m \ll \invm$ and, for every $x \in \bU$, with probability~$1$, we have
    $  
        \empi{x}_t \to \invm
    $
    as $t\to\infty$.
\end{theorem}
    
By Assumption~\ref{req:Hor}, there a is a neighborhood~$V$ of~$x_*$, 
a time~$t_0>0$, a Borel probability measure~$m$ and a constant $\alpha > 0$ such that 
\begin{align*}
    \inf_{x\in V}\pp^x\{X(t_0)\in A\} \geq \alpha m(A)
\end{align*}
for every Borel measurable set $A \subset \bXo$; see, e.g., Theorem~6.37 in~\cite{Benaim-Hurth:MR4559704}.
But then, for every such~$A$ and every point $x\in\bXo$, the Markov property yields
\begin{align*}
    \int_0^\infty \Exp{-t} \pp^x\{X(t) \in A\} \dd t
    &\geq \int_{t_0}^\infty \Exp{-t} \pp^x\{X(t) \in A\} \dd t \\
    &\geq \int_0^\infty \Exp{-(t+t_0)} \pp^x\{X(t_0+t) \in A\} \dd t \\
    &\geq \int_0^\infty \Exp{-(t+t_0)} \pp^x\{X(t) \in V\}\inf_{x' \in V} \pp^{x'}\{X(t_0) \in A\} \dd t \\
    &\geq \left(\int_0^\infty \Exp{-t} \pp^x\{X(t) \in V\}\dd t\right) \Exp{-t_0}\alpha m(A).
\end{align*}
By the Feller property and Fatou's lemma, the dependence in $x$ of the integral in parentheses is lower semicontinuous. By the Feller property and Assumption~\ref{req:irreducibility}, this integral is positive for every~$x$.
This provides the minorization condition in Theorem~\ref{thm:extracted-from-MT}.

For the choice of the compact set~$R = \bXo\setminus F^*_{r'}$, by the Markov property,~the recurrence property~\eqref{eq:YM-goal-done} and Corollary~\ref{cor:Dynkin}, we have
\begin{align*}
    \sup_{x \in R} \ee^x \inf\{ t \geq s : X(t) \in R\} 
    &\leq \sup_{x \in R} \ee^x[s + 3\Phi(X(s)) + 2KT] \\
    &\leq  s + 3\left(\sup_{x \in R} \Phi(x) + s \right) + 2KT
\end{align*}
for every fixed $s \geq 0$. Since $\Phi$ is continuous, this provides the strong recurrence property in Theorem~\ref{thm:extracted-from-MT}.

With the measure~$\invm$ from the conclusion of Theorem~\ref{thm:extracted-from-MT} at hand, the absolute continuity property claimed in case~\ref{it:recurrent-case} of Theorem~\ref{thm:2d-trich} follows from Assumptions~\ref{req:irreducibility} and~\ref{req:Hor} by, e.g., Theorem~6.34.ii in~\cite{Benaim-Hurth:MR4559704}.

Let us now show how to choose $(\tilde{\beta}_k)_{k=0}^3$ so that~\eqref{eq:decay-and-gluing-constraints-ABC} holds for all $k$.
The strategy depends on whether vertices of~$\bX$ form an unstable stochastic cycle 
(no other stochastic cycles are allowed by the definition of case~\ref{it:recurrent-case} and Assumption~\ref{req:index-stab-ne-1}).

In the stochastic cycle case,  without loss of generality, we assume that $\eivalv^{k}>0>\eivalh^{k}$
for all $k$. Requirement~\eqref{eq:decay-and-gluing-constraints-ABC} can then be rewritten as
\begin{equation}
    \label{eq:betas-rewrite}
        \tilde{\beta}_{k}>\tilde{\beta}_{k-1}\rho^k. 
\end{equation} 
So, we define $\beta_0=1$, and then, recursively,
$
\tilde{\beta}_k = (\rho^k+\eps)\tilde{\beta}_{k-1}
$ for some small parameter~$\eps > 0$.
Requirement~\eqref{eq:betas-rewrite} clearly holds for $k=1,2,3$ and for all~$\eps>0$. For $k=0$, it can be written as
\[
1> \bigg(\prod_{j=1}^3 (\rho^j+\eps)\bigg) \rho^0.
\]
Since the cycle is unstable, i.e., $\Index<1$,
this inequality holds for $\eps$ small enough.

If the vertices do not form a cycle, it is convenient to define an auxiliary directed graph $\Gamma$ treating the edges $E_k$, $k\in\zz_4$, as nodes of $\Gamma$. If $\eivalh^{k}>0$ and $\eivalv^{k}<0$, we connect $E_{k-1}$ to $E_k$ by an arc (arcs are directed).  If $\eivalh^{k}<0$ and $\eivalv^{k}>0$, we connect $E_{k}$ to $E_{k-1}$ by an arc. In other words, if two edges meet at a saddle point, we connect them by an arc whose orientation coincides with the orientation of the saddle. 
We denote the set of edges $E_k$ with no incoming arcs by~$\rootEd$. Since there is no cycle, $\rootEd \ne\emptyset$. For each $E_k\in \rootEd $, we set $\tilde{\beta}_k=1$, and then we sequentially assign values of $\tilde{\beta}_k$ to all the remaining edges following directed paths in $\Gamma$ and making sure~\eqref{eq:decay-and-gluing-constraints-ABC} holds.

\section{Discussion}\label{sec:discussion}

The goal of this section is to discuss possible extensions of our theory and connections to the existing results.

\subsection{Higher dimensions}

\subsubsection{General setup and basic properties}

We begin with the extensions to arbitrary dimensions. We will assume that $\X$ is the closed unit  cube $[0,1]^d$ for some $d\in\N$ although most of the statements below apply to sets diffeomorphic to convex polytopes which may be viewed as cell complexes consisting of faces or cells of various dimensions. For the cube, an $n$-dimensional face is an $n$-dimensional cube obtained by restricting $d-n$ coordinates to be $0$ or $1$ each, and letting each of the remaining $n$ coordinates vary in $(0,1)$.  The extreme cases are $\X^\circ=(0,1)^d$ which can be viewed as a (unique) $d$-dimensional face and vertices of the cube which can be viewed as $0$-dimensional faces. The collection of all faces is denoted by $\Faces$.

The diffusion on $\X$ is given by
\begin{equation}
    \label{eq:multd-SDE}
\dd X(t)=b(X(t)) \dd t+\sum_{i=1}^m \sigma_i(X(t))\circ \dd W_i(t),
\end{equation}
where we assume that $m\in\N$, $W_1,\ldots,W_m$ are independent standard  Brownian motions and the vector fields~$b$ and $\sigma_i,$ $i=1,\ldots,m,$ belong to $C^\infty(U)$ for an open set $U\supset \X$.  

It is crucial to describe the behavior of the vector fields on the boundary.
We assume that for each $F\in \Faces$, the restrictions of $b$ and $\sigma_i$, $i=1,\ldots,m$ to $F$ are tangent to~$F$. This implies varying (from face to face) degree of degeneracy of the noise and 
ensures that 
the SDE~\eqref{eq:multd-SDE} equipped with an initial condition 
$X(0)=x\in\X$ has a unique $\X$-valued strong solution $(X^x_\omega(t))_{t \geq 0}$. Each $F\in\Faces$ (in particular, $\X^\circ$) is invariant for the resulting
(strong) Markov process. 

We will also need an assumption ensuring that there is at most one invariant measure concentrated one each face and, if a face supports such a measure, the transition probabilities and empirical measures converge to it. This may be done similarly to the present paper via an irreducibility condition on each face and relative H\"ormander's condition at one point per face. For simplicity, we could make a stronger relative ellipticity assumption and require that if $x\in F$, then  $\sigma_1(x),\ldots,\sigma_m(x)$ span the tangent space $T_x F$.  

Basic properties of the long-term behavior discussed in Section~\ref{sec:main-results} stated for $d=1$ and $d=2$ still hold in this setting for arbitrary~$d$. The interior of the cube may be recurrent or transient. With probability $1$, the set of limit points is a union of faces (perhaps the entire cube),
and every weak limit point of empirical measures is a mixture of ergodic invariant distributions supported by some of these faces. These measures are {\it observable}: they have nontrivial contribution to the long-term statistics of the diffusion with positive probability. 
Scenarios~\ref{scen:uniqe-erg-attr}--\ref{scen:nonerg-attr} in
Section~\ref{sec:summary-result}, i.e., the existence of unique ergodic attractors,
random ergodic attractors, and unique nonergodic attractors,
can all occur in higher dimensions, and ergodic attractors
can be more general; see the discussion in Section~\ref{sec:attracting-faces}.
Random ergodic and nonergodic attractors can occur in dimension $d\ge 3$ (though
not in $d < 3$); examples are given in Section~\ref{sec:stoch-cycles-high-d}. We do not
know if this is a complete classification of diffusion
dynamics under the kind of nondegeneracy and hyperbolicity
assumptions imposed in this paper.

\subsubsection{Attracting faces}
\label{sec:attracting-faces}

Let us briefly discuss the case where {\it attracting faces} resulting in relatively simple behavior are present. 
For that, we need to introduce transversal Lyapunov exponents for each face $F$ supporting a (unique) invariant probability measure~$\invm_F$. For simplicity, here we describe this notion only for the case where
the noise is diagonal, i.e., $m=d$ and  $\sigma_i(x)$ is parallel to the $i$-th coordinate axis for $i=1,\ldots,d$ in \eqref{eq:multd-SDE}.
For an $n$-dimensional face~$F=(0,1)^{n}\times\{0\}^{d-n}$, we define, similarly to \eqref{eq:average-repulsion-over-edge-diag},
\[
\bar \Lambda_i (F)=\int_{F}\partial_i b_i(x)\,\invm_F(\dd x),\quad i=d-n+1,\ldots, d.
\]  
If the noise is not diagonal, the integrand in this formula 
must be corrected by contributions from the noise involving mixed partial derivatives of diffusion coefficients, similarly to~\eqref{eq:new-def-oldlambda}.
If $\bar \Lambda_i(F)<0$ for all transversal directions $i=n-k+1,\ldots, n$, then we call $F$ an attracting face. 
Similar definitions serve other faces.
\begin{conjecture} Under broad conditions
    guaranteeing that there is at most one invariant measure concentrated on each face,  for each 
    attracting $F\in\Faces$, the following holds: for all $x\in\bXo$, with positive probability, $\bar F$ is the set of limit points of $X^x_\omega(t)$ as $t\to\infty$ and the empirical measures $\empi{x}_{\omega,t}$ converge to $\invm_F$ as $t\to\infty$.
\end{conjecture}

For the case where  $F$ is a vertex of the cube,
this conjecture is easy to prove with the methods of this paper. However, translating the analysis of the function we denoted by $\psi$ to more general situations requires more work.

\subsubsection{Stochastic cycles in higher dimensions}\label{sec:stoch-cycles-high-d}.

For higher dimensional diffusions the asymptotic scenarios are not restricted to types A through~C described in this paper for dimensions $1$ and $2$. 
In this section, we discuss two examples of diffusions in the $3$-dimensional unit cube. In the first one, instead of a unique nonergodic attractor described in the $2$-dimensional scenario 
\ref{scen:nonerg-attr}, there are 
two of them, and one is chosen by the diffusion at random. In the second example, the nonergodic attractor is a union of four squares, and there are four observable measures each supported by an edge.

These examples can be introduced as product systems, where the ``base'' process $(X_1,X_2)$ is a diffusion on the square $[0,1]^2$ and the 
``height'' process $X_3$ is an independent  diffusion on $[0,1]$. 

To construct the first example, we suppose that $(X_1,X_2)$ has the stable stochastic cycle behavior (scenario \ref{scen:nonerg-attr}), and~$X_3$ is
a process on $[0,1]$ with both attracting endpoints.
The product system exhibits two nonergodic attractors, one consisting of all the edges of the square $(0,1)^2\times\{0\}$, another 
consisting of all the edges of the square $(0,1)^2\times\{1\}$. The diffusion chooses one of these attractors at random and converges to it.
Conditioned on convergence to one of the attractors,
four Dirac masses at its vertices are observable and contribute to the long-term statistical description.
This example provides a new scenario that combines features of scenarios~\ref{scen:random-erg-attr} and \ref{scen:nonerg-attr} and does not reduce to any of those. It is also easy to give similar examples (lacking the product structure) with, say, a stable stochastic cycle on one face of the cube and one or more attracting edges or vertices on the opposite face. 

Now let us modify this example keeping the behavior for the base process $(X_1,X_2)$ and assuming that both endpoints of $[0,1]$ are repelling for the height process $X_3$ implying that $X_3$ has an invariant probability measure~$\invm$ supported on the entire $[0,1]$.  In this case, for each vertex $\Or^k$, $k=0,1,2,3$ of the square, a unique invariant probability measure $\delta_{\Or^k}\times \invm$ is concentrated on the vertical edge $\Or^k\times(0,1)$, and the set of observable ergodic measures is composed of these four measures. 
During long epochs when the diffusion dwells near one of these edges, its statistics are governed by the invariant measure concentrated on that edge. During the transitions between edges, the process may touch any point on the square spanned by these edges. As a result, the only attractor is the closure of the union of these four squares. 

\medskip

Not all situations in high dimensions are this simple. In contrast with the concise classification that this paper provides in dimensions $1$ and $2$, the geometry of 
attractors and observable measures may be quite involved in higher dimensions. Attractors may overlap, i.e., one face may belong to more than one attractor. An attractor may be the closure of a union of faces of varying dimension. Also an attractor does not necessarily have to be cyclic; instead it may correspond to a scenario where the diffusion makes multiple random decisions. 
These features will be addressed in the future work.

\subsection{Nonhyperbolic cases}

In this paper, we have made several hyperbolicity assumptions. If they do not hold, new  kinds of behavior emerge.  This happens if, in dimension $1$, we allow the coefficients of linearization of the drift near endpoints to be zero. For example, suppose $W$ is a standard Brownian motion, and $x:\R\to(-1,1)$ is an odd increasing function 
such that 
$x(y)=1-\Exp{-y}$ for large $y>0$, $x(y)=-1+\Exp{y}$ for large $y<0$ and $x(0)=0$. 
Then, for every initial condition $y_0$, the process $X(t)=x(y_0+W(t))$ satisfies the Stratonovich equation~\eqref{eq:1d-SDE}
with drift $b(x)\equiv 0$:
\begin{equation}
    dX(t)=\sigma(X(t))\circ \dd W(t),
\end{equation}    
where $\sigma(x)=x'(y(x))$. Here $y:(-1,1)\to\R$ is the inverse of $x(\cdot).$ In particular, $\sigma(x)=1+x$ near $x=-1$ and $\sigma(x)=1-x$ near~$1$.
One can view the resulting system as a diffusion on $[-1,1]$. The invariant endpoints $\pm1$ correspond to $\pm\infty$ for the Brownian motion $W$.

Clearly, the linearization coefficients of the drift near $\pm1$ are zero. 

\begin{theorem}
    For the process described above and an arbitrary initial condition $x\in(-1,1)$,
    the set of weak limit points of empirical measures $\empi{x}_t$ consists of $\delta_{-1}$, $\delta_1$, and all their mixtures.
\end{theorem}

\begin{proof} We will argue in terms of the underlying Brownian Motion.
    Due to the null recurrence property of the Brownian motion, the only ergodic invariant measures of the diffusion on $[-1,1]$ are~$\delta_{\pm1}$.
    
    The process $(\empi{x}_t)_{t>0}$ is continuous in the space of probability measures on $[-1,1]$ equipped with the topology of weak convergence. Hence, due to compactness of this space, the set of weak limit points of~$(\empi{x}_t)_{t>0}$ as $t\to\infty$ is connected. Therefore, by Theorem~\ref{thm:benaim}, it suffices to show that both $\delta_{\pm1}$ are limit points with probability $1$. By symmetry, it suffices to prove this just for $\delta_1$. So we need to prove that      
    for every $r>0$ and $\eps>0$, with probability $1$, there is a sequence of times $\tau_k\to\infty$ such that the fraction of time spent by $y_0+W$ on $[r,+\infty)$ on the time interval $[0,\tau_k]$ is at least $1-\eps$.
    Due to recurrence and the strong Markov property, it suffices to prove this property instead for the process $r+W$.  Subtracting $r$, we can simplify this further: it suffices to prove that for every $\eps>0$,  
    \begin{equation}
        \label{eq:B_eps-prob_1}
        P(B_\eps)=1,
    \end{equation}    
    where
    \[
    B_\eps= \Big\{H_n\ge 1-\eps\quad \text{for infinitely many}\ n\in\N\Big\}.
    \]
    and $H_n$ denotes  
    the fraction of time spent by $W$ on $(0,\infty)$ during the time interval $[0,n]$. According to 
    L\'evy's Arcsine Law, for every $n$, the r.v.~$H_n$ has the same (arcsine) distribution, with positive density on $(0,1)$. Thus, there is $\delta>0$ such that 
    \begin{equation}
        \label{eq:positive-prob-of-closeness-to-delta}
    \pp(H_n\ge 1-\eps)= \delta,\quad n\in\N.
    \end{equation}
    It follows that 
    \begin{equation}
    \label{eq:i.o.-positive} 
        P(B_\eps)>0.
    \end{equation}
    Otherwise, we would have
    \[
        \pp\Big(\bigcap_{m=1}^\infty\bigcup_{n=m}^\infty \{H_n\ge 1-\eps\}\Big)=0,
    \]
    which implies $\pp\big(\bigcup_{n=m}^\infty \{H_n\ge 1-\eps\}\big)\to 0$ as $m\to\infty$ and thus contradicts~\eqref{eq:positive-prob-of-closeness-to-delta}. 
    Since $B_\eps$ is a tail event for the Brownian motion,  \eqref{eq:B_eps-prob_1} follows from \eqref{eq:i.o.-positive}, and the proof is completed.
\end{proof}    

The product of several such systems, where components of $X=(X_1,X_2,\ldots,X_d)$ are independent null-recurrent processes described above, is also null recurrent. It intermittently visits the corners 
of $[0,1]^d$, where it dwells for long random times. The result is that the set of all limit points of the empirical distribution is all mixtures of Dirac measures at the vertices. Note that although most transitions between the corners are made along the edges, this situation is different from the heteroclinic cycling studied in Theorem~\ref{thm:2d-cycle}. In fact, due to the null recurrence of the $2$-dimensional Brownian motion and transience of Brownian motion in higher dimensions, with probability~$1$, the attractor is the closure of the union of all $2$-dimensional faces of the cube.

Another family of interesting examples arises if one mixes the hyperbolic and neutral behavior, for example, when one assumes that only some of the transversal Lyapunov exponents for higher-dimensional faces are 0, or if the hyperbolicity Assumption~\ref{req:index-stab-ne-1} is broken for a stochastic cycle.

\subsection{Connections with the existing work}¸

Our work is closely related to~\cite{FK22,FK23}, where the authors consider\,---\,among other things\,---\,a situation similar to ours, but where the invariant sets on which the diffusion degenerates (boundary components) are pairwise disjoint smooth manifolds without boundary. Some elements of our analysis are adapted from these papers. However, the presence of lower-dimensional faces for our boundaries 
considerably enriches the class of possible long-term behaviors of empirical measures: for example, no heteroclinic behavior (scenario~\ref{scen:nonerg-attr}) is possible in~\cite{FK22,FK23}. We also need more detailed analysis near boundaries of invariant manifolds (invariant sets of lower dimensions).

\medskip 

Our work also naturally connects to the works~\cite{Bepreprint,FSpreprint} on the use of average Lyapunov functions in the context stochastic persistence and stochastic extinction for more general processes. Specializing their more general treatment of the topic to our setup, a key step is the analysis of a function $\Phi : \bXo \to \rr$ with the property that $\cL\Phi$ continuously extends to a function $H : \bX \to \rr$ whose integral with respect to every dynamically invariant probability measure concentrated on~$\partial\bX$ has a definite sign: negativity leads to persistence and positivity leads to extinction. While we focused on the sign of $\cL\Phi$ outside of regions that are quickly escaped rather than the integrals of $\cL\Phi$ against invariant measures, some statements about the latter can be extracted from our analysis of the former.
In our setup, {once the Lyapunov function is appropriately constructed}\,---\,as is done in Section~\ref{sec:proofs-Lyap} of the present manuscript \,---\,one could apply the techniques of Sections~6 and~7 in~\cite{FSpreprint} to prove extinction. However we are able to provide more detailed information about the behavior of the empirical measures.

\medskip

The analysis in the present paper allows to describe, under minimal assumptions, the behavior of {\it scalable} 
(see~\cite{LKYJW20}) 3-dimensional diffusions in the positive orthant and, specifically, the growth rates. The transition probabilities of those systems satisfy the relation 
\[
    \pp^x\{X(t)\in A\} = \pp^{cx}\{X(t)\in cA\}
\] 
for scaling factor $c>0$, every point~$x$ and every Borel set~$A$. Here $cA=\{cy:\ y\in A\}$. This allows to view the evolution as a two-component fibered system,
where the radial component is governed by the ``angle'' component (given by the projection, or normalization, removing the radial component) evolving in the unit 
2-dimensional simplex/triangle.  Our analysis directly applies to the dynamics on the simplex, and it means that any of the scenarios~\ref{scen:uniqe-erg-attr} through~\ref{scen:nonerg-attr} can happen. In scenario~\ref{scen:uniqe-erg-attr}, where the statistics are governed by a single invariant measure on the simplex, the radial component possesses a (almost surely) deterministic exponential growth rate given by the average of the instantaneous growth rate over the simplex with respect to the invariant measure. This is the situation studied in~\cite{LKYJW20}. In scenario~\ref{scen:random-erg-attr}, the system chooses an ergodic measure on the simplex boundary at random. The exponential growth rate is random in this case, since it is  given by averaging instantaneous rates with respect to the randomly chosen invariant measure. Finally, under scenario~\ref{scen:nonerg-attr}, the system undergoes stochastic cycling near the boundary of the simplex, and no single exponential growth rate of the radial component is well defined. Instead, there is a sequence of growing long epochs of dwelling near simplex vertices. During each epoch,  the effective growth rate equals the instantaneous growth rate at the respective vertex of the simplex.

In higher dimensions,
scenarios similar to~\ref{scen:uniqe-erg-attr} through~\ref{scen:nonerg-attr} and more complex ones are possible. In particular, it is possible to have multiple random nonergodic attractors on the simplex boundary, which corresponds to having a random family of growth rates that the system eventually chooses and starts going through at a decreasing speed. The growth rate at each epoch is given by the average of instantaneous growth rates over the ergodic measure governing the evolution over that epoch.   

\medskip 

In \cite{Hening-Nguyen:MR3809480,HNS22},
the authors provide a classification of the long-term behavior of
stochastic ecological systems (generalized Lotka--Volterra systems)
in unbounded 3-dimensional domains. Under the competitiveness assumption
they imposed on the dynamics together with the relative
nondegeneracy of the diffusion, this situation is, in essence,
that of 2-dimensional systems on bounded domains insofar as invariant
measures are concerned. In broad outline, the classification
results of our paper are similar to those in \cite{HNS22}.
Three major differences are: (i) our 
irreducibility and one-point hypoellipticity assumptions on the noise are very general;
(ii) our results give detailed, quantitative
information on the asymptotic dynamics, and (iii) our proofs
are different. Here we mention in particular our formulation
of an extension of the Foster--Lyapunov function, a general
result that we believe will be useful elsewhere.
We also note that the papers~\cite{SBA11,Hening-Nguyen:MR3809480,HNC21} contain\,---\,among other things\,---\,sufficient conditions for persistence and extinction in setting similar to that of~\cite{Hening-Nguyen:MR3809480,HNS22}, but in possibly higher dimension.

\medskip 

Speaking of systems with stochastic cycling, we must mention that their deterministic counterparts, systems with stable heteroclinic cycles, have been broadly studied. 
In fact, the failure of convergence and our description of the set of limit points for the empirical measures in Proposition~\ref{prop:quadri} can be viewed as a stochastic counterpart to known results about heteroclinic attractors in the deterministic setting; see, e.g.,~\cite{Ga92,Ta94}. Note however that we are not assuming that the deterministic part of our stochastic dynamics provides heteroclinic connections between the saddles.

Heteroclinic cycling and statistics of transitions in the presence of
elliptic (nondegenerate) white noise of vanishing amplitude has been studied in \cite{Bakhtin2010:MR2731621,Bakhtin2011,Rare-Heteroclinic}. There is a common theme between all these works and the classical metastability theory~\cite{FW}, namely, short transitions between long epochs of dwelling near metastable states.

\appendix

\section{Exit time estimates}
\label{app:martingale}

\subsection{Exit times estimates for one-dimensional processes}
\label{sec:1d-martingales}

Throughout the paper, it is often useful to compare various processes to a process of  the form 
\begin{equation}
\label{eq:drift-plus-mart}
    Z(t) = \nu t + M(t)
\end{equation}
where $\nu$ is a constant and $M = (M(t))_{t\ge 0}$  
belongs to $\cM$, a class of martingales we describe below. We work on 
a filtered probability space~$(\Omega, \cF, (\cF_t)_{t\ge 0},\pp)$ satisfying the usual conditions. We will write $M\in\cM$ if $M$ is a
continuous square-integrable martingale with respect to $(\cF_t)_{t\geq 0}$ satisfying $M(0)=0$ and whose quadratic variation $(\langle M\rangle_t)_{t\ge0}$ admits a deterministic Lipschitz constant, i.e., 
there is $A > 0$ such that, with probability 1, $\langle M\rangle _t - \langle M\rangle_s \leq A(t-s)$ for all $t \geq s \geq 0$.

We collect several facts concerning such processes. Some of these estimates depend on the Lipschitz constant $A$ introduced above.
We begin with the standard exponential martingale inequality for $M^*(t)=\sup\{|M(s)|:\ s\in[0,t]\}$; see, e.g., \cite[Section I.8]{Bass:MR1329542}.
\begin{theorem}
\label{thm:Bass-martingale}
   Let $t>0$, $z \geq 0$, $D>0$, and assume that $\langle M\rangle_t\le D$ with probability 1. Then,
    \begin{align*}
        \pp\{M^*(t) > z\} \leq 2\exp\left(-\frac{z^2}{2 D}\right),
    \end{align*} 
\end{theorem}
For $\zplus>0$ and $\zmoins<0$, we introduce the following stopping times: 
$\tauplus  = \inf\{t > 0 : Z(t) \geq \zplus\}$, $\taumoins  = \inf\{t > 0 : Z(t) \leq \zmoins\}$, and $\tauwtv  = \taumoins \wedge \tauplus$. 

\begin{lemma}
\label{lem:lb-tauplus-prob}
    If $\nu > 0$, then for every $\varkappa \in (0,1)$,
    \begin{align*}
        \pp\left\{\tauplus \leq \frac{\varkappa\zplus}{\nu}\right\} 
        &\leq 
            2\exp\left(-\frac{\nu(1-\varkappa)^2\zplus}{2A\varkappa}\right).
    \end{align*}    
\end{lemma}

\begin{proof}
    For such an event to happen, we need $\nu t + M(t) \geq \zplus$ for some $t \leq \tfrac{\varkappa\zplus}{\nu}$. In particular, 
    \begin{align*}
        M^*\left(\frac{\varkappa\zplus}{\nu}\right) &\geq \zplus - \nu t 
        \geq (1-\varkappa)\zplus,
    \end{align*}
   and the lemma follows from the exponential martingale inequality.
\end{proof}

\begin{corollary}
 \label{cor:lin-lb-tauplus-prob}
    Suppose $\nu > 0$. Then, there is $C=C(A,\nu)>0$  such that
    $
        \pp\left\{ \tauplus \leq C^{-1}\zplus \right\} \leq \tfrac 14
    $ 
    for all~$\zplus \geq 1$.
    The same estimate holds for all processes $Y \leq Z$.
\end{corollary}

\begin{lemma}
\label{lem:ub-tauplus-prob}
    If $\nu > 0$, then  for every $\varkappa \in (0,1)$,
    \begin{align*}
        \pp\left\{\tauplus > \frac{\zplus}{\varkappa\nu}\right\} 
        &\leq 2\exp\left(-\frac{\nu \left(1 - \varkappa \right)^2 \zplus}{2A\varkappa}\right).
    \end{align*}  
\end{lemma}

\begin{proof}
    For such an event to happen, we need $\nu t + M(t) < \zplus$ for all $t \leq \tfrac{\zplus}{\varkappa\nu}$. In particular, we need
    \begin{align*}
        -M\left(\frac{\zplus}{\varkappa\nu}\right) 
        &> \frac{\zplus}{\varkappa\nu}\nu - \zplus 
        = \left(\frac{1}{\varkappa} - 1\right) \zplus,
    \end{align*}
    and the lemma follows from the exponential martingale inequality.
\end{proof} 

\begin{corollary}
\label{cor:new-exp-moment}
    If $\nu > 0$, then $\tauplus$ is almost surely finite, and there exists $\alpha = \alpha(\nu,A) > 0$ such that $\ee[\Exp{\alpha \tauplus}] \leq \Exp{\frac{2\alpha\zplus}{\nu}} + 2$ for all $\zplus > 0$.
\end{corollary}
\begin{corollary}
\label{cor:finite-mean}
    If $\nu > 0$, then $\tauwtv$ is almost surely finite and has finite expectation.
    Moreover, under the same conditions,  Doob's optional stopping theorem  (see Theorem 2.13 in \cite[Section~{2.2}]{EK}) applies to~$M$ and~$\tauwtv$, and $(M(\tauwtv+s))_{s>0}$ is $(\cF_{\tauwtv+s})_{s>0}$-martingale.
\end{corollary}

\begin{lemma}
\label{lem:ee-ap-bound}
    If $\nu > 0$, then $\ee \tauplus \leq C \zplus$ for all~$\zplus > 0$, with $C = \nu^{-1}$. The same estimate holds for all processes $Y \geq Z$.
\end{lemma}
\begin{proof} For every $T>0$, the optional stopping theorem implies 
    \begin{align*}
        \zplus\ge \ee[Z(\tauplus \wedge T)] &= \ee [\nu (\tauplus \wedge T) + M(\tauplus\wedge T)]
        = \nu \ee [\tauplus \wedge T].
    \end{align*}
    Due to Corollary~\ref{cor:finite-mean}, we can complete the proof of the Lemma applying the dominated convergence theorem to the right-hand side of this display.
\end{proof}

\begin{lemma}
\label{lem:tauwtv-bias}
    Let $\nu > 0$. For every $z_0>0$, there is $p=p(\nu,A,z_0)>1/2$ such that if $\zplus\ge z_0$ and  $\zmoins \leq -\zplus$, then
    \begin{align*}
        \pp\{\tauwtv = \tauplus\} \geq p.
    \end{align*} 
\end{lemma}

\begin{proof}
    Corollary~\ref{cor:finite-mean} and our assumption  on $\zmoins$ and  $\zplus$ imply
    \begin{align*}
        \nu \ee \tauwtv=\ee[Z(\tauwtv)] &= \zplus\pp\{\tauwtv = \tauplus\} + \zmoins\pp\{\tauwtv = \taumoins\} 
        \leq 
        2\zplus\pp\{\tauwtv = \tauplus\} - \zplus.
    \end{align*}
    Rearranging and applying Markov's inequality, we obtain, for any $\varkappa>0$,
    \begin{align*}
        \pp\{\tauwtv = \tauplus\}
            \,\geq\, \frac{1}{2} + \frac{\nu \ee\tauwtv}{2\zplus}\, \ge\, \frac{1}{2} + 
            \frac{1}{4\varkappa}
           \pp \Big\{\tauwtv \geq \frac{ \varkappa\zplus}{2\nu} \Big\} \,
           =\, \frac{1}{2} + \frac{1}{4\varkappa}\Big(1-\pp \Big\{\tauwtv < \frac{ \varkappa\zplus}{2\nu} \Big\} \Big).
    \end{align*} 
    The next step is similar to the proof of Lemma~\ref{lem:lb-tauplus-prob}.
    If $\varkappa\in(0,1)$, then  the event on the r.-h.s.\ implies $M^*\left(\frac{\varkappa\zplus}{2\nu}\right)>(1-\varkappa)\zplus$. Using the exponential martingale inequality, we obtain
    \[
        \pp \Big\{\tauwtv < \frac{ \varkappa\zplus}{2\nu} \Big\}\le 2\exp\left(-\frac{\nu(1-\varkappa)^2z_0^2}{2A\varkappa}\right).
    \]
    Choosing $\varkappa$ small enough to guarantee $\exp\left(-\frac{\nu(1-\varkappa)^2z_0}{2A\varkappa}\right)<1/4$, we complete the proof.
\end{proof} 

\begin{lemma}
\label{lem:drift-wins-over-mart-bound}
    If $\nu < 0$, then for every $\zplus_0>0$, there is a constant $C=C(\nu,A,\zplus_0)>0$  such that 
    \[
        \pp\{ \tauplus < \infty\} \leq C \exp \left( -C^{-1} {\zplus}\right)
    \]
    for all~$\zplus \geq \zplus_0$.
    The same holds for all processes $Y$ satisfying $Y\leq Z$.
\end{lemma}

\begin{proof}
    First, we write
    \begin{align*}
        \pp\{\tau^*<\infty\}
        &\le \sum_{k=0}^\infty \pp\{\tau^*\in[k,k+1]\} 
        \le \sum_{k=0}^{\infty}\pp\{ M^*(k+1)> \zplus-\nu k\}.
    \end{align*}
    Then, we note that, by Theorem~\ref{thm:Bass-martingale} and elementary inequalities,
    \begin{align*}
        \pp\{ M^*(k+1)> \zplus-\nu k\} &\le  2 
        \exp\left(-\frac{(\zplus-\nu k)^2}{2A(k+1)}\right)
        \le  2 
        \exp\left(-\frac{(\zplus)^2}{2A}\right)\exp\left(-\frac{\nu^2k}{4A}\right),
    \end{align*}
    so
    \begin{align*}
        \pp\{\tau^*<\infty\} 
        &\leq 2 
        \exp\left(-\frac{(\zplus)^2}{2A}\right) \sum_{k=0}^\infty \exp\left(-\frac{\nu^2k}{4A}\right),
    \end{align*}
    and the lemma follows.
\end{proof}

\begin{corollary}
\label{cor:drift-wins-over-mart-bound}
    If $\nu < 0$, then there exists $L = L(\nu,A)$  such that $\pp\{ \tauplus < \infty\} \leq \tfrac 14$
    for all~$\zplus \geq L$.
    The same holds for all processes $Y$ satisfying $Y\leq Z$.
\end{corollary}

\begin{corollary}
\label{cor:drift-wins-over-mart-div}
    If $\nu < 0$, then
    \[
        \pp\left\{ \lim_{t\to\infty} Z(t) = -\infty \right\} = 1.
    \]
    The same holds for all processes $Y$ satisfying $Y\leq Z$.
\end{corollary}

\begin{lemma}
\label{lem:positive-prob-to-escape-right}
    If, in addition to our standing assumptions, 
    there exists $a>0$ such that $$\langle M \rangle_t - \langle M \rangle_s \geq a(t-s)$$ for all $0 \leq s \leq t$,
    then
    there is a constant $c=c(A,a,\nu,\zplus,\zmoins) > 0$ such that
    \begin{align*}
        \pp\{\tau = \tauplus \text{ and } \tauplus  \leq 1\} \geq c.
    \end{align*}
    The same holds for all processes $Y$ satisfying $Y\geq Z$.
\end{lemma}

\begin{proof}
    First, we represent $M(t) = B(\langle M \rangle_t)$ for some standard Brownian motion $B$; see, e.g., \cite[Section~{3.4.B}]{KaSh}. Introducing 
    \[   
        \gamma(s) = \frac{\zplus - \zmoins}{a}s+ \frac{|\nu|}{a}s,\quad s \in [0,A],
    \]
    we define an event $D=D(A,a,\nu,\zplus,\zmoins)=\big\{|B(s)-\gamma(s)|\le \frac{1}{2}\zplus \ \text{for}\ s\in [0,A] \big\}$ and 
    $c=\pp(D)>0$.
    Since $at \leq \langle M \rangle_t \leq At$, on $D$ we have that for all~$t \in [0,1]$, 
    \begin{align*}
        M(t) &= B(\langle M \rangle_t) 
            > \frac{\zplus - \zmoins}{a}\langle M \rangle_t + \frac{|\nu|}{a}\langle M \rangle_t  + \frac 12 \zmoins 
            \geq (\zplus-\zmoins) t +  |\nu| t  + \frac 12 \zmoins
    \end{align*}
    implying
    \begin{align*}
        Z(t) = \nu t + M(t) &> \nu t + {(\zplus - \zmoins)} t + |\nu| t + \frac 12 \zmoins 
        \geq \zplus t + \left(\frac 12 - t \right) \zmoins,
    \end{align*}
    for $t \in [0,1]$. Thus,  $\taumoins > 1$ and $\tauplus \leq 1$ on $D$, which completes the proof.
\end{proof}

\subsection{Exit times estimates for two-dimensional diffusions near an edge}
\label{app:martingales-further}
In this section, we extend the results of Section~\ref{sec:1d-martingales} to 2-dimensional diffusions of the form 
\begin{subequations}
\label{eq:alt-SDE}  
\begin{align}
    \label{eq:y1-ito}
    \dd Y_1(t) &= \ell_1(Y_1(t),X_2(t)) \dd t + s_{11}(Y_1(t),X_2(t)) \dd W_1(t) + s_{21}(Y_1(t),X_2(t)) \dd W_2(t) \\
    \dd X_2(t) &= b_2(Y_1(t),X_2(t)) \dd t + \sigma_{12}(Y_1(t),X_2(t)) \circ \dd W_1(t) + \sigma_{22}(Y_1(t),X_2(t)) \circ \dd W_2(t)
\end{align}
\end{subequations}
with bounded Lipschitz coefficients $\ell_1, b_2,  s_{11}, s_{21}, \sigma_{12}, \sigma_{22}: \rr\times[0,1] \to \rr$. We assume the following relative ellipticity property:
 $s_{11}(y,x)^2+s_{21}(y,x)^2>0$ for all $(y,x)\in\R\times[0,1]$.

The results of this section will be applied to the study of the joint process $(Y_1(t),X_2(t))=(y_1(X_1(t)),X_2(t))$ in the notation introduced in Section~\ref{sec:local-attr}. 
We stress that~\eqref{eq:y1-ito} is an It\^o equation.

For a starting point $(Y_1(0),X_2(0))=(y,x)$, in this section, we denote the resulting distribution on paths by~$\pp^{(y,x)}$. The associated expectation is denoted by~$\ee^{(y,x)}$. 
Throughout this section, we work with a fixed segment 
$I=(-b,b)\subset \rr$ containing $0$. 
For $\ymoins$ and $\yplus$ satisfying $\ymoins < -b < b < \yplus$, and $r\in(0,1)$, we introduce
\begin{align*}
    \tau &= \tau(\ymoins,\yplus) = \inf\{t > 0 : Y_1(t) \notin (\ymoins,\yplus) \},
    \\
    \eta&=\eta(r)=\inf\{t\ge 0:\ X_2(t)\ge r\},
    \\
    \zeta&=\zeta(\ymoins,\yplus,r)=\tau(\ymoins,\yplus)\wedge \eta(r),
\end{align*}
the latter being 
the exit time of the process $(Y_1,X_2)$ from the rectangle $(\ymoins,\yplus)\times[0,r)$.
The main goal of this section is to obtain estimates on these stopping times.
First, we remind the following well-known property:
\begin{lemma}
\label{lem:finite-2d-mean} 
   For every $t>0$, we have $$\inf_{(y,x) \in I\times[0,1]}\pp^{(y,x)}\{Y_1(t) > \yplus\}>0.$$  
   Moreover, there exists $\alpha > 0$ such that 
    \[  
        \sup_{(y,x) \in I\times[0,1]} \ee^{(y,x)} \Exp{\alpha \tau} < \infty.
    \]
\end{lemma}

\begin{proof} 
    Our assumptions on the coefficients allow to write $Y_1(t)-y\ge Z(t)$ where $Z$ is a process of the form \eqref{eq:drift-plus-mart} satisfying the assumptions of Section~\ref{sec:1d-martingales}. Thus the first claim follows from Lemma~\ref{lem:positive-prob-to-escape-right} (and a time rescaling) and implies
    \begin{align*}
        \varkappa  = \inf_{(y,x) \in [\ymoins,\yplus]\times[0,1]} \pp^{(y,x)} \{\tau > 1\} < 1.
    \end{align*}
    The strong Markov property gives
    $
        \inf_{(y,x) \in [\ymoins,\yplus]\times[0,1]} \pp^{(y,x)} \{\tau > n\} \leq \varkappa^n
    $,
    and our second claim follows since $I \subset [\ymoins,\yplus]$.
\end{proof}

\begin{lemma}
\label{lem:exp-lb-suppt-case}
    If there exist $\nu > 0$ and $r\in(0,1)$ such that
    \begin{alignat*}{2}
        \ell_1(y,x) &\geq \nu,\quad  &(y,x)\in(-\infty,-b)\times[0,r] \\
        \ell_1(y,x) &\leq -\nu, \quad &(y,x)\in(b,\infty)\times[0,r],
    \end{alignat*}
    then there exist constants~$C,L>0$ 
    such that
    \begin{equation}
        \sup_{\substack{(y,x) \in I\times[0,r] }} \pp^{(y,x)}\Big\{\tau < C^{-1} \exp(C^{-1}(\yplus \wedge (-\ymoins)))\,\wedge\, \eta(r)\Big\} \leq \frac 14
    \end{equation}
    for all choices of $\ymoins\le -b-L$ and $\yplus\ge  b+L$ in the definition of~$\tau$.
\end{lemma}
    
\begin{proof} 
    It suffices to prove this estimate for $(y,x)\in\{-b,b\}\times[0,r]$.
     We define $\tau^{(0)}  = 0$ and, for $n\in\nn$, 
    \begin{align*}
        \sigma_{n} & = \inf\{ t > \tau_{n-1} : Y_1(t) \in \{-b-1,b+1\} \}\wedge \eta, \\
        \tau_{n} & = \inf\{t > \sigma_{n} : Y_1(t) \in \{\ymoins ,-b,b,\yplus \}\}\wedge \eta,
        \\
        \Delta_n&=\tau_n-\sigma_n.
    \end{align*}
    Introducing $F_n  = \{ Y_1(\tau_{n}) \in \{-b,b\} \}$
    for $n=0,1,2,\ldots,$ and $N  = \inf\{n : Y_1(\tau_{n}) \in \{\ymoins ,\yplus\}\}$, 
    we obtain $\tau = \tau_{N}$ and $\{N \geq n\} = \bigcap_{k=0}^{n-1} F_k$.
    
    By Lemma~\ref{lem:drift-wins-over-mart-bound}, there exist $C,K>0$ such that for all $\ymoins\le -b-L$,  $\yplus\ge  b+L$ and  $n\in\nn$,
    \begin{align*}
        \one_{F_{n-1}}\pp\big[\,F_n^\complement\cap\{\tau_n<\eta\} \,|\, \cF_{\tau_{n-1}}\big] \leq r(C,\yplus,\ymoins)\wedge \tfrac{1}{8}
    \end{align*} 
    almost surely, 
    where
    \[
    r=r(C,\yplus,\ymoins)=C \exp(-C^{-1}(\yplus \wedge (-\ymoins))).
    \]
    Letting $N_*$ be the largest even number not exceeding  $\tfrac {r^{-1}}8$, we use the union bound to obtain 
    \begin{align}
        \label{eq:N-N-star}
        \pp\{N \le  N_* \vee 1;\ \tau<\eta\} \leq \tfrac 18.
    \end{align}
    We then use Lemma~\ref{lem:lb-tauplus-prob} to choose $\delta>0$ such that with probability 1, for all $n$,
    \begin{align}
        \label{eq:choice-of-delta}
        \one_{F_{n-1}}\pp[\Delta_n < \delta;\ \tau_n<\eta|\, \cF_{\tau_{n-1}}] < \frac{1}{16\Exp{}}.
    \end{align}
    Denoting $N^*=N_*\vee 2$, we note that for some $c>0$
    \begin{align*}
        \Big\{N \ge N^*;\ \# \{n \leq N^* : \Delta_n \ge \delta\} \ge \tfrac {N^*}{2}  \Big\}\subset
        \big\{\tau\ge \tfrac {N^*}2\delta   \big\}\subset\big\{\tau\ge cr^{-1}\delta\big\}.  
    \end{align*}   
    Thus,
    \begin{multline}
        \label{eq:low-tau-estimate}
        \pp\left\{\tau < (cr^{-1}\delta)  \wedge \eta\right\}
        \leq \pp\{N \le  N_*\vee 1;\ \tau<\eta\}
            \\+ \pp\Big\{N \geq N^*;\ \# \{n \leq N^* : \Delta_n < \delta \} \ge \tfrac {N^*}2;\ \tau<\eta  \Big\}, 
    \end{multline}
    where the first term is estimated in~\eqref{eq:N-N-star}. Using the strong Markov property, 
    the inequality $\binom{n}{k}\le (n\Exp{}/k)^k$, the 
    requirement~\eqref{eq:choice-of-delta}, and the fact that $N^*\ge 2$,  we can bound the second one by  
    \begin{multline*}
        \sum_{\substack{K\subset \{1,\ldots,N^*\}\\ |K|=N^*/2}}\ee\bigg[\one_{\{N \geq N^*\}} 
        \prod_{n\in K} \one_{\{\Delta_{n} < \delta;\ \tau_n<\eta\}} \bigg]
        \\
        \le \sum_{1\le n_1<\ldots< n_{N^*/2}\le N_*} \prod_{k=1}^{ N^*/2}
        \Big\|
        \one_{F_{n_k-1}} \ee\big[\one_{\{\Delta_{n_k} < \delta;\ \tau_n<\eta\}}|\cF_{\tau_{n_k-1}}
        \big]\Big\|_\infty
        \\
        \leq \binom{N^*}{ \frac 12 N_*} \Big(\frac{1}{16\Exp{}}\Big)^{\frac{N^*}{2}} 
            \,\leq\, 2^{\frac{N^*}{2}} \Exp{\frac{N^*}{2}}  \Big(\frac{1}{16\Exp{}}\Big)^{\frac{N^*}{2}} 
            \,\leq\,  \Big(\frac 18\Big)^{\frac{N^*}{2}} 
        \,<\, \frac 18,
     \end{multline*}
    and the desired estimate follows with an appropriately adjusted constant.
\end{proof}

\begin{lemma}
\label{lem:lin-ub-different-extremes}
    If there are $\nu > 0$ and $r\in(0,1)$ such that
    \begin{alignat*}{2}
        \ell_1(y,x) &\leq -\nu, &\quad   (y,x)&\in (-\infty,-b)\times[0,r], \\
        \ell_1(y,x) &\geq \nu,  &   (y,x)&\in (b,\infty)\times[0,r],
    \end{alignat*}
    then there is a constant~$C$ such that  for all  $\ymoins<-b$ and $\yplus>b$,
    \begin{equation}
        \sup_{\substack{(y,x) \in I\times[0,r] }} \ee^{(y,x)}\, \zeta(\ymoins,\yplus,r) \leq C (\yplus  \wedge -\ymoins).
    \end{equation}
  
\end{lemma}

\begin{proof}
    It suffices to prove our claim for $\ymoins< a-1$ and $\yplus> b+1$. For those $\yplus$ and $\ymoins$, we can
    define $\tau_0  = 0$ and, for $n\in\nn$, 
    \begin{align*}
        \sigma_{n} & = \inf\{ t > \tau_{n-1} : Y_1(t) \in \{-b-1,b+1\} \} \wedge \eta, \\
        \tau_{n} & = \inf\{t > \sigma_{n+1} : Y_1(t) \in \{\ymoins ,-b,b,\yplus \}\}\wedge \eta,\\
        \Delta_n&=\sigma_n-\tau_{n-1},\\
        \Delta'_n&=\tau_n-\sigma_n.
    \end{align*}
    Introducing $F_n  = \{ Y_1(\tau_n) \in \{-b,b\}\}$ 
    and $N  = \inf\{n : Y_1(\tau_{n}) \notin \{-b,b\}\}$, we almost surely have $\zeta = \tau_{N}$
    and $\{N \geq n\} = \bigcap_{k=1}^{n-1} F_k$. 

    By Lemma~\ref{lem:finite-2d-mean}
    and Corollary~\ref{lem:ee-ap-bound}, there are constants $C',C''>0$ 
    such that for all $\yplus$, $\ymoins$, and~$n$, we almost surely have
    \begin{align*}
        \ee^{(y,x)}[\one_{\{N \geq n\}}\Delta_n\,|\,\cF_{\tau_{n-1}}]&\le C',\\
        \ee^{(y,x)}[\Delta'_n\, |\, \cF_{\sigma_{n}}] &\leq C''(\yplus  \wedge -\ymoins ).
    \end{align*}
    Also, by Lemma~\ref{lem:drift-wins-over-mart-bound}, there exists $p > 0$ independent of $\ymoins$ and $\yplus$ such that 
    \[
        \ee^{(y,x)}[ \one_{N\geq n} \one_{F_n} |\cF_{\sigma_n}]
        \leq 1 - p.
    \]
    Therefore,
    \begin{align*}
        \ee^{(y,x)}\, \zeta
            &= \sum_{n=1}^\infty \ee^{(y,x)} [\one_{\{N=n\}} \tau_{n}] \\
            &\leq \sum_{n=1}^\infty \ee^{(y,x)} \left[\one_{\bigcap_{k=1}^{N-1} F_k} \sum_{k=1}^n(\Delta_k+\Delta'_k)\right] \\ 
            &\leq \sum_{n=1}^\infty \sum_{k=1}^{n} \ee^{(y,x)}[\one_{F_1} \dotsb \one_{F_{k-1}}(\Delta_k+\Delta'_k) \one_{F_{k+1}} \dotsb \one_{F_{n-1}}] \\
            &\leq \sum_{n=1}^\infty n (1-p)^{(n-2) \wedge 0} \left(C' + C''(\yplus  \wedge -\ymoins ) \right),
    \end{align*}
    and our claim follows.
\end{proof}

Finally, the same ingredients can be used to derive the following lemma, whose proof we omit.

\begin{lemma}
\label{lem:lin-ub-same-extremes}
    If there is $r\in(0,1)$ and $\nu>0$ such that $\ell_1(y,x) \geq \nu$ for all $y \notin I$ and $x\in[0,r]$,
    then there is a constant~$C$ such that for all 
    for all  $\ymoins<-b$ and $\yplus>b$,
    \begin{equation}
        \sup_{\substack{(y,x) \in I\times[0,r] }} \ee^{(y,x)}\, \zeta(\ymoins,\yplus,r)\leq C \yplus. 
    \end{equation}
\end{lemma}

\section{Properties of the corrector defined in Section~\ref{sec:coord-prelim}}
\label{app:fk-function}

Throughout this section, we restrict ourselves to $\Ed$. Since $x_2=0$ for all $x\in \Ed$, we will use $X(t)$ as a shorthand for both $X_1(t)$ and $(X_1(t),0)$, use $x$ for both $x_1$ and $(x_1,0)$, and so on.
We suppose that there exists a (unique) invariant probability measure~$\invm^\Ed$ concentrated on $\Ed=(0,1)\times \{0\}$ as in Theorem~\ref{thm:1d}.2. In other words, we assume that both endpoints of $\Ed$ are sources for the dynamics restricted to $\Ed$. 
This also means that the process $Y(t)=y(X(t))$ on $\rr\times\{0\}$ is a solution of
the one-dimensional restriction
\begin{align*}
    \dd Y(t) &= \ell(Y(t)) \dd t + s_{1}(Y(t)) \dd W_{1}(t)+s_{2}(Y(t)) \dd W_{2}(t),
\end{align*}
of the It\^o SDE~\eqref{eq:alt-SDE-first-h},  
where $\ell,s_1,s_2\in C^2(\rr)$,  $s_1(y)^2+s_2(y)^2>0$ for all $y\in\rr$ and 
there are $\lambda^0,\lambda^1>0$ such that
\begin{subequations}
\label{eq:coefs-near-endpoints}
\begin{align}
    \lim_{y\to-\infty} \ell(y)&= \lambda^0,\quad
    \lim_{y\to\infty} \ell (y)= -\lambda^1, 
\end{align}  
\end{subequations}

Due to Remark~\ref{rem:mixing},  for every bounded $g\in C(\Ed)$ satisfying $\int g\dd \invm^{\Ed}=0$, the integral in the following definition converges absolutely:
\begin{equation}
    \label{eq:def_of_psi}
    \psi(x) = \int_0^\infty \ee g(X^x(t))\dd t,\quad x\in \Ed.
\end{equation}
We are interested in $g$ given by \eqref{eq:g_is_centered_Lambda}.
This function $g$ and its first derivative are bounded. 
This will be our only standing assumption on~$g$ in this section.
We denote by $\cL^\Ed$ the (local) generator  of the restriction of~\eqref{eq:main-eq-h} to~$\Ed$: for $\varphi\in C^2(\Ed)$,
\begin{align}
\label{eq:def-cL-as-diff-Ed}
    \cL^{\Ed} \varphi(x_1)
        &= b_1(x_1,0) \varphi'(x_1) 
        + \frac 12 \sum_{m} \partial_{1}\sigma_{m 1}(x_1,0) \sigma_{m 1}(x_1,0) \varphi'(x_1) + \frac 12 \sum_{m} \sigma_{m 1}(x_1,0)^2 \varphi''(x_1).
\end{align}

\begin{lemma}
\label{lem:psi-properties-i}
    The function $\psi$ belongs to $C^2(\Ed)$ and 
    satisfies the Poisson equation
    \begin{equation}
    \label{eq:ode-psi}
        \cL^{\Ed} \psi = {-g}.
    \end{equation}
\end{lemma}

\begin{proof}  
    The idea of the proof is to use the classical relation between diffusions and second-order differential operators and interpret
    \begin{align*}
        \ee \psi(X^{x}(t)) + \int_{0}^t \ee g(X^{x}(r)) \dd r
        &= \ee \left[\int_0^\infty \ee g(X^{(X^{x}(t),0)}(r)) \dd r\right] + \int_{0}^t \ee g(X^{x}(r)) \dd r2r
        \\ &= \int_t^\infty \ee[g(X^{x}(r))]\dd r+ \int_{0}^t \ee g(X^{x}(r))  \dd r 
        =\psi(x),
    \end{align*}
    as the (stationary) Feynman--Kac formula for~\eqref{eq:ode-psi}. We just need to make this interpretation precise but, since $\Ed$ is noncompact and $\sigma_1$ is not bounded away from $0$ on $\Ed$, we cannot apply existing results like Theorem~1 of \cite{Pardoux-Veretennikov:MR1872736} directly. It suffices though to prove that for an arbitrary interval $[a,b] \subset \Ed$, the function
    $\psi$ is a (unique) classical solution of the Poisson problem
    \begin{align*}
    \cL^\Ed\phi(x) &=-g(x),\quad x\in (a,b),\\
    \phi(a)&=\psi(a),\\   
    \phi(b)&=\psi(b).
    \end{align*}
    According to Proposition~7.2  and the discussion after Remark~7.5 in~\cite[Section~{5.7.A}]{KaSh}, it is sufficient to 
    introduce $\tau=\inf\{t\ge 0:\ X(t)\in\{a,b\}\}$ and 
    check that for all 
    $x\in(a,b)$
    \begin{equation}
    \label{eq:expected-exit-finite} 
    \ee^x\tau<\infty
    \end{equation}    
    and 
    \begin{equation}
    \label{stoch-representation}
    \psi(x)=\ee\left[\psi(a)\one_{\{X^x(\tau)=a\}}\right]+\ee\left[\psi(b)\one_{\{X^x(\tau)=b\}}\right]+\ee\int_0^\tau g(X^x(t))dt.
    \end{equation}  

    Since $\inf_{x\in [a,b]}\sigma_1(x)>0$, relation~\eqref{eq:expected-exit-finite} directly follows from Lemma~7.4 in~\cite[Section~{5.7.A}]{KaSh}. Relation~\eqref{stoch-representation} follows from the strong Markov property once we rewrite its right-hand side as 
    \begin{equation*}
    \ee\bigg[\one_{\{X^x(\tau)=a\}}  \int_\tau^{\infty} g(X^x(r))dr   +\one_{\{X^x(\tau)=b\}} \int_\tau^{\infty} g(X^x(r))dr   +\int_0^\tau g(X^x(t))dt\bigg],
    \end{equation*}    
    and the proof is completed.
\end{proof}    

Let us study the behavior of~$\psi$ near the endpoints of~$\Ed$.
The following result concerns the right endpoint $\Or^1=(1,0)$. A similar result holds near the left endpoint $\Or^0=(0,0)$. 
We  will need the map $x(\cdot)$ 
defined on~$\rr$ as the inverse of the smooth map $y(\cdot)$ introduced in Section~\ref{sec:coord-prelim}.

\begin{lemma}
\label{lem:psi-as-time}
    There is $y_0$ such that  if $y\ge y_0$ and $h>0$ then
    \begin{align}
        \label{eq:increment_of_psi}
        \psi(x(y+h)) - \psi(x(y)) 
        =
        \ee \int_{0}^{\tau} g(X^{x(y+h)}(t)) \dd t,
    \end{align}
    where $\tau  = \inf\{t > 0 : X^{x(y+h)}(t) \leq x(y) \}$. The function $y\mapsto \partial_y \psi(x(y))$ is bounded on $[y_0,+\infty)$.
\end{lemma}

\begin{proof} 
    Using~\eqref{eq:coefs-near-endpoints}, we choose $y_0$ such that $\ell(y)< -\tfrac 12 \lambda^1 < 0$ for all $y>y_0$. This allows to apply results from Appendix~\ref{app:martingale} to the process~$Y(t)$ on $[y_0,+\infty)$.   Our goal is to compute $I$, the right-hand side of  \eqref{eq:increment_of_psi}.  Let $D=\sup|g(x)|$.  Since~$D<\infty$ and $\ee\tau<\infty$ due to 
    Corollary~\ref{cor:finite-mean},  $I$  is well-defined and finite. Moreover, the dominated convergence theorem implies that for all  $T_0>0$,
    \begin{equation}
        \label{eq:limit_Busemann}
        I=\lim_{T\to\infty}  I(T,T_0),
    \end{equation}
    where
    \[
        I(T,T_0)= \ee\bigg[\int_0^\tau g(X^{x(y+h)}(t)) \dd t \, \one_{\{\tau<T-T_0\}} \bigg].
    \]
    Let us fix arbitrary $\eps>0$ and choose $T_0$ such that 
    \begin{align}\label{eq:remainders}
        \int_{T_0}^\infty|\ee g(X^{x(y+h)}(t))|\dd t,\ 
        \int_{T_0}^\infty|\ee g(X^{x(y)}(t))|\dd t&<\eps. 
    \end{align}
    We can write
    \begin{align*}
        I(T,T_0)
        =& \ee \bigg[\int_0^T g(X^{x(y+h)}(t))\dd t \, \one_{\{\tau<T-T_0\}}\bigg]  - \ee\bigg[ \int_\tau^T  g(X^{x(y+h)}(t)) \dd t  \, \one_{\{\tau<T-T_0\}}\bigg]
        \\=
        & \int_0^T \ee g(X^{x(y+h)}(t))\dd t - \ee \bigg[\int_0^T g(X^{x(y+h)}(t))\dd t\, \one_{\{\tau>T-T_0\}}\bigg]\\
        &- \ee\bigg[ \int_\tau^{\tau+T_0}  g(X^{x(y+h)}(t)) \dd t  \, \one_{\{\tau<T-T_0\}}\bigg]
        - \ee\bigg[ \int_{\tau+T_0}^T  g(X^{x(y+h)}(t)) \dd t  \, \one_{\{\tau<T-T_0\}}\bigg]
        \\
        =& I_1(T)-I_2(T,T_0)-I_3(T,T_0)-I_4(T,T_0).
    \end{align*}
    If $T>T_0$, we have $|I_1-\psi(x(y+h))|<\eps$,  
    $|I_3-\psi(x(y))|<\eps$, $|I_4|<\eps$ due to the strong Markov property and~\eqref{eq:remainders}. Also, 
    Lemma~\ref{lem:ub-tauplus-prob} implies that 
    for sufficiently large $T$,
    \[
    |I_2|<TD\pp\{\tau>T-T_0\}<\eps. 
    \]
    Combining these estimates with \eqref{eq:limit_Busemann} and letting $\eps\to0$, we obtain \eqref{eq:increment_of_psi}. We can use it along with the bound on $\ee \tau$ given in Lemma~\ref{lem:ee-ap-bound} and boundedness of $g$ to derive the Lipschitz property of~$\psi(x(\cdot))$. Boundedness of $\partial_y\psi(x(y))$ follows.
\end{proof}    

\begin{lemma}
\label{lem:psi-properties-ii}
    The functions $x \mapsto x(1-x)\psi'(x)$ and $x \mapsto [x(1-x)]^2\psi''(x)$ are bounded.
\end{lemma}

\begin{proof}
    It suffices to prove that $\partial_y\psi(x(y))$ and $\partial^2_y\psi(x(y))$ are bounded. 
    We focus on estimates as $y\to \infty$; those near $-\infty$ are similar. Changing variables and using 
    Lemmas~\ref{lem:psi-properties-i} and~\ref{lem:psi-as-time}, we obtain that $\psi(x(\cdot))$ is a $C^2$ and globally Lipschitz solution of 
    \begin{equation}
        \label{eq:backwardK-for-trans-psi}
        \frac{1}{2}s(y)^2\partial_{y}^2\psi(x(y)) + \ell(y)\partial_y\psi(x(y))
        = -g(x(y)).
    \end{equation}
    Boundedness of $\partial_y\psi(x(y))$ follows from the Lipschitz property.  If $s(y)$ is assumed bounded away from~$0$ as $y\to\infty$, then equation~\eqref{eq:backwardK-for-trans-psi} immediately implies that $\partial_{y}^2\psi(x(y))$ is bounded (since $\ell(y)$ and $g(y)$ are bounded).

    In the general case,  we instead introduce
    $
        \phi(y) = \ell(y) \partial_y \psi(x(y)) + g(y).
    $
    It is straightforward to show using \eqref{eq:backwardK-for-trans-psi} and Lemma~\ref{lem:psi-as-time}, that $\phi$ is a bounded solution of an equation
    of the form $\partial_y \phi(y) + \beta(y) \phi(y) = f(y)$, where $f$ is bounded and $\beta$ is negative and bounded away from 0 as $y\to\infty$. One can use these properties and variation of constants to show that a bounded solution is unique and allows for a concise expression whose derivative is easy to bound.
\end{proof}

\begin{proof}[Proof of Lemma~\ref{lem:psi}]
    We derive the estimates for $x_1 \leq \tfrac 12$ only; those for $x_1 \geq \tfrac 12$ are obtained in the same way, replacing estimates as $x_1 \to 0$ with their counterparts as $x_1 \to 1$. Inspecting the structure in~\eqref{eq:def-cL-as-diff}, we see that there are two kinds of contributions to $(\cL\varphi)(x_1,x_2)$: those that involve $x_{i'}$-derivatives (as part of the Stratonovich-to-It\^o correction), and those that do not.

    But since the function at hand depends on~$x_1$ only (namely $\varphi(x_1,x_2) = \psi(x_1)$), the only terms of the first kind are the following:
    \[  
        \frac 12 \sum_{m} \partial_2 \sigma_{m 1}(x_1,x_2) \sigma_{m 2}(x_1,x_2) \psi'(x_1).
    \]
    By Lemma~\ref{lem:psi-properties-ii}, we have $\psi'(x_1) = O(\tfrac{1}{x_1})$ as $x_1 \to 0$. By Assumptions~\ref{req:smoothness-2d} and~\ref{req:invar-2d}, we have $\partial_2\sigma_{m 1}(x_1,x_2) = O(x_1)$ and $\sigma_{m 2}(x_1,x_2) = O(x_2)$ as $x_1 \to 0$ and $x_2 \to 0$. Hence, these terms are indeed $O(x_2)$ as $x_2 \to 0$, uniformly in~$x_1\leq \tfrac 12$. 

    As for terms of the second kind, we claim that they amount to an $O(x_2)$-perturbation of $(\cL^{\Ed}\psi)(x_1)$ as $x_2 \to 0$. To see this, we use Assumptions~\ref{req:smoothness-2d} and~\ref{req:invar-2d} to expand
    \begin{align*}
        & b_1(x_1,x_2)\psi'(x_1) + \frac 12 \sum_{m} {\partial_1 \sigma_{m 1}(x_1,x_2) \sigma_{m 1}(x_1,x_2)} \psi'(x_1)
             + \frac 12 \sum_{m} {\sigma_{m 1}(x_1,x_2)^2} \psi''(x_1) 
        \\
        &\qquad = 
         [b_1(x_1,0) + O(x_2)O(x_1)]\psi'(x_1) 
         + \frac 12 \sum_{m} {[\partial_1 \sigma_{m 1}(x_1,0) + O(x_2)] [\sigma_{m 1}(x_1,0) + O(x_2)O(x_1)]} \psi'(x_1) \\
             &\qquad\qquad {} + \frac 12 \sum_{m} {[\sigma_{m 1}(x_1,0) + O(x_2)O(x_1)]^2} \psi''(x_1).
     \end{align*}
    By Lemma~\ref{lem:psi-properties-ii}, we have $\psi'(x_1) = O(\tfrac{1}{x_1})$ and $\psi''(x_1) = O(\tfrac{1}{x_1^2})$ as $x_1 \to 0$. By Assumptions~\ref{req:smoothness-2d} and~\ref{req:invar-2d}, we have $\sigma_{m 1}(x_1,0) = O(x_1)$  as $x_1 \to 0$. Hence, these terms are indeed $O(x_2)$-perturbations of $(\cL^{\Ed}\psi)(x_1)$
    as $x_2 \to 0$, uniformly in $x_1 \leq \tfrac 12$; see~\eqref{eq:def-cL-as-diff-Ed}.
\end{proof}

\providecommand{\bysame}{\leavevmode\hbox to3em{\hrulefill}\thinspace}

\end{document}